%% file: arxiv_submit.tex
\documentclass[11pt,oneside]{article}

\include{head_val}

\title{Learning the smoothness of noisy curves with application to online curve estimation}

\author{%
Steven Golovkine\thanks{Groupe Renault \& CREST - UMR 9194, Rennes, France, \href{mailto:steven.golovkine@ensai.fr}{steven.golovkine@ensai.fr}}
\and
Nicolas Klutchnikoff\thanks{Univ Rennes, CNRS, IRMAR - UMR 6625, F-35000 Rennes, France, \href{mailto:nicolas.klutchnikoff@univ-rennes2.fr}{nicolas.klutchnikoff@univ-rennes2.fr}}
\and
Valentin Patilea\thanks{Ensai, CREST - UMR 9194, Rennes, France, \href{mailto:valentin.patilea@ensai.fr}{valentin.patilea@ensai.fr}}
}
\date{\today}

\begin{document}

\maketitle

\begin{abstract}
Combining information both within and across trajectories, we propose a
simple estimator for the local regularity of the trajectories of a stochastic process. Independent trajectories are measured with errors at randomly sampled time points. The proposed approach is model-free and applies to a large class of stochastic processes. Non-asymptotic bounds for the concentration of the estimator are derived. Given the estimate of the local regularity, we build a nearly optimal local polynomial smoother from the curves from a new, possibly very large sample of noisy trajectories. We derive non-asymptotic pointwise risk bounds uniformly over the new set of curves. Our estimates perform well in simulations, in both cases of   differentiable or non-differentiable trajectories. Real data sets illustrate the effectiveness of the new approaches.
\end{abstract}

\tableofcontents

\input{main/introduction}

\input{main/local_regularity}

\input{main/adaptive_smoothing}

\input{main/empirical_analysis}

\begin{appendix} 

\input{appendix/theorem_Ht0}

\input{appendix/proof_Lp}

\input{appendix/proof_sup}

\input{supplement/presmoothing_regularity}

\input{supplement/technical_lemma}

\input{supplement/moment_bound}

\input{supplement/additional_simulation_results}

\end{appendix}

\section*{Acknowledgements}
The authors  thank Groupe Renault and the ANRT (French National Association for Research and Technology) for their financial support via the
CIFRE convention no.~2017/1116. Valentin Patilea gratefully acknowledges support from the Joint Research Initiative “Models and mathematical processing of very large data” under the aegis of Risk Foundation, in partnership with MEDIAMETRIE and GENES, France. We thank the Associate Editor and an anonymous
reviewer for their careful reading and constructive comments, which helped us to improve the manuscript.

\bibliographystyle{abbrv}
\bibliography{ref_short_rev}

\end{document}

%% file: main/introduction.tex
\section{Introduction}\label{sec:introduction}

More and more phenomena in modern society produce observation entities in the form of a sequence of measurements recorded intermittently at several discrete points in time. Very often the measurements are noisy and the observation points in time are neither regularly distributed nor the same across the entities. Functional data analysis (FDA) considers such data as being values on the trajectories of a stochastic process, recorded with some error,  at discrete random times. One of the main purposes of the FDA is to recover the trajectories, also called curves or functions,  at any point in time. See, \emph{e.g.}, \cite{ramsay_functional_2005,horvath_inference_2012,wang_review_2015,zhang_sparse_2016} for some recent references. Whatever the approach for recovering the curve is, in the existing literature it is usually assumed that, for each curve, a certain number of derivatives exist. However, many applications, some of them presented in the following, indicate that assuming that the curves admit second, third,... order derivatives is not realistic. Assuming that the curves to be reconstructed are smoother than they really are could lead to missing important information carried by the data. In this contribution, we propose a definition of the local regularity of the curves which could be easily estimated from the data and used to estimate the curves.  

To formalize the framework, let $I\subset \RR$ be a compact interval of time. We consider $N$ functions $\Xp{1},\dotsc,\Xp{n},\dotsc,\X{N}$ generated as a random sample of a stochastic  process $X = (X_t : t\in I)$ with continuous trajectories. 
For each $1\leq n \leq N$, and given a positive integer $M_n$, let $\Tnm$, $1\leq m \leq  M_n$, be the random observation times for the curve $\Xp{n}$. These times are obtained as independent copies of a variable $T$ taking values in $I$. The integers $M_1, \dotsc,M_N$ represent an independent sample of an integer-valued random  variable $M$ 
with expectation $\mu$. 
Thus $M_1, \dotsc,M_N$ is the $N$th line in a triangular array of integer numbers. 
We assume that the realizations of $X$, $M$ and $T$ are mutually independent.   The observations associated with a curve, or trajectory, $X^{(n)}$ consist of  the pairs  $(\Ynm , \Tnm ) \in\mathbb R \times I $ where $\Ynm $ is defined as
\begin{equation}\label{model_eq}
    \Ynm = \Xn({\Tnm}) + \enm,
    \qquad 1\leq n \leq N,  \; \; 1\leq m \leq M_n,
\end{equation}
and  $\enm$  are independent copies of a centered error variable $\varepsilon$. For the sake of readability, here and in the following, we use the notation $X_t$ for the value at $t$ of the generic process $X$ and $X^{(n)}(t)$ for the value at $t$ of the realization  $X^{(n)}$ of $X$. The $N-$sample of $X$ is composed of two sub-populations: a \emph{learning} set of $\N0$ curves and a set of $\N1$ curves to be recovered that we call the \emph{online} set. Thus,  $1\leq \N0,\N1 <N$ and $\N0+\N1=N$. 
Let $\Xp{1},\dotsc,\Xp{\N0}$ denote the curves corresponding to the \emph{learning} set.

Our first aim is to define a meaningful, model-free concept of local regularity for the process $X$ and to build an estimator for it.
The estimator could be computed easily and rapidly from the observations $(\Ynm , \Tnm )$ corresponding to the curves in the \emph{learning} set, and does not require a very large number $\N0$ of curves. Moreover, it could be easily updated if more curves are added to the \emph{learning} set. The problem of estimating the regularity of $X$ is related to the estimation of the Hausdorff, or fractal, dimension of time series. See, for instance, \cite{const_hall94,chan_wood_04,gneiting12} and the references therein. However, herein, we adopt the FDA point of view and use the so-called \emph{replication} and \emph{regularization} features of functional data (see \cite{ramsay_functional_2005}, ch.22). More precisely, we combine information both across and within curves. Thus, taking strength from the information contained in the whole set of $\N0$ available time series, we are able to investigate more general situations: $X$ need not to be a Gaussian, or a transformed Gaussian process, it is not necessarily  stationary or with stationary increments, it could have a fractal dimension which changes over time, it is observed with possibly heteroscedastic noise, at random moments in time. 

The local regularity we study determines the regularity of the sample paths of $X$. Sample paths regularity determines, for instance,  the minimax optimal rate for the nonparametric estimators of the mean and covariance functions.  In particular, knowing the local regularity serves to distinguish the so-called sparsely and dense sampled curve cases. See \cite{cai2011,zhang_sparse_2016}. For some widely used examples, the local regularity  is also related to the rate of decrease of the eigenvalues of the covariance operator, a property which is commonly used in FDA literature. In almost all existing contributions, the sample paths regularity and the rate of the eigenvalues are supposed given. We here propose a simple method to estimate them.

Based on the regularity estimates, our second objective is to build an adaptive, nearly optimal smoothing for a possibly very large set of $\N1$ new curves. While several smoothers could be used, we focus on local polynomials.  Optimal  curve reconstruction is an important step in FDA, for instance for computing the median curve  or the depth of a curve, to detect outliers. See, \emph{e.g.,} \cite{romo2009}. It can also serve to compute optimal mean and covariance functions estimator in the dense case. See, \emph{e.g.,}  \cite{cai2011,zhang_sparse_2016}.
Let 
$$
\Xc{1} = \Xp{\N0+ 1},\dotsc,\Xc{\N1} = \Xp{N},
$$ 
denote the curves from the \emph{online} set to be recovered from the corresponding observations $(\Ynm , \Tnm )$. 
This  issue  is a nonparametric estimation problem and, if each curve regularity is given, nonparametric estimators of the curves $X^{[1]},\dotsc,X^{[\N1]}$ could be easily built, for instance using the local linear smoother or the series estimator. Nevertheless in applications, there is no reason to suppose that the sample paths of the random process $X$ have a known regularity. 
When it is not reasonable to assume a given regularity for the trajectories, one could use one of the existing data-driven procedures for determining the optimal smoothing parameter. However, the  existing procedures, such as the cross-validation or the Goldenshluger-Lepski method \cite{GL2011}, were designed for the case where one observes only one curve. Thus one has to apply them for each curve separately, which could require large amounts of resources. 

In Section \ref{sec:local-regularity}, we define the local regularity and provide concentration bounds for the estimator of the local regularity of the trajectories of $X$. Our results are new and of non-asymptotic type, in the sense that they hold for any values of the sample sizes $\N0$ and the mean value of observation times $\mu$, provided these values are sufficiently large. In Section \ref{sec:ad_opt}, we explain the relationship between the probabilistic concept of local regularity for the trajectory of $X$ and the analytic regularity of the curves which usually determines the optimal risk rate in nonparametric estimation. We also provide insight into the relationship between the local regularity and the rate of decrease for the eigenvalues of the covariance operator. 
Given the estimate of the local regularity of the trajectories of $X$,  in Section \ref{sec:ad_opt} we build adaptive local polynomial estimators and provide a non-asymptotic bound for the pointwise risk of the local polynomial smoother, uniformly over the \emph{online} set. This uniform bound is obtained using an exponential-type moment bound for the pointwise risk for the local polynomial smoother, a new result of interest in itself. The pointwise risk bound is optimal, in the nonparametric regression estimation sense, up to some logarithmic factors induced by our stochastic curves model,
the concentration of the local regularity estimator, and the uniformity over the \emph{online} set. 
Assuming that $X$ has a constant regularity over the interval $I$, we also derive a non-asymptotic bound for the risk of the local polynomial smoother uniformly over $I$, and uniformly over the \emph{online} set.   
In Section \ref{sec:empi}, we provide some additional guidance for the implementation  of the local polynomial smoother and report results from simulation showing that our  estimator of the local regularity and the adaptive local polynomial estimator  perform well in both cases, whether or not the trajectories are differentiable. As a further application of our local regularity estimation approach, we consider the median curve estimation problem for samples of noisy  curves. Our median curve is obtained from the smoothed curves with the optimal bandwidth given by the local regularity estimate. We compare the accuracy of our median curve  with that obtained with the curves  smoothed by cross-validation. While the accuracy is comparable, the computation time is far shorter when using our approach, and this makes it suitable for embedded systems or for applications with online data. 
A real data application on vehicle traffic flow analysis illustrates the effectiveness of our approaches. The proofs of our results are postponed to the Appendix. Additional technical aspects, simulation results, and details on traffic flow application are also relegated to the Appendix. To further illustrate the irregularity of the curves in applications, we also report in the Appendix the local regularity estimates for another three functional data sets often analyzed in the literature.

%% file: main/local_regularity.tex
\section{Local regularity estimation}\label{sec:local-regularity}

The new local regularity estimator  is introduced and studied in this section. After providing some insight into the ideas behind the construction, we provide a concentration result for our estimator under general mild assumptions which do not impose a specific distribution for $X$. In particular, $X$ could, but need not, be a Gaussian process.  The case where the variance of the noise is not constant is also discussed. 

\subsection{The methodology}\label{sec:heur}

Let us present the main ideas behind the construction of the regularity estimate. For this, let us 
introduce some more notation used throughout the paper. Let $K_0$ be an integer value which will be defined below, and consider the order statistics of a $M$-sample $T_1,\dotsc, T_M$ distributed as $T$ which admits the density $f$. Let $\T \in I$ such that $f(\T) > 0$. We extract the subvector of the $K_0$ closest values to $\T$ and denote  these values $T_{(1)}\leq\dotsc\leq T_{(K_0)}$. If $\T = \inf(I)$ then $\T\leq T_{(1)}$, while if $\T = \sup(I)$, then $T_{(K_0)}\leq \T$. When $\T$ is an interior point of $I$, $\T$ likely lies between $T_{(1)}$ and $T_{(K_0)}$. 
Next, we define the interval
\begin{equation}\label{def_J}
    J_{\mu}(\T) = \big(\T-|I|/\log(\mu), \T+ |I|/\log(\mu)\big) \cap I,
\end{equation}
where $|I|$ denotes the length of the interval $I$ and, recall, $\mu$ is the expectation of $M$.
 In the following, we introduce our conditions using the interval $  J_{\mu}(\T) $, which depends on $\mu$. The theoretical results we derive are non-asymptotic, in particular they hold for any fixed $\mu$, provided it is sufficiently large. 
If one is interested by asymptotic results corresponding to the case where $\mu$ increases to infinity, then our $ J_{\mu}(\T) $ is eventually contained in any fixed neighborhood of $\T$. In this case, one can state all the assumptions on $X$ in a more standard way, using a fixed interval instead of our $ J_{\mu}(\T) $.

We assume that the process $X$ generating the continuous  curves $X^{(1)}\!,\dotsc,X^{(N)}$  satisfies 
\begin{equation}\label{key_condX}
    \EE\left[(X_{u}-X_{v})^{2}\right]
    \asympp L_{t_0}^{2}|v-u|^{2\HT} , \quad u,v\in J_\mu(\T),
\end{equation}
for some $\HT\in (0,1]$ and $L_{t_0}>0$ which could both change with $\T$.
 Here and in the following, $\asympp$ means the left-hand side is equal to the
right-hand side times a quantity which tends towards 1 when $|v-u|\rightarrow 0$. 
When the trajectories of $X$ are not differentiable, $\HT$ is what we call the \emph{local regularity of the process $X$ at $\T$}. For now, we focus on this case. When, with probability 1,  the trajectories of $X$ admits derivatives of order $\KT\geq 1$ in a neighborhood of $\T$, the property \eqref{key_condX} will be used for the derivative of order $\KT$ of the smooth trajectories. In this smooth case, the local regularity of the process $X$ at $\T$ will be $\KT+\HT$.  See the comment following Theorem~\ref{thm:Ht0}.

Some commonly used processes have the  eigenvalues  of the covariance operator such that, for some $\nu >1$, $\lambda_j \sim j^{-\nu}$, $j \geq 1$. Such processes  have a constant local regularity. 
Moreover,  $\KT+\HT= (\nu-1)/2$. 
As an example,  the stationary fractional Ornstein-Uhlenbeck process with index $\rho\in (0,2)$ has  the covariance function 
$$
\Gamma (s,t) = \exp(-a|s-t|^\rho), \text{ for some } a>0,
$$
which yields $\nu = 1+\rho$, $\HT\equiv  \rho/2$ and $\KT=0$. 
 Among the nonstationary processes satisfying our condition \eqref{key_condX}, we can mention the fractional Brownian motion with Hurst exponent $ H \in(0,1)$, for which $\HT\equiv H$ and $\lambda_j \sim  j^{-(1+2H)}$. Examples with $\KT >0$ could be obtained by $\KT-$times integration of the processes with $\KT=0$, such as for instance 
the so-called $\KT-$integrated Brownian motion. See, \emph{e.g.}, \cite{pages04} for more details on these examples.

To  construct our estimator of $\HT$, 
we consider  the event 
$$
\mathcal{B}=\{M\geq K_0, T_{(1)}\in J_\mu(\T),\dotsc,T_{(K_0)} \in J_\mu(\T)\},
$$ 
which is expected to be of high probability. Let  $\mathbf 1_{\mathcal{B}}$ denote the indicator of $\mathcal{B}$ and let us define the expectation operator 
$$
\EEB (\cdot ) = \EE (\cdot \mathbf 1 _{\mathcal B} ).
$$
Using~\eqref{key_condX} and the independence between $X$ and $T$,  for any $1\leq k < l \leq K_0$, 
\begin{align*}
    \EEB\left[(X_{T_{(l)}}-X_{T_{(k)}})^{2}\right]
    &\asympp L_{t_0}^{2}\EEB\left(|T_{(l)}-T_{(k)}|^{2\HT}\right).
\end{align*}
From this and the moments of the spacing $T_{(l)}-T_{(k)}$ as given in the Lemma \ref{lem:Tl-Tk_main}, we obtain
\begin{equation*}
    \EEB\left[(X_{T_{(l)}}-X_{T_{(k)}})^{2}\right]
    \asympp 
    L_{t_0}^{2} \left( \frac{l - k}{f(\T)(\mu+1)} \right)^{2\HT}.
\end{equation*}
Now, for any $1\leq k \leq K_0$, let $\varepsilon_{(k)}$ be a generic error term corresponding to the generic realization $X_{T_{(k)}}$, and denote
\begin{equation*}
    Y_{(k)} = X_{T_{(k)}} + \varepsilon_{(k)}.
\end{equation*}
Moreover, for $k$ such that $2k-1\leq K_0$, let
\begin{equation*}
    \theta_k = \EEB\left[(Y_{(2k-1)}-Y_{(k)})^{2}\right].
\end{equation*}
 Let $\sigma^2$ denote the variance of the error term, assumed to be finite.  We then obtain
\begin{equation}\label{eq:equivY2}
    \frac{\theta_k - 2\sigma^{2}}{L_{t_0}^{2}}
    \asympp 
    \left( \frac{k-1}{f(\T)(\mu+1)} \right)^{2\HT}.
\end{equation}
We distinguish two situations~: the case where $\sigma^2$ is known and the case where it is unknown. In the former case, we suppose that $4k-3$ is also less than $K_0$ and use twice the relationship~\eqref{eq:equivY2} with $k$ and $2k-1$, respectively.
We deduce
\begin{equation}
    \frac{\theta_{2k-1} - 2\sigma^{2}}{\theta_{k} - 2\sigma^{2}}
    \asympp 
    4^{\HT}.
\end{equation} 
Taking the logarithm on both sides, we obtain the proxy value
\begin{equation*}
    \HT(k, \sigma^{2}) = \frac{\log(\theta_{2k-1} - 2\sigma^{2})
    -  \log(\theta_{k} - 2\sigma^{2})}{2\log2},
\end{equation*}
of the local regularity parameter $\HT$, when $\sigma^2$ is given. In the case where $\sigma^2$ is unknown, assuming that $8k-7\leq K_0$,  we use the relationship~\eqref{eq:equivY2} three times with $k$, $2k-1$ and $4k-3$, respectively,
to obtain
\begin{equation}
    \frac{\theta_{4k-3} - \theta_{2k-1}}{\theta_{2k-1} - \theta_{k}}
    \asympp 
    4^{\HT}.
\end{equation}
A natural proxy of $\HT$ is then given by
\begin{equation}\label{eq:proxy}
    \HT(k) = \frac{\log(\theta_{4k-3} - \theta_{2k-1})
    -  \log(\theta_{2k-1} - \theta_{k})}{2\log2}.
\end{equation}

Our estimator of the local regularity parameter $\HT$ is the empirical version of the proxy value $ \HT(k) $, or $\HT(k, \sigma^{2}) $, built 
from a random sample of $\N0$ trajectories of $X$, the learning set of curves. Formally, we consider the sequence of events, for $1 \leq n \leq \N0$,
\begin{equation}\label{eq:eventA}
  \mathcal{B}_n = \mathcal{B}_n(\mu, \N0)
    =
    \left\{M_n \geq K_0, T^{(n)}_{(1)} \in J_\mu(\T), \dotsc,  T^{(n)}_{(K_0)} \in J_\mu(\T) \right\},
\end{equation}
and we define
\begin{equation}\label{eq:def-hattheta}
    \hat\theta_k = \frac{1}{\N0}\sum_{n=1}^{\N0}
    \big[Y^{(n)}_{(2k-1)}-Y^{(n)}_{(k)}\big]^{2}\mathbf{1}_{\mathcal{B}_n},
\end{equation}
where, for any $n$ and  $k$, $Y^{(n)}_{(k)}$ denotes the noisy measurement of $X^{(n)}(T_{(k)}^{(n)})$. 
If $\HT(k, \sigma^{2}) $ is indeed a good approximation of $\HT$, a simple estimator of $\HT$ when $\sigma^2$ is known is then 
\begin{equation}\label{eq:hatH-sigma}
    \hHT(k,\sigma^2) = \begin{cases}
        \dfrac{\log(\hat\theta_{2k-1} \!- 2\sigma^{2})
        -  \log(\hat\theta_{k}\! - 2\sigma^{2})}{2\log2}
        &\text{if}\; \min( \hat\theta_{2k-1},  \hat\theta_{k}) >2\sigma^2\\
        1 &\text{otherwise}.
    \end{cases}
\end{equation}
The default value 1 is arbitrary and  could be replaced by any number between 0 and 1.
When $\sigma^2$ is unknown the corresponding estimator is
\begin{equation}\label{eq:hatH}
   \hHT(k) = \begin{cases}
        \dfrac{\log(\hat\theta_{4k-3} - \hat\theta_{2k-1})
        -  \log(\hat\theta_{2k-1} - \hat\theta_{k})}{2\log2}
        &\text{if  $\hat\theta_{4k-3}>\hat\theta_{2k-1}>\hat\theta_{k}$}\\
        1 &\text{otherwise},
    \end{cases}
\end{equation}
where $\hat\theta_{4k-3}$ is obtained from the formula of $\hat\theta_{2k-1}$ after replacing $k$ by $2k-1$. 

It is worth noting that our estimator could be easily updated every time new curves are included in the learning sample, without revisiting the learning set already used. Indeed, one should only add new terms in the sums defining $\hat\theta_{k}$, $\hat\theta_{2k-1}$ and $\hat\theta_{4k-3}$.

The Associate Editor drew our attention on a large literature related to the estimation of the regularity of nonparametric functions.
\cite{gh07} consider that  data are noisy measurements of one sample path from a scaled fractional Brownian motion (fBm) with unknown Hurst parameter and unknown scale. The measurements are sampled on an equidistant grid, and the noise is allowed to be heteroscedastic. The authors derive the optimal rate for estimating the Hurst parameter. Moreover, they provide an estimator achieving the optimal rate. The estimator of \cite{gh07} is  based on the self-similarity property of the fBm process. Our estimator relies on a related property imposed to the second order moment of the increments; see \eqref{key_condX} above. Our condition defines a significantly larger class of processes. By construction, our results are not easily comparable to that of Gloter and Hoffmann, which are more refined but derived in a different, more restrictive context. Our estimator is designed to take advantage of the replication feature of functional data, where sample averages provide simple estimates for the moments of the increments. Another related topic extensively studied in the literature, is the construction of confidence or credible sets for a curve of unknown regularity. See, for instance, \cite{gn10,bu12,bn13}, or \cite{ra17} for the Bayesian approach. Such quite elaborate methods, which often require the calibration of some tuning parameters, are not specifically designed for the  functional data context. Finally, another related problem is testing the regularity of a signal. See, for instance, \cite{ca15}. However, such methods, designed to be applied to one signal, do not provide a direct estimator of the regularity.

\subsection{Concentration bounds for the local regularity estimator}

Below, we focus on the more complicated and realistic case with unknown variance. The case with given variance could be treated after obvious adjustments. 
The results in this section depend on $\mu$, the mean number of observation times $T$, and the cardinality $\N0$ of the learning set of curves. However,  they are non-asymptotic in the sense that they hold true for any sufficiently large $\mu$ and $\N0$ satisfying our conditions. 
For deriving our results, we impose the following mild assumptions.

\begin{assumptionH}
    \item\label{ass:DGP} The data consist of the pairs $(\Ynm , \Tnm ) \in\mathbb R \times I $ defined as in~\eqref{model_eq}, with $I\subset \mathbb R$ a compact interval,   and the realizations of $X$, $M$ and $T$ are mutually independent.
    
    \item\label{ass:T} The random variable $T$ admits a density $f : I \to \RR$ such that $f(\T)>0$. Moreover, there exist $L_f>0$ and $0<\beta_f\leq 1$ such that
$$|f(u) - f(v)|\leq L_f |u-v|^{\beta_f}, \qquad \forall u,v\in J_\mu(\T).$$
       
    \item\label{ass:L2} 
    There exist a function $\phi_\T (\cdot,\cdot)>0$, the  constants $L_{\T}, L_\phi>0$ and $0<\beta_\phi\leq 1$ such that, for any $u,v \in J_\mu(\T)$, we have
    \begin{align}
        \label{eq:def-phi}
        \EE\left[(X_{u}-X_{v})^{2}\right]
        &= L_{t_0}^{2}|u-v|^{2\HT} \left\{ 1+\phi_{\T}(u,v)\right\} \quad\text{and}\\  
        |\phi_\T(u,v)| &\leq L_\phi |u-v|^{\beta_\phi}.
    \end{align}

    \item\label{ass:Lp} 
    Two constants $\mathfrak{a}, \mathfrak{A}>0$ exist such that
    \begin{equation}
        \EE\big[|X_{u} - X_v|^{2p}\big]
        \leq 
        \frac{p!}2 
        \mathfrak{a}
        \mathfrak{A}^{p-2} |u-v|^{2p\HT},
\qquad \forall p\geq 2,\;\forall u,v\in J_\mu(\T)  .
    \end{equation}
    
    \item \label{ass:eps} The variables  $\varepsilon^{(n)}_m$, $ n, m \geq 1$, are independent copies of a centered variable $\varepsilon$, with finite variance $\sigma^2$, for which  
    constants $\mathfrak{b}\geq\sigma^{2}>0$ and $\mathfrak{B}>0$ exist such that
    \begin{equation*}
        \EE(|\varepsilon|^{2p}) \leq \frac{p!}2 \mathfrak{b}\mathfrak{B}^{p-2} ,\qquad \forall p\geq 1.
    \end{equation*}

    \item\label{ass:M} The random variable $M$ is such that $M\geq 9$ and  $\gamma_0>0$ exists such  that, for any $s>0$, $\PP\left( |M-\mu| > s \right) \leq \exp(-\gamma_0 s).$
\end{assumptionH}

Assumption \assrefH{ass:T} imposes a  mild condition on the distribution of the random observation points which provides convenient moment bounds for their spacings.
In particular, it implies that, for a sufficiently large $\mu$,
$f(\T)/2 \leq f(t)\leq 2 f(\T)$, $\forall t\in J_\mu(\T)$.
Assumption \assrefH{ass:L2} is a version of the so-called local stationarity condition. More precisely, \assrefH{ass:L2} implies that the trajectories of the process
$X = (X_u : u\in I)$ are  Hölder continuous in quadratic mean in the neighborhood of $\T$, with exact exponent $\HT$ and  local Hölder constant $L_\T$. Let us call $\HT$ the \emph{local regularity of the process $X$ at $\T$}. Examples include, but are not limited to,  stationary or stationary increment processes $X$. 
See, \emph{e.g.}, \cite{blanke2014} for some examples and references on processes satisfying the mild condition in \assrefH{ass:L2}. 
Assumptions \assrefH{ass:Lp} and \assrefH{ass:eps} are needed for deriving exponential bounds for the concentration of our local regularity estimator, and are satisfied by the sub-Gaussian random variables. In particular, \assrefH{ass:Lp} is satisfied by the Gaussian processes. \assrefH{ass:M}  is a mild condition for controlling the variability of number of observation points on the curves. The lower bound on $M$ guarantees that each curve in the learning set has a sufficient number of observation times for building our estimator.  
For a real number $a$, let $\lfloor a \rfloor$ denote the largest integer not exceeding $a$. 

\begin{theorem}\label{thm:Ht0}
    Let Assumptions~\assrefH{ass:DGP}--\assrefH{ass:M} hold true.  Let $K_0$ be a positive integer  such that 
    \begin{equation}\label{eq:selectK0}
        (\mu+1)^{\frac{\beta_f\alpha}{4+\beta_f\alpha}}\leq K_0 \leq  \frac{\mu}{2 \log(\mu)},
    \end{equation}
    with $\alpha = 2\HT+\beta_\phi \in (0,3]$.
    Let 
    \begin{equation}\label{eq:cdt_epsi}
        \mathfrak{c}
          \left( \frac{K_0-1}{f(\T)(\mu+1)} \right)^{\min(\beta_\phi, \beta_f \HT /2)} < \epsilon <  \frac{2}{\log2} ,
    \end{equation}
    with $\mathfrak{c}$ a constant depending only on $L_f$, $\beta_f$, $\beta_\phi$, $f(\T)$ and $\HT$.
    Define $k=\lfloor (K_0+7)/8\rfloor $ and let $\hHT=\hHT(k)$ be defined as in \eqref{eq:hatH}.
    Then, for a sufficiently large $\mu$~:
    \begin{equation}\label{eq:concen}
        \PP\left( \big|\hHT - \HT\big| > \epsilon \right)
        \leq 
        12\exp\left[ -\mathfrak{f} 
        N_0 \epsilon^2 \left( \frac{k-1}{f(\T)(\mu+1)} \right)^{4\HT}\right],
    \end{equation}
    where $\mathfrak{f}$ is a positive constant depending on $\mathfrak{a, A, b, B}$ and the length of the interval $I$.
\end{theorem}

To obtain a non-trivial  estimator of $\HT$, we need $k\geq 2$, thus the upper bound in  \eqref{eq:selectK0} should be larger than 9, and this happens as soon as $\mu\geq 80$. For the estimator in \eqref{eq:hatH-sigma}, which requires an estimate of $\sigma^2$, we would only need $\mu > 35$. The exact expressions of the constants  $\mathfrak{c}$ and $\mathfrak{f} $ could be traced in the proof of Theorem  \ref{thm:Ht0}. The condition imposed on $K_0$ provides a panel of choices depending on $\N0$ and $\mu$. 
As a result, up to some constants, and depending on $L_f$, $\beta_f$, $\beta_\phi$, $f(\T)$ and $\HT$, the concentration rate $\epsilon$, one could expect, could be in a range such that $\epsilon\mu \gg 1$ and $\epsilon \log ^{1/2} (\mu) \ll 1$. The best possible concentration of $\hHT$ is guaranteed as soon as  $\N0$ is larger than some power of $\mu$, while for a concentration as fast as some negative power of $\log (\mu)$,  one only needs a small number  $\N0$ of curves in the learning set, that is larger than some power of $\log (\mu)$.

For the purpose of building an adaptive optimal kernel estimator for the trajectories of $X$, we will impose $\epsilon = \log^{-2}( \mu)$ and an exponential bound  equal to $ \exp(-\mu)$. The following corollary proposes a data-driven choice of $K_0$ which guarantees these requirements. This choice is guided by the fact that, for any constants $a,b>0$,  we have the relationship $
\log^a (\mu) \leq \exp ((\log\log(\mu))^2) \leq
\mu^b,
$
provided $\mu$ is sufficiently large.

\begin{corollary}\label{cor:enfin}
    Assume the conditions of Theorem \ref{thm:Ht0} hold true. Let
    $$
    \widehat \mu = \N0^{-1} \sum_{n=1}^\N0 M_n, \qquad
    \widehat K_0 = \lfloor \widehat \mu \exp(-(\log\log(\widehat \mu  ))^2) \rfloor,
    $$
    and $\hHT=\hHT( \lfloor   (\widehat K_0+7)/8 \rfloor )$, with $ \hHT$ defined in \eqref{eq:hatH}.
  Then, for any constant $C>0$,
    \begin{equation}\label{eq:concen4}
        \PP\left( \big|\hHT- \HT\big| > C\log^{-2}( \mu) \right)
        \leq
        \exp(-\mu),
    \end{equation}
   provided   $N_0\geq \mu^{1+b}$ for some $b>0$ and  $\mu$ is sufficiently large. 
    
\end{corollary}

The conditions in Corollary \ref{cor:enfin} impose $\mu$ to be large, but still allow for many cases in the three regimes non-dense, dense and ultra-dense, as defined by \cite{zhang_sparse_2016}. 

One could also  build $\hHT$ with only one trajectory of a 
process $X$ with stationary increments. 
If the density of $T$ is uniform and sufficiently many measurements are available, it suffices to split the interval $[0,1]$ into $\N0$ intervals of the same length  and apply our methodology considering the measuring times and the noisy measured values in each block as belonging to a different curve in the learning set. Theorem \ref{thm:Ht0} and Corollary \ref{cor:enfin} remain valid.

\subsection{The case of conditionally heteroscedastic noise}\label{sec:hetero}

In some applications, the assumption of constant variance for the error term $\varepsilon$ could be unrealistic. Therefore, we consider the following conditional heteroscedastic error extension of model \eqref{model_eq}: 
\begin{equation}\label{model_eq2}
    \Ynm = \Xn({\Tnm}) + \sigma\left(  \Xn({\Tnm}), \Tnm\right) \; \unm, \; 1\leq n \leq N, 1\leq m \leq M_n,
\end{equation}
where $\sigma(\cdot,\cdot)$ is some unknown function and $\unm$  are independent copies of a centered variable $u$ with unit variance. 

Our approach also applies to the model \eqref{model_eq2} under some additional mild conditions. Indeed, assuming the expectations exist, we have
\begin{align*}
   \theta_k= \EEB\left[(Y_{(2k-1)}-Y_{(k)})^{2}\right] &= \EEB\left[(X_{T_{(2k-1)}}-X_{T_{(k)}})^{2}\right] \\ &\qquad + \EEB\left[ \sigma^2 \left(  X_{T_{(2k-1)}}, T_{(2k-1)}\right) \right]  \\ &\qquad + \EEB\left[ \sigma^2\left(  X_{T_{(k)}}, T_{(k)}\right) \right]  .
\end{align*}
From this identity it is clear that the arguments presented in Section \ref{sec:heur} remain valid as long as the value of the last two expectations on the right-hand side of the last display does not depend on $k$. Thus, in this case, even if the conditional variance of $\enm$ is not given, we could consider the same estimator $\hHT$. This remark leads us to the following additional assumption. 

\begin{assumptionE}
    \item\label{ass:Hhetero} 
    The variables $u^{(n)}_m$ from model \eqref{model_eq2} satisfy the Assumption \assrefH{ass:eps} with unit variance. Moreover, the function $\sigma(\cdot,\cdot)$ is bounded and the map $u\mapsto\EE\left[ \sigma^2( X_u,u) \right]$,  $u\in I$, is constant in a fixed neighborhood of $\T$.  
\end{assumptionE}

Assumption \assrefE{ass:Hhetero} allows the error term to be conditionally heteroscedastic, but imposes marginal (unconditional) homoscedasticity in a neighborhood of $\T$. 

Under Assumption \assrefE{ass:Hhetero}, for any $k$ we have
\begin{align*}
\EEB\left[ \sigma^2(  X_{T_{(k)}}, T_{(k)} ) \right]  &= \EE \! \left[ \EE\left(  \sigma^2(  X_{T_{(k)}}, T_{(k)}) \mid M,T_1,T_2,\ldots,T_M \right)  \mathbf 1 _{\mathcal B} \right] \\&= \EE\left[ \sigma^2( X_u, u) \right] \PP(\mathcal B),
\end{align*}
and thus the  terms like $\EEB[ \sigma^2(  X_{T_{(k)}}, T_{(k)} ) ] $ cancel when considering the differences $\theta_{4k-3} - \theta_{2k-1}$ and $\theta_{2k-1} - \theta_{k}$.

\begin{corollary}\label{thm:Ht0b}
    Assume the observations consist of  the pairs  $(\Ynm , \Tnm ) \in\mathbb R \times I $ where $\Ynm $  defined as in \eqref{model_eq2} and the realizations of $X$, $M$ and $T$ are mutually independent. Assume that Assumptions~\assrefH{ass:T}--\assrefH{ass:Lp}, \assrefH{ass:M}, \assrefE{ass:Hhetero} hold. Then Corollary \ref{cor:enfin} remains valid with the same local regularity estimator $\hHT$.  
\end{corollary}

The proof of Corollary \ref{thm:Ht0b} follows from the proof of Theorem \ref{thm:Ht0} after obvious modifications, and hence will be omitted. It is worthwhile noting that, even if the regularity $\HT$ is the same at any point $\T$, one may not be able to estimate the regularity $\HT$ using only  one observed noisy trajectory with conditionally  heteroscedastic noise. This because, intuitively, it might be impossible to identify the oscillations of the signal of interest, that is to separate the increments of the trajectory of $X$ from the differences of the error terms with variable variance. With our approach  based on local observed increments averaged over several curves, the effect of the noise vanishes, provided the expectation of the conditional variance is constant. Hence, eventually the identification of the oscillations of $X$ is recovered and there is no difference with respect to the case of homoscedastic errors.

\subsection{The case of differentiable sample paths}\label{sec:xheur}

The definition of the local regularity extends to the case of differentiable curves. When the curve admits derivatives up the order $\KT>0$, condition \eqref{key_condX} has to be stated for the $\KT$-th order derivative of the curve. To build an estimate of the local regularity of the $\KT $-th derivative of the curve, we propose to use a smoothing-based approximation of the $\KT $-th derivative.  In the Appendix, we derive concentration bounds for the estimator of $\KT+\HT$ when $\KT>0$. In particular, we propose an estimator of $\KT$. The implementation of the estimator of $\KT+\HT$ is described in Section \ref{sec:simuls}, and the simulation results we report shows that it performs well. However, many real data we analyzed, revealed that in many applications, the sample paths do not seem differentiable. For this reason, and to save space, we leave to the supplement, the details of learning the smoothness in the general case $\KT\geq 0$.

%% file: main/adaptive_smoothing.tex

\section{Adaptive optimal smoothing}\label{sec:ad_opt}

With at hand an estimate of the local regularity $\ST = \KT + \HT$ obtained from a learning set of $N_0$ curves, we aim at recovering $N_1$ new noisy trajectories of $X$ from what we call the  online dataset. One of the most popular smoother is the local polynomial estimator, see \cite{fan_local_1996}. 
This estimator  depends on a tuning parameter, the bandwidth, which should ideally be chosen according to the
regularity of the target function. Using the local regularity estimate, one can build optimal smoothing using alternative approaches, such as the splines. Here we focus on local polynomials.

One has to connect a definition of local regularity that is meaningful from the theory of stochastic processes to the usual definition of function regularity used in nonparametric curve estimation. Fortunately, in our framework, the parameter $\ST$, which is understood as the local regularity of the process $(X_{t} : t\in J_\mu(\T))$ in \emph{quadratic mean}, see~\eqref{key_condX}, is intrinsically linked with the regularity of the sample paths of the process.  Indeed, in many important situations, which are covered by our assumptions, the regularity of the sample paths of a process does not depend on the realization of this process. For example, the regularity of any Brownian path is $1/2$,  in the sense that for any $\epsilon>0$, almost surely the sample path belongs to the Hölder space $\mathcal C^{1/2-\epsilon}(I)$ and does not belong to $\mathcal C^{1/2+\epsilon}(J)$ whatever $J\subset I$. Here, for any $a>0$,  $\mathcal C^{a}(I)$ denotes the space of uniformly $a-$Hölder continuous functions defined on $I$, see Theorem~2.2 and Corollary~2.6 of \cite{MR1725357} for precise definition.
More generally the regularity of the sample paths of a process is linked to integrated regularities through the Kolmogorov's Continuity Theorem~\cite[][Theorem~2.1]{MR1725357}. In particular, Assumption~\assrefH{ass:L2} ensures that, with probability 1, the trajectories of the process $(X_t ^{(\KT)}: t\in J_\mu(\T))$ are
Hölder continuous with any  exponent parameter $0<a < \HT$. 

Below, we define the local polynomial estimator and derive its theoretical properties. Since our focus of interest is the \emph{simultaneous} denoising of the additional $\N1$ curves, we consider the following pointwise risk: for a generic estimator $\hXT1$ of $ \XT1$, let
\begin{equation}\label{risk_max}
    \mathcal R (\widehat X;\T) =  \EE\left[  \max_{1\leq \n1 \leq \N1} \left|\hXT1 - \XT1 \right|^2 \right].
\end{equation}
First, we provide a sharp bound for this risk with $\N1=1$, in the case where a suitable estimator of $\ST= \KT + \HT$, computed from another independent sample, is given. 
Such a result, of interest in itself in nonparametric curve estimation,
seems to be new. In this case, the expectation defining the risk $ \mathcal R (\widehat X;\T) $ should be understood as the conditional expectation given the estimator of $\ST$. Next, we provide a sharp bound for $\mathcal R (\widehat X;\T)$ in the case where $\N1 \geq 1$ and the estimator of $\ST$ is obtained using the approach introduced in Section \ref{sec:local-regularity}.

\subsection{Local polynomial estimation}

We assume that $\KT\geq 0$ is an integer and $\HT\in(0,1)$. Let $\hKT$ and $\hHT$ be some generic estimators of $\KT$ and $\HT$, respectively, and let $\hST = \hKT + \hHT$ be the corresponding estimator of $\ST = \KT + \HT$. We assume that  $\hKT$ and $\hHT$ are independent of the $\N1$ from the online dataset, generated according to \eqref{model_eq}.

The estimator of $\ST$ could be used to smooth any curve $Y^{[\n1]}$ ($\n1= 1 ,\dotsc,\N1$) from the online dataset.
For the sake of readability, we omit the superscript $[\n1]$ and  we consider a generic curve from the online dataset:
\begin{equation*}
    Y_m = X(T_m) + \varepsilon_m,\qquad 1\leq m \leq M. 
\end{equation*}
For any $u\in\RR$, we consider the vector $U(u) = (1, u, \dotsc, u^{ \hKT}/\hKT!)$. Let $K:\RR\to\RR$ be a positive kernel and define:
\begin{equation}\label{eq:loc-poly-min}
    \vartheta_{M,h}=
    \operatorname*{arg\,min}_{\vartheta\in\RR^{\degree+1}}
    \sum_{m=1}^{M}\left\{ Y_m - \vartheta^\top 
    U\left(\frac{T_m-\T}{h}  \right)\right\}^2
    K\left(\frac{T_m-\T}{h}  \right),
\end{equation}
where $h$ is the bandwidth. 
The vector $\vartheta_{M,h}$ satisfies the normal equations $A \vartheta_{M,h} = a$ with
\begin{align}
    A = A_{M,h} &= \frac{1}{Mh} \sum_{m=1}^{M} U\left( \frac{T_m-\T}{h} \right)U^\top\left( \frac{T_m-\T}{h} \right)K\left( \frac{T_m-\T}{h} \right)\label{eq:Anstar}\\
    a = a_{M,h} &= \frac{1}{Mh} \sum_{m=1}^{M}
    Y_m U\left( \frac{T_m-\T}{h} \right) K\left( \frac{T_m-\T}{h} \right).\label{eq:anstar}
\end{align}
Let  ${\lambda}$ be the smallest eigenvalue of the matrix ${A}$ and remark that, whenever ${\lambda}>0$, we have ${\vartheta}_{M,h} = A^{-1} {a}$.

Taking into account the expression of the bandwidth minimizing the pointwise mean squared risk for a regression function defined on $I$, with derivative of order $\KT$  which is Hölder continuous in a neighborhood of $\T$, with exact exponent $\HT$, we consider a  bandwidth 
\begin{equation*}
   \widehat h \sim M^{- 1/(2\hST +1)}.
\end{equation*}
Our focus of interest is on determining a nearly optimal rate of the bandwidth to be used to recover the trajectories of $X$. For the applications, one could also be interested in a nearly optimal constant, which in general needs to be estimated. In Section \ref{sec:simuls} we propose a simple way to estimate a suitable constant for the applications. 

With at hand the bandwidth $\widehat h$, we propose the following definition of the local polynomial estimator of $X_{\T}$ of order $\degree$:
\begin{equation}\label{eq:loc-poly}
    \hat X_\T =
    \begin{cases}
        U^\top(0) \widehat{\vartheta}
        &\text{if } \lambda > \log^{-1}(M)  \text{ and } | U^\top(0) \widehat{\vartheta} | \leq \widehat \tau^{5/12}(M) 
        \\
                \widehat \tau^{5/12}(M) 
        &\text{if } \lambda > \log^{-1}(M)  \text{ and } | U^\top(0) \widehat{\vartheta} | > \widehat \tau^{5/12}(M) 
        \\
        0 &\text{otherwise,}
    \end{cases},
\end{equation}
where $\widehat \vartheta = \vartheta_{M,\widehat h}$ and, 
for any $y>1$,
\begin{equation*}
  \widehat   \tau(y) = \frac{1}{ \log^{2}(y)}
    \left(\frac{y}{\log (y)}\right)^{2\hST/(2\hST+1)}.
\end{equation*}
The upper trimming with $\widehat \tau^{5/12}(M) $ is a technical device used to control the tails of $ \hat X_\T$. It has practically no influence in applications.  
For deriving our results on $ \hat X_\T$, we impose the following mild assumptions. 

\begin{assumptionLP}
    
    \item\label{ass:Lp2} 
    There exist two positive constants, $\mathfrak{a}$ and $\mathfrak{A}$, such that for any $p \geq 1:$ 
    \begin{equation}
        \EE\big[|X_{t}|^{2p}\big]
        \leq 
        \frac{p!}2 
        \mathfrak{a}
        \mathfrak{A}^{p-2},\quad \forall t\in[0,1] .
    \end{equation} 
    Moreover, 
for $\KT \geq 0$ in Assumption \assrefH{ass:L2}, 
if $\nabla^{\KT} X$ denotes the $\KT$--th derivative of the process $X$, 
  \begin{equation}
        \EE\big[|  \nabla^{\KT}
 X_{u} - \nabla^{\KT}
X_v|^{2p}\big]
        \leq 
        \left(
        \frac{p!}2 
        \mathfrak{a}
        \mathfrak{A}^{p-2} 
        \right)
        |u-v|^{2p\HT}, \; \forall p\geq 2,\;\forall u,v\in J_\mu(\T)  .
    \end{equation}

    \item\label{ass:M2} We assume that, almost surely, $\mu/\log(\mu) \leq M \leq \mu\log(\mu).$

    \item\label{ass:hatHt} 
    The estimator $\hHT$ satisfies the property
    \begin{equation*}
        \PP\left( \big|\hHT - \HT\big| > \log^{-2}( \mu) \right)
        \leq \mathfrak{K}_1\exp(-\mu),\qquad \forall \mu >0,
    \end{equation*}
    where  $\mathfrak{K}_1$ is some positive constant.
    
    \item\label{ass:hatkt0} 
    The estimator $\hKT$ satisfies the property
    $
    \PP (  \hKT \neq \KT )\leq \mathfrak{K}_1\exp(-\mu)$ , $\forall \mu >0.
    $
\end{assumptionLP}

The first part of Assumption~\assrefLP{ass:Lp2} provides a suitably tight control on the moments of $X_t$, but still allows for unbounded trajectories. The second part of   Assumption~\assrefLP{ass:Lp2}, is a technical condition which reinforces Assumption \assrefH{ass:Lp}. It allows to control the analytic regularity of the sample paths. More precisely, it is implicitly used in the definition of the variable $\Lambda_\beta$ in \eqref{eq:lambda_RY}. Assumption~\assrefLP{ass:M2}   is a convenient, but mild, technical condition. It could be relaxed at the price of controlling the probability of the complement of the event $\{\mu/ \log(\mu) \leq M \leq \mu\log(\mu)\}$, for instance using \assrefH{ass:M}. Assumptions~\assrefLP{ass:hatHt} and \assrefLP{ass:hatkt0} are very mild conditions that the generic estimators of the regularity should satisfy. Since $\mu ^{1/\log^2(\mu)} = e^{1/\log(\mu)}$ for any $\mu >1$, the concentration of $\hHT$ at a suitable negative power of $\log (\mu)$ will suffice for the smoothing purposes. For simplicity, and without loss of generality we consider the same constant $\mathfrak{K}_1$ in  Assumptions~\assrefLP{ass:hatHt} and \assrefLP{ass:hatkt0}.

\begin{theorem}\label{thm:exponential-bound}
    Assume that Assumptions~\assrefH{ass:DGP}, \assrefH{ass:T}, \assrefH{ass:eps} and \assrefH{ass:M} and Assumptions~\assrefLP{ass:Lp2}--\assrefLP{ass:hatkt0} hold true and let $K(\cdot)$ be a kernel such that, for any $t\in\RR :$
    \begin{equation}
        \label{eq:bounded-kernel}
        \kappa^{-1}\1_{[-\delta, \delta]}(t)  \leq K(t) \leq \kappa \1_{[-1,1]}(t) ,\quad \text{for some } 0<\delta<1 \text{ and } \kappa\geq1.
    \end{equation}
    There then exists a constant $\Gamma_{\!0}$ such that for any $\mu\geq 1$,
    \begin{equation*}
        \EE\left[
        \exp\left\{\left(
        \tau(\mu)
        \left|\hat X_\T - X_\T\right|^2
        \right)^{1/4}\right\}
        \right]
        \leq \Gamma_{0}
        \end{equation*}
        where
        \begin{equation*}
        \tau(\mu) = \frac{1}{ \log^{2}(\mu)}\left(\frac{\mu}{\log (\mu) }\right)^{\frac{2\ST}{2\ST+1}}.
        \end{equation*}
    
\end{theorem}

The bound on the $\exp(\sqrt{x})-$moment of the $|\hat X_\T - X_\T|$ seems a new result for local polynomial estimators. For our purposes, it will entail a sharp bound for   $\mathcal R (\widehat X;\T) $. More precisely,  the price for considering a risk measure uniformly over the whole online dataset is very low, that is a multiplying factor as large as a power of $\log (\N1)$  in the risk bound we derive below. \cite{gaiffas2007} derived sharp bounds for all the moments of  $|\hat X_\T - X_\T|$. However, his bounds on the moments would induce a power of $\N1$ as multiplying factor for our risk bound, instead of the power of $\log (\N1)$.

\begin{theorem}\label{thm:local-poly}
    Assume that assumptions of Theorem \ref{thm:exponential-bound} hold true, and let $K$ be a kernel which satisfies~\eqref{eq:bounded-kernel}. There then exists a positive constant $\Gamma_{1}$ such that
    \begin{align*}
      \mathcal R (\widehat X;\T) &= \EE\!\!\left[\!
        \max_{1\leq \n1 \leq  \N1} \left|\hXT1 - \XT1\right|^2
        \right] \\ &\leq \Gamma_{ 1} \log^2  (\mu)  \log^4  (1+N_1) \{\log (\mu)\}^{\frac{2\ST}{2\ST+1}}\mu ^{- \frac{2\ST}{2\ST+1}}.
    \end{align*}
\end{theorem} 

If all the trajectories $X$ were in $\mathcal{C}^{\ST}(J_\mu(\T))$, $\ST$ were known and $N_1=1$, the risk bound for $\mathcal{R}(\widehat X;\T)$  would be of the usual nonparametric rate $\mu^{-2\ST/(2\ST+1)}$.
Let us note that the fact that $\HT$ is not known does not have any consequence on the risk bound in Theorem~\ref{thm:local-poly}. Indeed, since $\mu^{1/\log^2\mu}=e^{1/\log\mu}$ for any $\mu$, the order of the  risk bound does not change as soon as the probability of the event $\{  \hKT = \KT \}\cap \{|\hHT-\HT|\leq 1/\log^2\mu\}$ tends to 1. The $\log (1+N_1)$ factor is given by the maximum over the $\N1$ curves in the online dataset. The factor $\{\log(\mu)\}^{2\ST/(2\ST+1)}$ is due to the concentration properties of $M$ around its mean $\mu$. This factor would not appear if $M/ \mu$ is almost surely bounded and bounded away from zero. The factor $\log^2(\mu)$ comes from  probability theory. The trajectories of a stochastic process $X$ with local regularity $\HT$ does not necessarily belong to $\mathcal{C}^{\ST}(J_\mu(\T))$ but they are almost surely in any $\mathcal{C}^{\ST-\epsilon}(J_\mu(\T))$ for any $0<\epsilon <\ST$. 

Finally, let us notice that Corollary~\ref{cor:enfin} states that  the estimator defined by~\eqref{eq:hatH} satisfies \assrefLP{ass:hatHt} for $\KT=0$ and any $0<\HT<1$. This leads us to the following result. 

\begin{corollary}\label{corr:local-poly}
    Assume $\KT=0$ and let $\hHT$ be the estimator of $0<\HT<1$ defined in Corollary \ref{cor:enfin}.   
    Moreover,  Assumptions~\assrefH{ass:DGP}--\assrefH{ass:M} and Assumptions~\assrefLP{ass:Lp2}--\assrefLP{ass:M2} hold true. If $\N0\geq \mu^{1+b}$ and $\N1\leq \mu^B$ for some $b,B>0$, then
    \begin{equation*}
           \mathcal R (\widehat X;\T) 
        \leq \Gamma_{1} B^4    \log^7(\mu)  \mu ^{- \frac{2\HT}{2\HT+1}}.
    \end{equation*}
\end{corollary}

Finally, we establish the uniform convergence of the recovered trajectories. Below, the uniformity is with respect to $\T\in I$ and over all the curves in the online sample.

    \begin{theorem}\label{thm:local-poly-sup}
        Assume that assumptions of Theorem \ref{thm:exponential-bound} hold true, uniformly, for any $\T\in I$. Let $\varsigma = \ST$ be the global regularity of the process. Let also $K$ be a kernel which satisfies~\eqref{eq:bounded-kernel}. There then exists a positive constant $\Gamma_{\!1}$ such that
        \begin{equation*}
        \EE\!\!\left[\!
            \max_{1\leq \n1 \leq  \N1} 
            \sup_{\T\in I}\left|\hXT1 - \XT1\right|^2
            \right] \! \leq \Gamma_{\!1} \Psi(\mu,\N1).
        \end{equation*}
        \begin{equation}\label{eq:Psi}
            \Psi(\mu,\N1) = \frac{ \log^4(\N1) + \log^8(\mu)  }{ \log^{2}(\mu)}
    \left(\frac{\log(\mu)}{\mu}\right)^{\frac{2\varsigma}{2\varsigma+1}}.
        \end{equation}
    \end{theorem}

%% file: main/empirical_analysis.tex
\section{Empirical analysis}\label{sec:empi}

In the usual local polynomial (LP) smoothing framework, for a regression function defined on $I$, given a sample of size $M$, a bound of the pointwise, mean squared error risk is first derived and this bound is then minimized with respect to the bandwidth $h$.  See for instance \cite{tsybakov2009}. When the regression function admits a derivative of order $\KT$ which is Hölder continuous in a neighborhood of $\T$, with exact exponent $\HT$ and  local Hölder constant $L_\T$, 
the optimal bandwidth  is 
\begin{equation}\label{eq:b}
h_{opt} = \left(\frac{C}{M}\right)^{1/(2\ST + 1)}
\quad \text{with}
\quad C = C_\T = \frac{q_2}{2\ST q_1^2},
\end{equation} 
with $q_1$ and $q_2$ defined on pages 39-40 of \cite{tsybakov2009}. With the Nadaraya-Watson estimator, we take 
$q_1= L_\T \{ \lfloor \ST \rfloor ! \}^{-1} \int   K(v)   \lvert v \rvert^{\ST}dv$ and $q_2=\sigma^2_\T f(\T)^{-1} \int K^2(v) dv $, where  $\sigma^2_\T$ is the variance of the noise that could depend on $\T$.  This yields a refined constant $C$ in this case. Details are given in the Appendix.  The  value $f(\T)$ can be estimated using all the $\Tnm$. 

Thus our target bandwidth $h_{opt}  $ depends on two more unknown quantities, $L_\T$ and $\sigma^2_\T$, for which we now propose estimation procedures. 
The estimation of $L_\T$ could be based on similar ideas as used for $\HT$. For simplicity, we assume $\KT=0$. The extension to the case $\KT \geq 1$ could follow the same pattern as for the estimation of the local regularity, using the trajectories of the derivatives. 
Using twice the relationship \eqref{eq:equivY2} with $k$ and $2k-1$, respectively,  we deduce 
$$L_\T^2 \approx \frac{\theta_{2k-1} - \theta_{k}}{4^{\HT} - 1}\left(\frac{f(\T)(\mu + 1)}{k - 1}\right)^{2\HT}  .$$
On the other hand, using  the approximation of the moments of the spacings, as given in Lemma \ref{lem:Tl-Tk_main}, we have
\begin{align*}
\eta_{2k-1}  - \eta_{k} &:= \EEB\left[|T_{(4k-3)}-T_{(2k-1)}|^{2H_{\T}}\right] - \EEB\left[|T_{(2k-1)}-T_{(k)}|^{2H_{\T}}\right]  \\ &\approx (4^{\HT} - 1) 
\left(\frac{k - 1}{f(\T)(\mu + 1)}\right)^{2\HT}.
\end{align*}
Given an estimator of $\HT$, the empirical counterparts of $\eta_k$ obtained from the learning set of $N_0$ independent trajectories of $X$ is
$$\widehat{\eta}_k = \frac{1}{\N0}\sum_{n = 1}^{\N0} \left\lvert T_{(2k - 1)}^{(n)} - T_{(k)}^{(n)} \right\rvert^{2\widehat{H}_\T}\mathbf{1}_{\mathcal{B}_n},$$
where   $\mathcal{B}_n$ is  the sequence of events
defined in \eqref{eq:eventA}. An estimate of $\eta_{2k-1}$ could be obtained similarly. 
These facts lead us to the following estimator of the local Hölder constant $L_\T$~:
\begin{equation}\label{est:Lt0}
\widehat{L}_\T^2 =\widehat{L}_\T^2  (\hHT) = \begin{cases}
        \dfrac{\widehat{\theta}_{2k-1} - \widehat{\theta}_k}{\widehat{\eta}_{2k-1} - \widehat{\eta}_k}
        &\text{if  $\widehat\eta_{2k-1}>\widehat\eta_{k}$ and $\widehat{\theta}_{2k-1} > \widehat{\theta}_k$,}\\
        1 &\text{otherwise}.
    \end{cases}
    \end{equation}
For the implementation we propose $k=\lfloor (\widehat K_0 +7)/8 \rfloor$ with $$\widehat K_0 = \lfloor\widehat{\mu}\exp\left(-(\log\log \widehat{\mu})^2\right)\rfloor \quad \text{and} \quad 
\widehat \mu = \N0^{-1} \sum_{n=1}^\N0 M_n.$$

To estimate the variance, we propose
\begin{equation}\label{eq:sigma2_estim}
\widehat{\sigma}^2=\widehat{\sigma}^2 _\T = \frac{1}{\N0}\sum_{n = 1}^{\N0} \frac{1}{2|\mathcal{S}_n|}\sum_{m \in  \mathcal S_n }
\left[Y_{(m)}^{(n)} - Y_{(m-1)}^{(n)}\right]^2,
\end{equation}
where $\mathcal S_n\subset \{2,3,\ldots,M_n\}$ is a set of indices for the $n-$th trajectory and $|\mathcal S_n|$ is the cardinal of $\mathcal S_n$. When the variance of the error $\varepsilon$ is considered constant, one could take $\mathcal S_n = \{2,3,\ldots,M_n\}$. When it depends on $\T$, one could take  
$$\mathcal S_n = \left\{m: T_{(1)}^{(n)}\leq  T_{(m)}^{(n)} \leq T_{(\widehat K_0)}^{(n)} \right\},$$
with $\widehat K_0$ defined above.  
This is the choice we used in our empirical investigation. When the variance of the errors also depends on the realizations $X_u$, as described in Section \ref{sec:hetero}, in general it is no longer possible to  consistently estimate $\sigma^2 ( X_\T, \T)$. Our simulation experiments indicate that the estimate \eqref{eq:sigma2_estim} remains a reasonable choice. 

Finally, the constant involved in the definition of the bandwidth could be estimated by $\widehat C$ obtained by plugging the estimates of the unknown quantities into the definition of $ C$ in \eqref{eq:b}. Concerning the kernel, we use $K(t) = (3/4)\left(1 - t^2\right)\mathbf{1}_{[-1, 1]}(t), $ that is the Epanechnikov kernel for which  $\| K \|^2 =3/5$ and $\int \lvert K(v) \rvert \lvert v \rvert^{\ST}dv = 3\{(\ST+ 1)(\ST+ 3)\}^{-1}.$

\subsection{Simulation experiments}\label{sec:simuls}

We now illustrate the behavior of our local regularity estimator $\hST =\hKT+\hHT$ computed using the \emph{learning} set of noisy curves, and the performance of 
kernel smoother it induces for estimating the noisy curves from the \emph{online} set. The procedure for calculating $\hST$ is summarized in the following algorithm where $LP(d)$ means local polynomial smoother with degree $d\geq 0$. The Nadaraya-Watson  smoother corresponds to $LP(0)$. 

\begin{algorithm*}[H]
\SetAlgoLined
\KwResult{Estimation of $\ST$ from the learning set of $\N0$ noisy curves}
 Calculate $\widehat \mu = \N0^{-1} \sum_{n=1}^\N0 M_n$ and $\widehat K_0 = \lfloor\widehat \mu \exp(-(\log\log(\widehat \mu))^2)\rfloor$\;
 Calculate $\hHT$ and set  $\hKT = 0$\;
 \While{$\hHT > 1 - \log^{-2}(\widehat \mu)$}{
    Calculate  $\widehat L_\T(1)$, as in \eqref{est:Lt0}, and $\widehat \sigma^2_\T $\;
    Calculate $\widehat C$ with $\hST= \hKT+\hHT$, $\widehat L_\T(1)$  and $\widehat \sigma^2_\T $\;
    Calculate the bandwidth $ \widehat h_n =  (\widehat C/M_n )^{1/(2\hST + 1)}$, $1\leq n \leq \N0$\;
    Estimate the $(\hKT+1)-$th derivative of the trajectories of $X$ with $LP( \hKT+1)$\;
    Calculate $\hHT$ using the estimated trajectories of the $(\hKT+1)-$th derivative\;
    Set $\hKT = \hKT + 1$\;
 }
 \label{alg:reg1}
 \caption{Estimation of the local regularity $\ST=\KT+\HT$}
\end{algorithm*}

For the curve estimation, we use $(Y_1^{[n_1]}, T_1^{[n_1]}), \ldots, (Y_{M_{n_1}}^{[n_1]}, T_{M_{n_1}}^{[n_1]})$, $1\leq n_1 \leq \N1$, and $LP(\hKT)$ with $\hKT$ delivered by Algorithm \ref{alg:reg1}. 
The bandwidth is calculated as $ \widehat h_{n_1} =  (\widehat C/M_{n_1} )^{1/(2\hST + 1)}$, $1\leq n_1 \leq \N1$, with 
$\hST$ obtained from Algorithm \ref{alg:reg1}. The constant estimate $\widehat C$ is the same for all curves in the \emph{online} set, that is that obtained with $\hST$, $\widehat L_\T(\hHT)$ and $\widehat \sigma^2_\T $. We compare  our approach with the classical   \emph{cross-validation (CV)} (least-squares leave-one-out) method applied for each curve $\X1$ separately. For CV, we use the \textbf{\textsf{R}} package \texttt{np} \cite{hayfield_nonparametric_2008}, after rescaling the CV bandwidth to account for their different definition of the Epanechnikov kernel. At this stage, we want to point out that our smoothing method is much faster than any standard, trajectory-by-trajectory approach, such as CV. We report a time comparison in the Appendix, and as expected, the ratio between the times needed for CV and for our approach is at least of the same order as $\N1$. It is worth noting that one cannot follow an \emph{ad-hoc} approach and transfer one CV bandwidth from a curve $\X1$ to another because  $\HT$ is not known, and could even vary with $\T$. 

The data are generated from the model \eqref{model_eq} using different settings for $X$, the distribution of $T$ and the variance of the noise, as well as for $\N0$ and $\N1$. For $X$, we consider three types of Gaussian processes: fractional Brownian motion (fBm) with constant Hurst parameter $H\in(0,1)$, fBm with piecewise constant Hurst parameter, and integrated fBm. In the later case, $X_t = \int_0^tW_H(s)ds$, where $W_H$ denotes a fBm with constant Hurst parameter $H$. The local regularity is constant for the first and the third type, and variable for the second. The third type is an example of $X$ with smooth trajectories. We identify the setting for $X$ by $s\in\{1,2,3\}$. A more detailed description of these processes, as well as plots of their trajectories, are provided in the Appendix. The number $M$ of measuring times of a curve is  a Poisson random variable  with expectation $\mu$, while for the measuring times $T$, we considered either a uniform distribution (identified by \texttt{unif}), or a deterministic  equispaced grid (\texttt{equi}) on the range $[0, 1]$. For the noise, we considered the Gaussian distribution with both constant and variable variance. The cases are identified by $\sigma^2$ which could be a number or a list, respectively. The values of $\sigma^2$ are chosen in such a way that the variance ratio signal-to-noise remains almost unchanged. 
Thus, one simulation setting  is defined by the $7$-tuple $(s, \N0, \N1, \mu, f, H, \sigma^2)$, with $f\in\{ \texttt{unif}, \texttt{equi} \}$ and $H$, the Hurst parameter, is a list in the case of fBm with piecewise constant local regularity. Below, we present the results for a few settings, complementary results are reported in the Appendix.  For each type of experiment, the reported results are obtained from 500 replications of the experiment. Our methodology is implemented in the \textbf{\textsf{R}} package \texttt{denoisr} available at {\small \url{https://github.com/StevenGolovkine/denoisr}}.

Figure \ref{fig:set2H0} presents the results for the local regularity estimation for piecewise fbM with homoscedastic noise.
The local estimations of $\HT$ are performed at $\T = 1/6, 1/2$ and $5/6$ which correspond to the middle of the interval for each regularity. The true values of $\HT$ are 0.4, 0.5 and 0.7, respectively. The results show a quite accurate estimator $\hHT$ and confirm the theoretical result on its concentration. 
Increasing either $\mu$ or $\N0$ improves the concentration. The results for  \texttt{unif} and \texttt{equi} are quite similar. 
Figure \ref{fig:set3_H0} presents the estimation of $\ST$ for different settings $(3, \N0, 500, \mu, \texttt{equi}, 1.7, 0.005)$. As expected, our local regularity estimation approach also performs well for smooth trajectories. 

\begin{figure}
    \centering
    \includegraphics[scale=0.5]{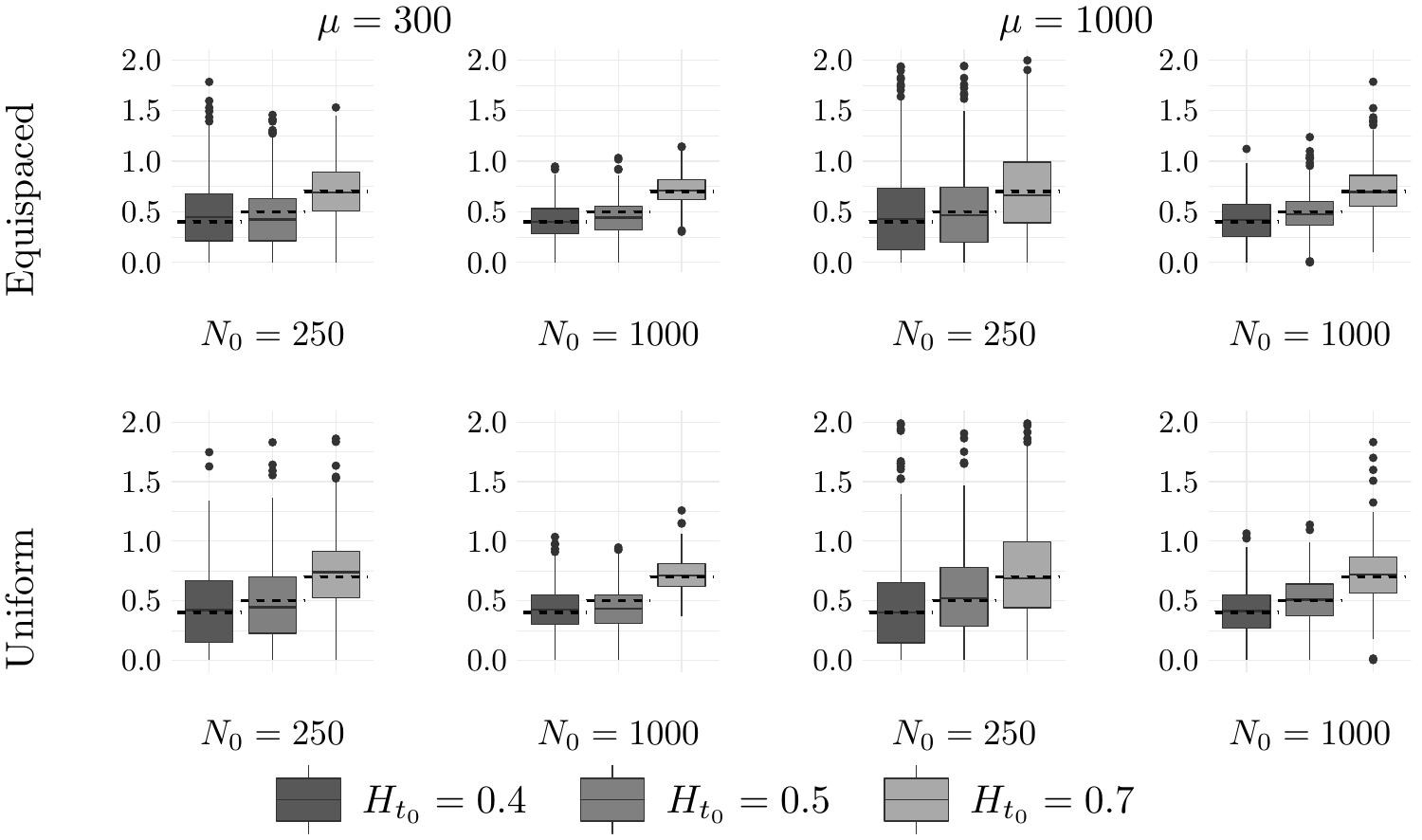}
    \caption{\small Estimation of the local regularity for piecewise fBm, with constant noise variance $\sigma^2=0.05$, at $\T=1/6, 1/2$ and $5/6$. True values: $\ST=\HT$ equal to $0.4, 0.5$ and $ 0.7$, respectively.}
    \label{fig:set2H0}
\end{figure}

\begin{figure}
\centering
\includegraphics[scale=0.425]{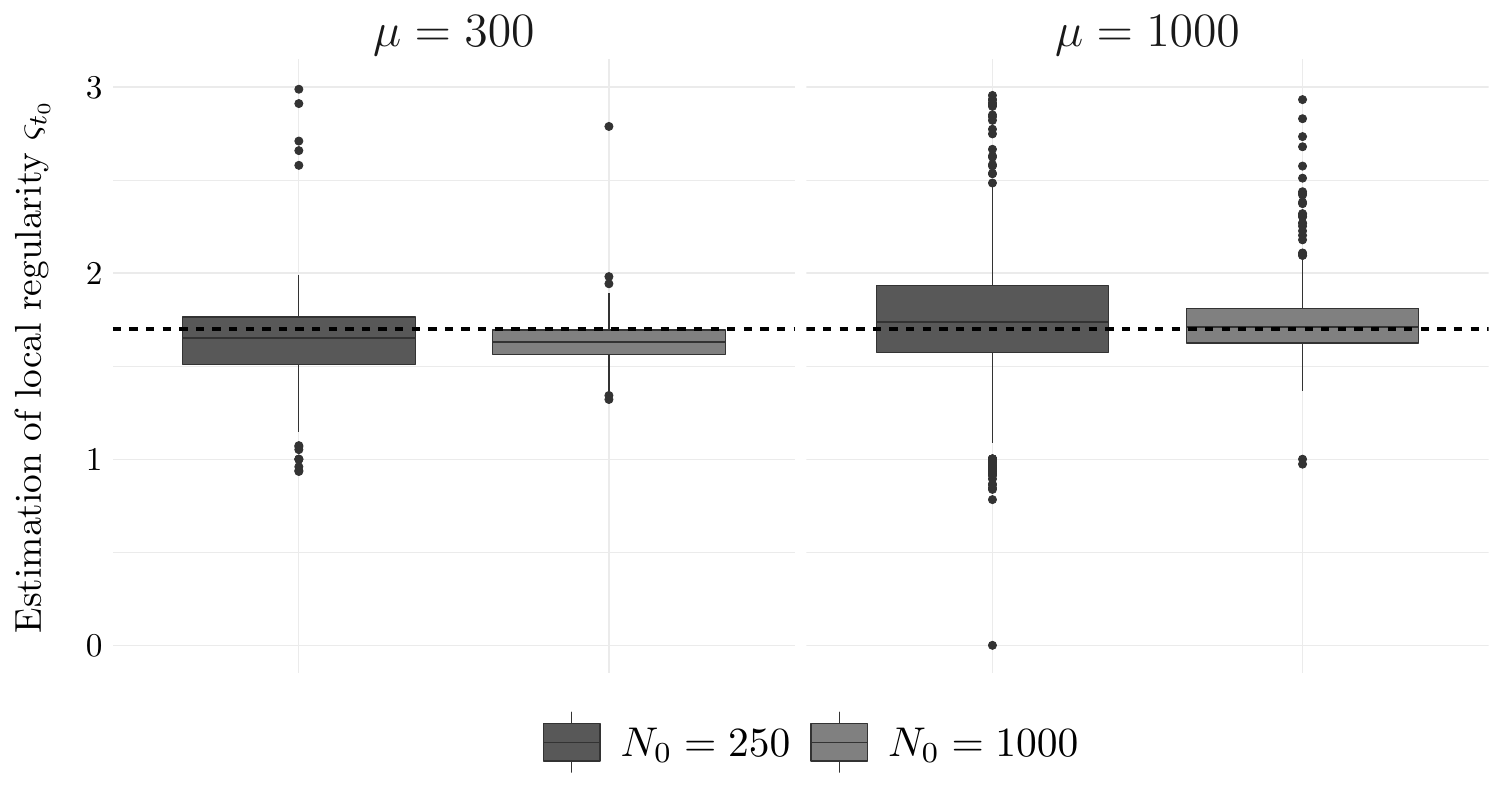}
\caption{\small Estimation of the local regularity for integrated fBm, with constant noise variance $\sigma^2=0.005$, at $\T=0.5$. True value: $\varsigma_\T=1.7$.}
\label{fig:set3_H0}
\end{figure}

Next, we present the results on the risk $\mathcal R (\widehat X;\T)$. Figure \ref{fig:set2maxrisk} presents the boxplots of the risk $\mathcal R (\widehat X;\T)$ defined in \eqref{risk_max} in the case of piecewise constant local regularity, with three values of $\T$, each one in the middle of the interval of the changes of regularity are defined. 
The results are quite good. Part of the curves with lower regularity are harder to estimate and thus results in higher risks than the more regular parts.
It appears that $N_0$ and $\mu$ do not have the same influence on the risk as the estimation of the local regularity, and this is in line with the risk bound in Theorem \ref{thm:local-poly}. Thus, going from $300$ to $1000$ sampling points leads to large improvement in terms of risk whereas going from $250$ to $1000$ curves in the \emph{learning} dataset only results in little or no improvement. 
Finally, it seems that the method achieves better results for equispaced sampling points.
\begin{figure}
    \centering
    \includegraphics[scale=0.5]{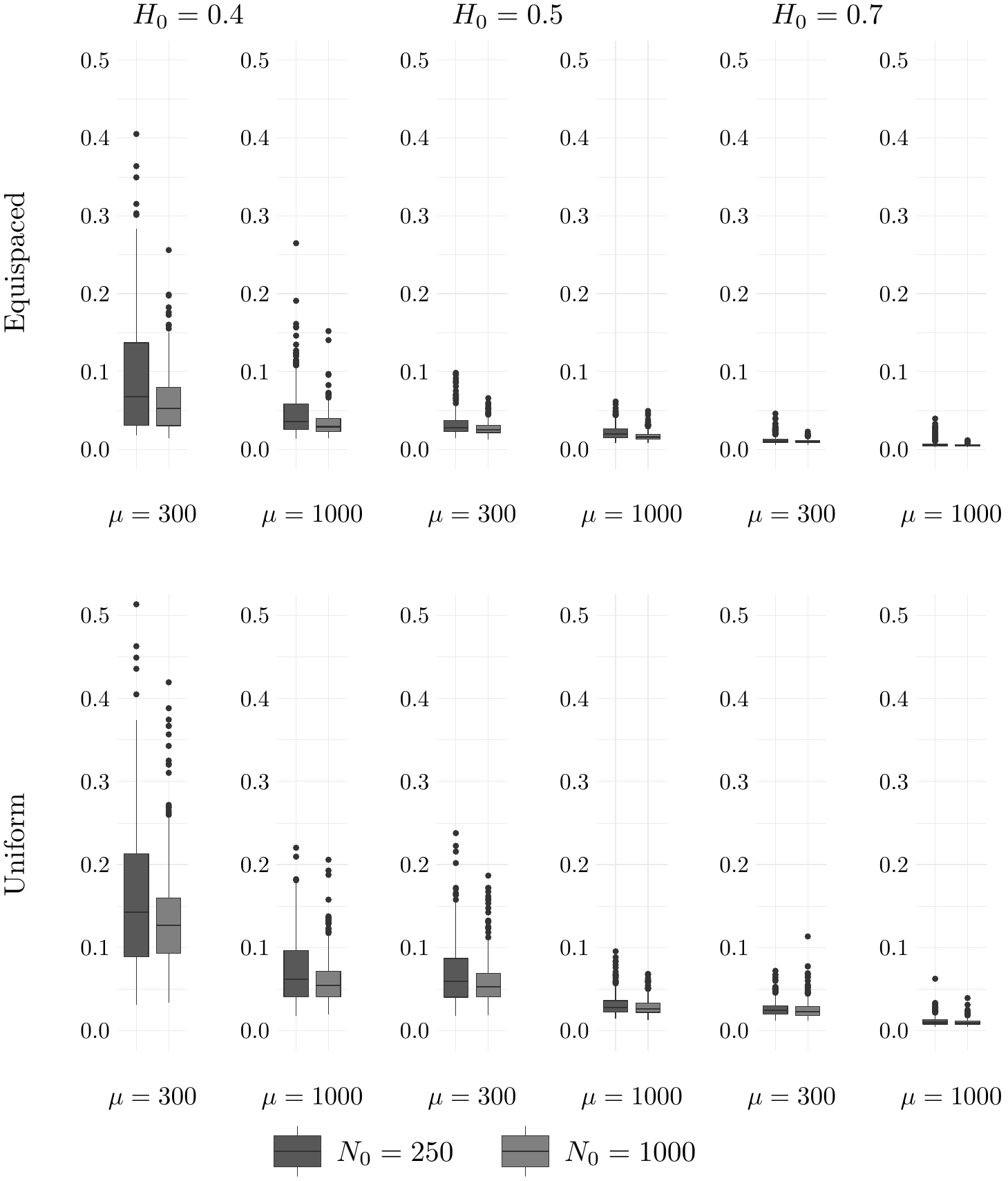}
    \caption{\small Estimation of the risks $\mathcal R (\widehat X;1/6)$, $\mathcal R (\widehat X;0.5)$ and $\mathcal R (\widehat X;5/6)$
     for piecewise fBm, with constant noise variance $\sigma^2=0.05$.}
    \label{fig:set2maxrisk}
\end{figure}

The same conclusions could be drawn from the results presented in Figure \ref{fig:set3_risk}, obtained for the simulation experiment defined by the $7$-tuple $(3, 1000, 500, 1000$, $\texttt{equi}, 1.7, 0.005)$. 
\begin{figure}
\centering
\includegraphics[scale=0.4]{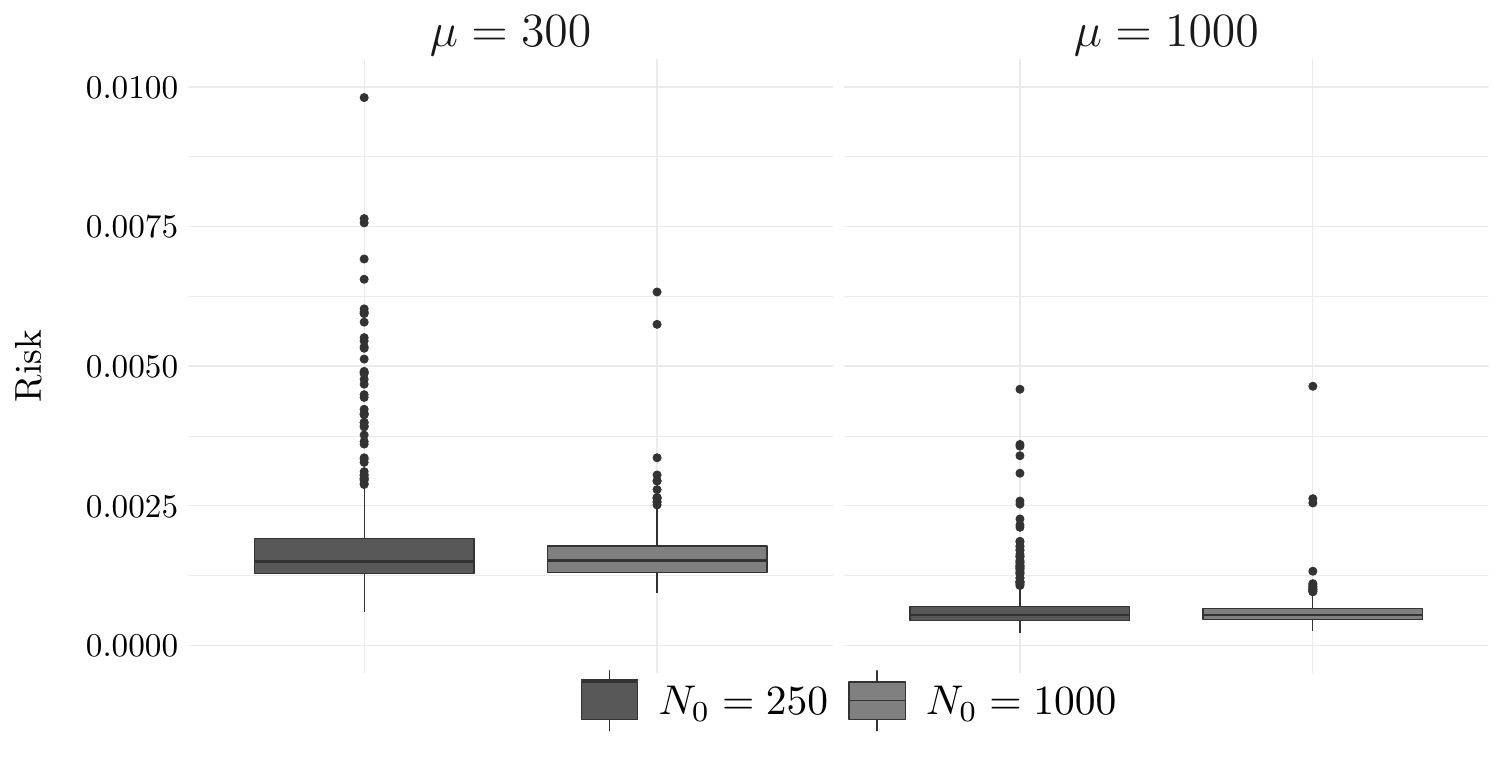}
\caption{Estimation of the risk $\mathcal R (\widehat X;0.5)$ for smoothing the noisy trajectories of an integrated fBm, with constant noise variance $\sigma^2 = 0.005$.}
\label{fig:set3_risk}
\end{figure}

Finally, we present a comparison with the CV. Because of the large amount of computing resources required by CV, we only considered a few cases.
Figure \ref{fig:set1risk} presents the results in terms of the risk  calculated at $t_0 = 0.5$ for the setting $(1, 1000, 500, 300, \texttt{equi}, 0.5, 0.05)$. We make the remark that our method and CV perform similarly despite the fact that CV uses a specifically tailored bandwidth for each curve in the \emph{online} set. The homoscedastic setting is favorable to CV which, for a given curve, uses a global bandwidth at any $\T$. Figure \ref{fig:set2comparison} presents the  the heteroscedastic setting $(2, 1000, 500, 1000, \text{equi}, (0.4, 0.5, 0.7), 0.05)$. CV preserves good performances when the local regularity varies moderately. Our method shows close performance in this case, slightly better when $\HT=0.7$. Figure \ref{fig:set3risk} presents the results in the setting $(3, 1000, 500, 1000, \texttt{equi}, 1.7, 0.005)$. Again, CV and our method perform quite similarly. 
\begin{figure}
    \centering
    \includegraphics[scale=0.33]{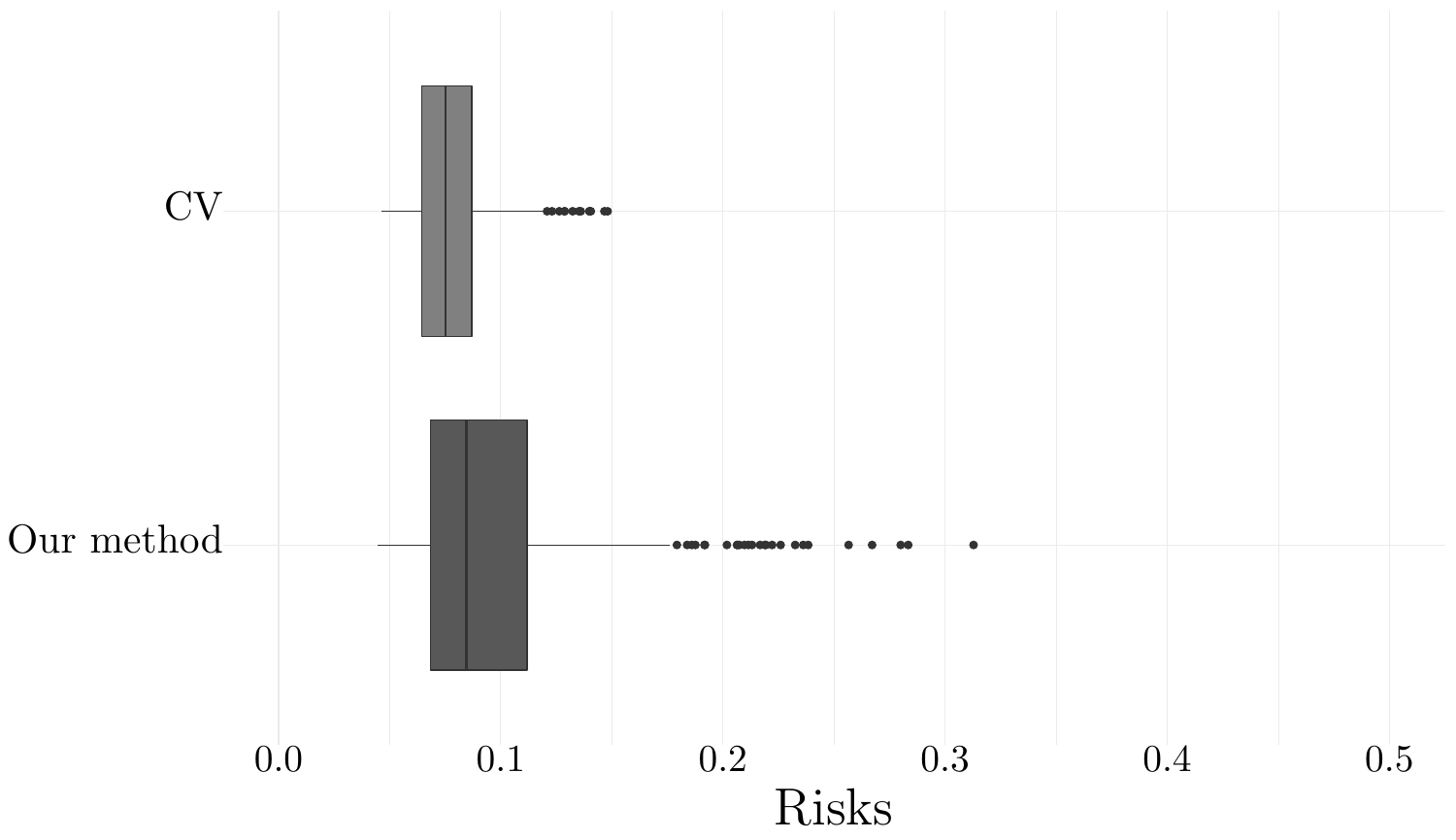}
    \caption{\small CV versus our method: comparing the pointwise risk $\mathcal R (\widehat X;0.5)$ for smoothing the noisy trajectories of a fBm; simulation $(1, 1000, 500, 300, \texttt{equi}, 0.5, 0.05)$.}
    \label{fig:set1risk}
\end{figure}

\begin{figure}
	\centering
	\includegraphics[scale=0.4]{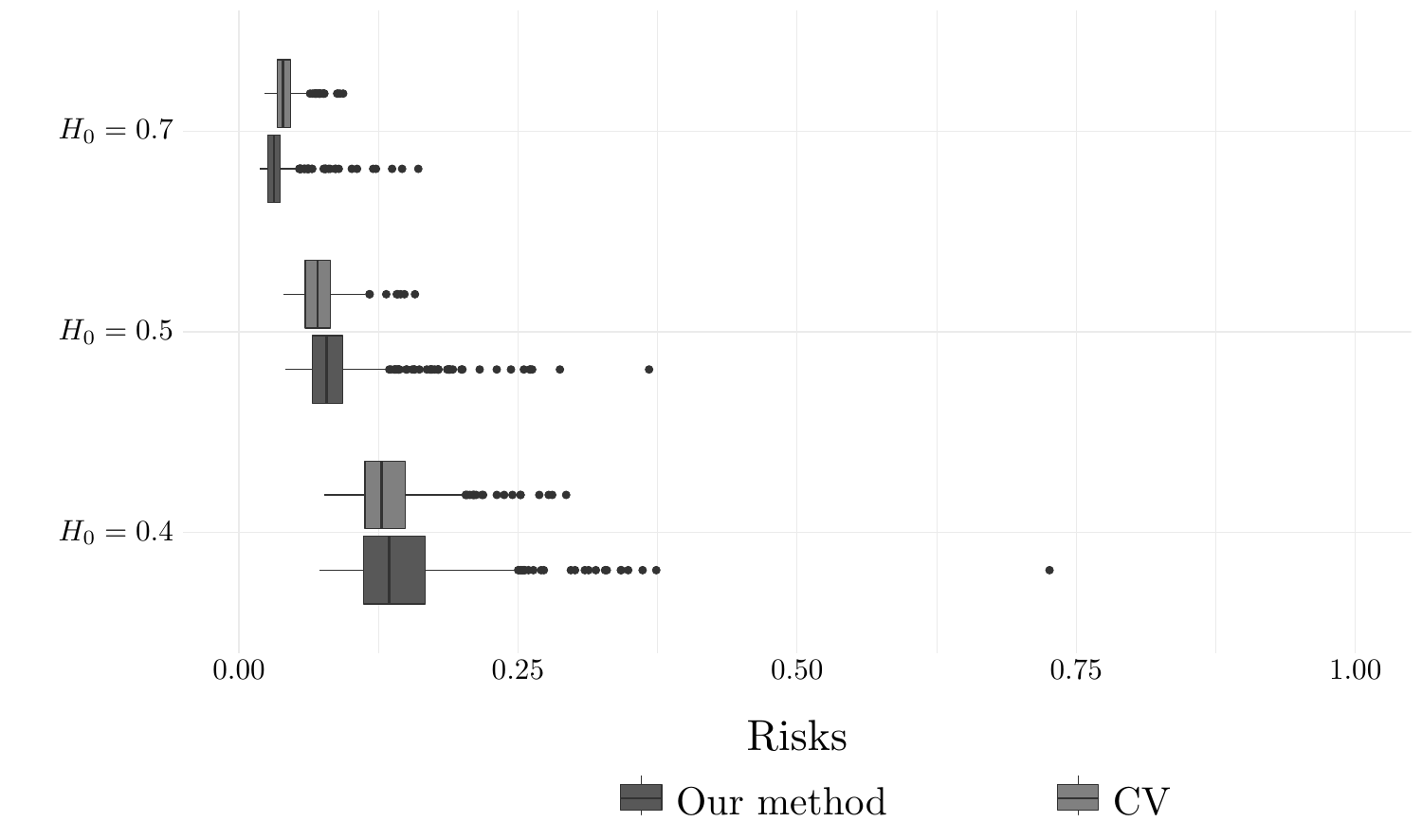}
	\caption{\small CV versus our method: comparing  $\mathcal R (\widehat X;1/6)$, $\mathcal R (\widehat X;0.5)$ and $\mathcal R (\widehat X;5/6)$ for smoothing the noisy trajectories of a piecewise fBm; simulation  $(2, 1000, 500, 1000, \text{equi}, (0.4, 0.5, 0.7), 0.05)$.}
	\label{fig:set2comparison}
\end{figure}

\begin{figure}
\centering
\includegraphics[scale=0.33]{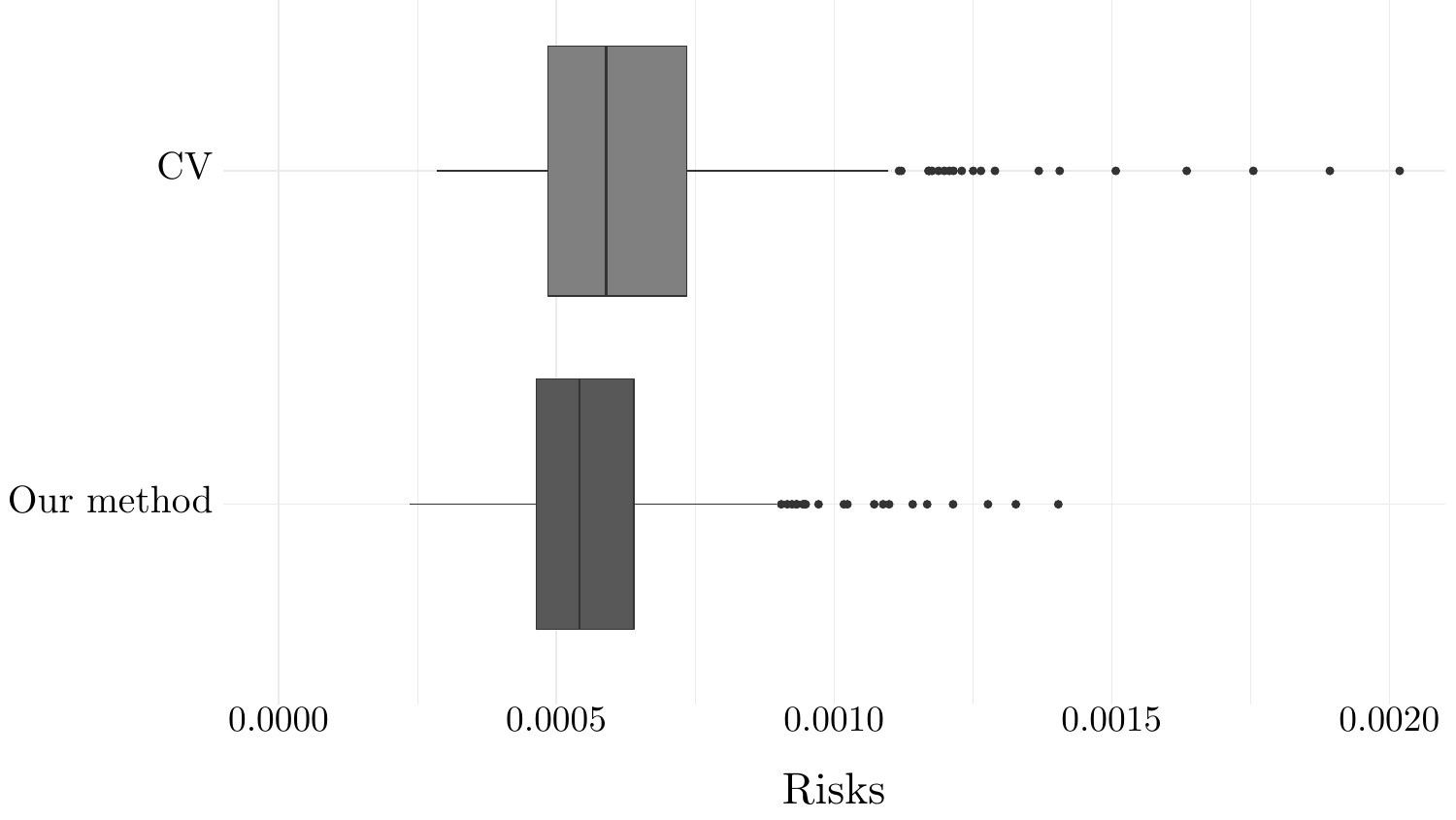}
\caption{\small CV versus our method: comparing the pointwise risk $\mathcal R (\widehat X;0.5)$ for smoothing the noisy trajectories of an integrated  fBm; simulation $(3, 1000, 500, 1000, \text{equi}, 1.7, 0.005)$.}
\label{fig:set3risk}
\end{figure}

Let us end our simulation experiments presentation with the results on the median curve estimation from noisy curves.  We consider a Gaussian simulation design given by Model 4 in \cite{romo2009}, page 726. The mean function is $g(t) = 3\sin(3\pi (t+0.5))+2t^3$ and the covariance function is $\gamma(s,t) = \exp(-|t-s|^{0.8})$. The curves are measured on a common equidistant grid with $M\in\{300,1000\}$ points and the samples of size $N\in\{100,200\}$ have a proportion of 25\% contaminated observations. We used the contamination by peaks with the parameter values used by \cite{romo2009} in their Figure 6. To the simulation design used in \cite{romo2009}, we add a Gaussian measurement error with constant variance $\sigma^2$. We compute the median curve using the mode depth with the noiseless sample (the benchmark median) and with the noisy samples, without smoothing. Next we estimate the median curve using the smoothed curves. For smoothing, we use, on the one hand, 
$LP(\hKT)$ with $\hKT$ and the bandwidth delivered by Algorithm \ref{alg:reg1}, and, on the other hand, the \textbf{\textsf{R}} package \texttt{np} with the bandwidth selected by CV for each curve separately. We compare the mean integrated squared error (IMSE) of the smoothing-based median curves  with respect to the benchmark. We also compare the IMSE obtained by our approach with that obtained with the median curves computed from the noisy samples directly. Finally, we compare the computing times for our method and that using CV. The results obtained from 500 replications, with $\sigma^2=1$, are reported in Figure \ref{fig:median_1}.  The accuracy of the estimates, obtained in a small fraction of a second using our smoothing method, is slightly lower than that obtained by CV, which requires far more computation time. Ignoring the noise in the sample leads to very large errors for the median curve estimation. 

 \begin{figure}
\centering
\includegraphics[height=5.7cm, width=12.5cm]{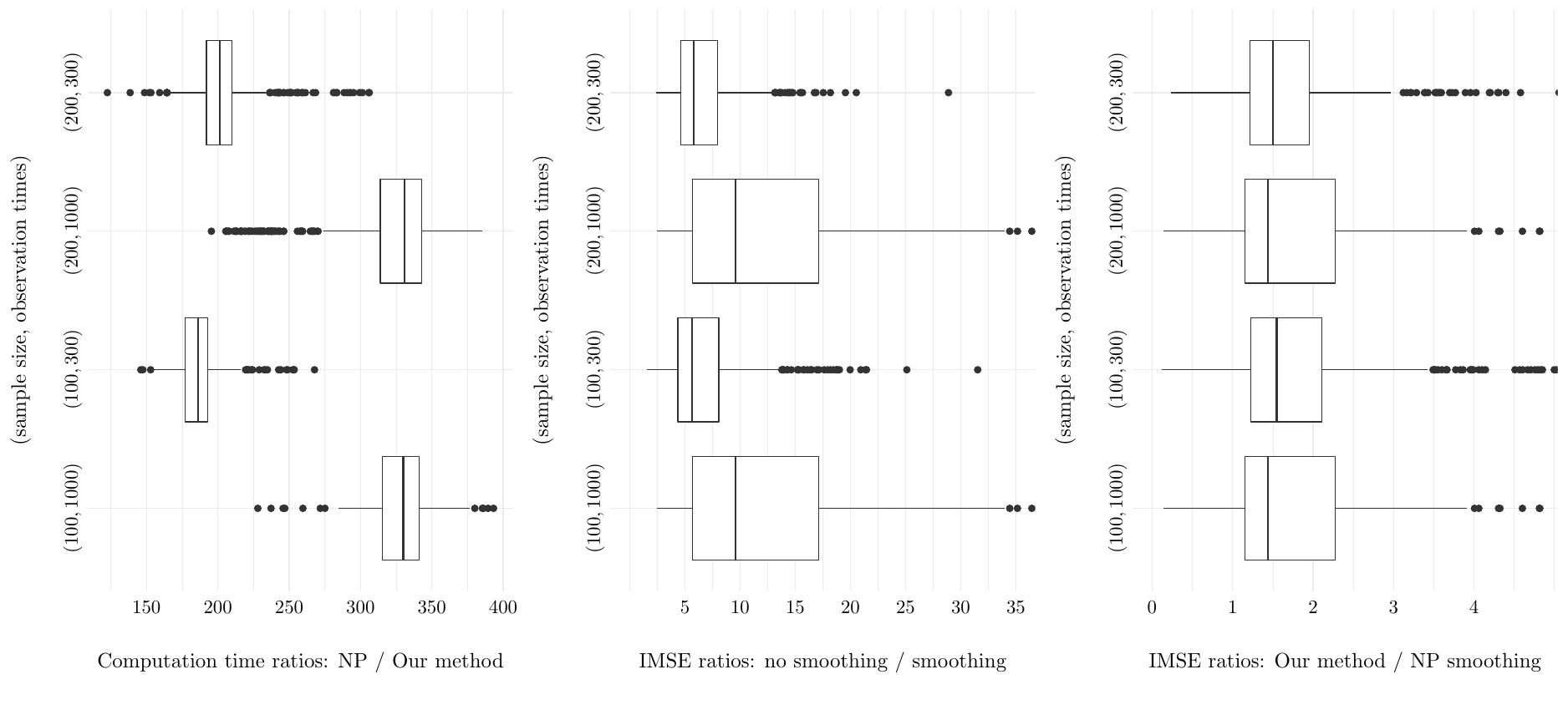}
\caption{\small Median curve estimation from samples of $N\in\{100,200\}$ noisy curves generated from Model 4 of  \cite{romo2009}, observed at $M\in\{300,1000\}$ points. The noise variance is $\sigma^2=1$. Plotted ratios:  computation times with  CV and our method (left); IMSE for the median curve without smoothing and with smoothing using method (middle); IMSE with our method and CV (right).}
\label{fig:median_1}
\end{figure}

\subsection{Real data analysis: the NGSIM Study}

In this section, our method is applied to data from the Next Generation Simulation (NGSIM) study, which aims to ``describe the interactions of multimodal travelers, vehicles and highway systems'', see \cite{halkias_ngsim_2006}. This study is known to be one of the largest publicly available source of naturalistic driving data. This dataset is widely used in traffic flow studies from the interpretation of traffic phenomena such as congestion to the validation of models for trajectories prediction (see \emph{e.g.} \cite{dong_intention_2017,hu_framework_2018,wulfe_real-time_2018,mercat_inertial_2019,henaff_model-predictive_2019} for some recent references). However, such data have been proved to be subject to measurement errors revealed by physical inconsistency between the space traveled, velocity and acceleration of the vehicles, \emph{cf.} \cite{punzo_estimation_2009}. Montanino and Punzo \cite{montanino_making_2013} developed a  trajectory-by-trajectory four-steps method to recover the signals from the noisy curves, and their methodology is now considered as a benchmark in the traffic flow engineering community for analyzing NGSIM data. The steps, finely tuned for the NGSIM data,  are~: 1. removing the outliers; 2. cutting off the high- and medium-frequency responses in the speed profile; 3. removing the residual nonphysical acceleration values, preserving the consistency requirements; 4. cutting off the high- and medium-frequency responses generated from step 3. The detailed description of these steps is  provided in the Appendix. 

To compare our smoothed curves to those of \cite{montanino_making_2013},
we consider the following ratio:
\begin{equation}\label{eq:ratio_monta}
r (\widehat{X}, \widetilde{X}) = \frac{\sum_{m = 1}^{M_n}\left[Y^{(n)}_m- \widehat{X}(T_m^{(n)}) \right]^2}{\sum_{m = 1}^{M_n}\left[Y^{(n)}_m- \widetilde{X}(T_m^{(n)}) \right]^2}, \quad 1 \leq n \leq 1714,
\end{equation}
where
$\widehat{X}$ denotes our curve estimation while  $\widetilde{X}$ is that obtained by \cite{montanino_making_2013}. 
A value of the ratio 
$r (\widehat{X}, \widetilde{X}) $
less than $1$ indicates smoothed values closer to the observations.

For our illustration, we consider a subset of the NGSIM dataset, known as the I-80 dataset. It contains $45$ minutes of trajectories for vehicles on the Interstate 80 Freeway in Emeryville, California, segmented into three 15-minute periods (from 4:00 p.m. to 4:15 p.m.; from 5:00 p.m. to 5:15 p.m. and from 5:15 p.m. to 5:30 p.m.) on April 13, 2005 and corresponds to different traffic conditions (congested, transition between uncongested and congested and fully congested). In total, the dataset contains trajectories, velocities and accelerations for  $\N0=1714$ individual vehicles that passed through this highway during this period, recorded every $0.1$\si{\second}. The number $M_n$ of measurements for each curve varies from 165 to 946. We focus on the velocity variable and rescale the measurement times for each of the 1714 velocity  curves such that  the first velocity measurement  corresponds to $t = 0$ and the last one to $t = 1$. Figure \ref{fig:data_desc} presents a sample of five curves from this data. It can easily be noticed that the velocities are quite erratic and their variation is not physically realistic, indicating the presence of a noise. Moreover, the data have been recorded at a moment of the day when traffic is evolving, it goes from fluid to dense traffic. Therefore, we consider that there are three groups in the data: a first group corresponding to a  fluid (high-speed) traffic, a second one for in-between fluid and dense  traffic, and a third groups corresponding to the dense (low-speed) traffic. To determine the three clusters,  we fit a finite Gaussian mixture model to the vector of number of sampling points. The model is estimated by an EM algorithm initialized by hierarchical model-based agglomerative clustering as proposed by Fraley and Raftery \cite{fraley_model_2002} and implemented in the \textbf{\textsf{R}} package \texttt{mclust} \cite{scrucca_mclust_2016}. The optimal model is then selected according to BIC. The three resulting classes have 239, 869 and 606 velocity trajectories, respectively. Plots of randomly selected subsamples of trajectories from each groups are provided in the Appendix. 
The respective numbers of measures $M_n$  are plotted in Figure \ref{fig:data_desc_b}. The mean estimates $\widehat \mu$ obtained in the three groups are $218, 474$ and $684$, respectively, and the corresponding values $\widehat K_0$, as defined as in Corollary \ref{cor:enfin}, are $13, 17$ and $20$. 

\begin{figure}
    \centering
    \includegraphics[scale=0.37]{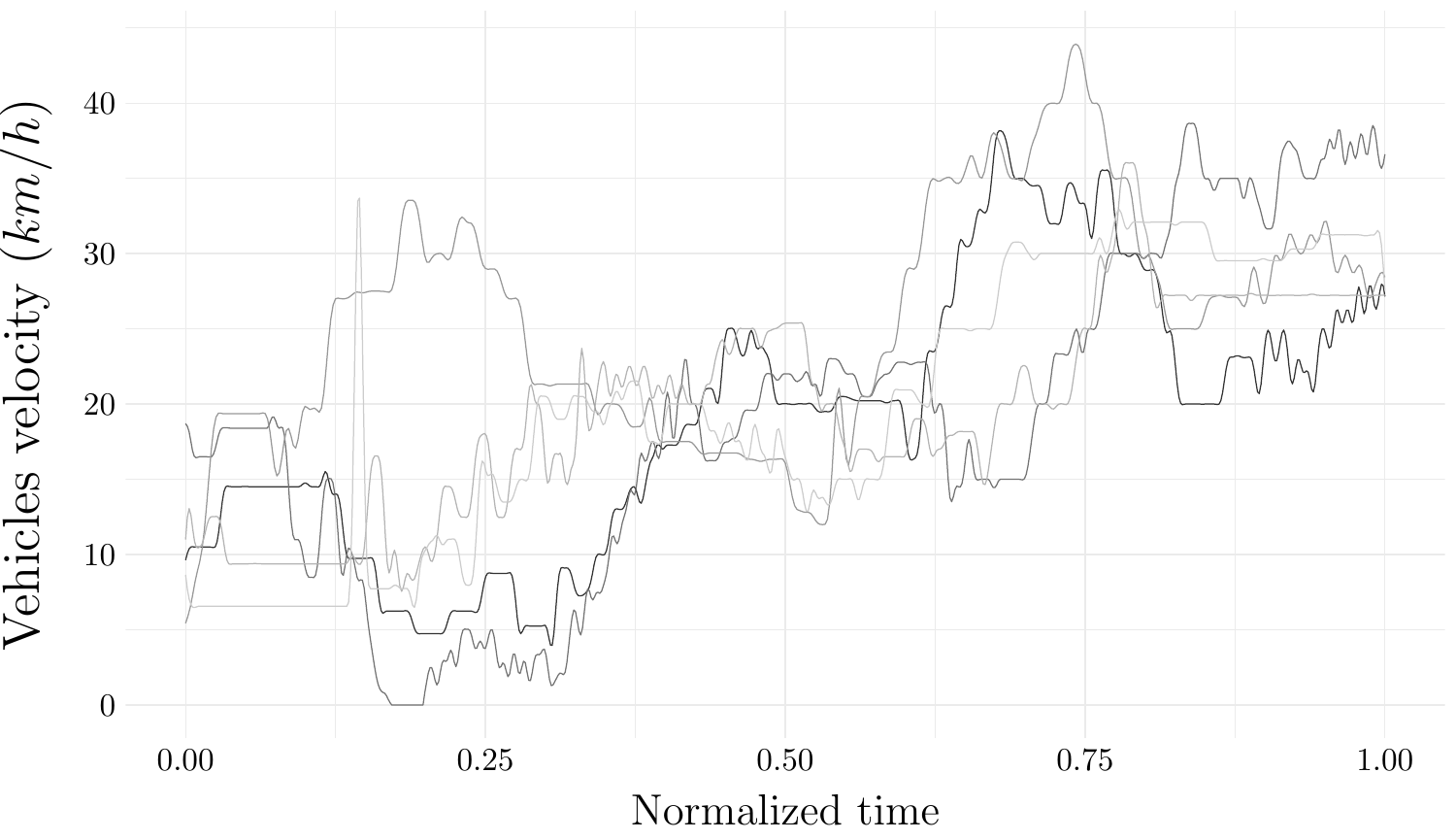}
    \caption{I-80 dataset illustration: a sample of five velocity curves.}
    \label{fig:data_desc}
\end{figure}

\begin{figure}
    \includegraphics[scale=0.37]{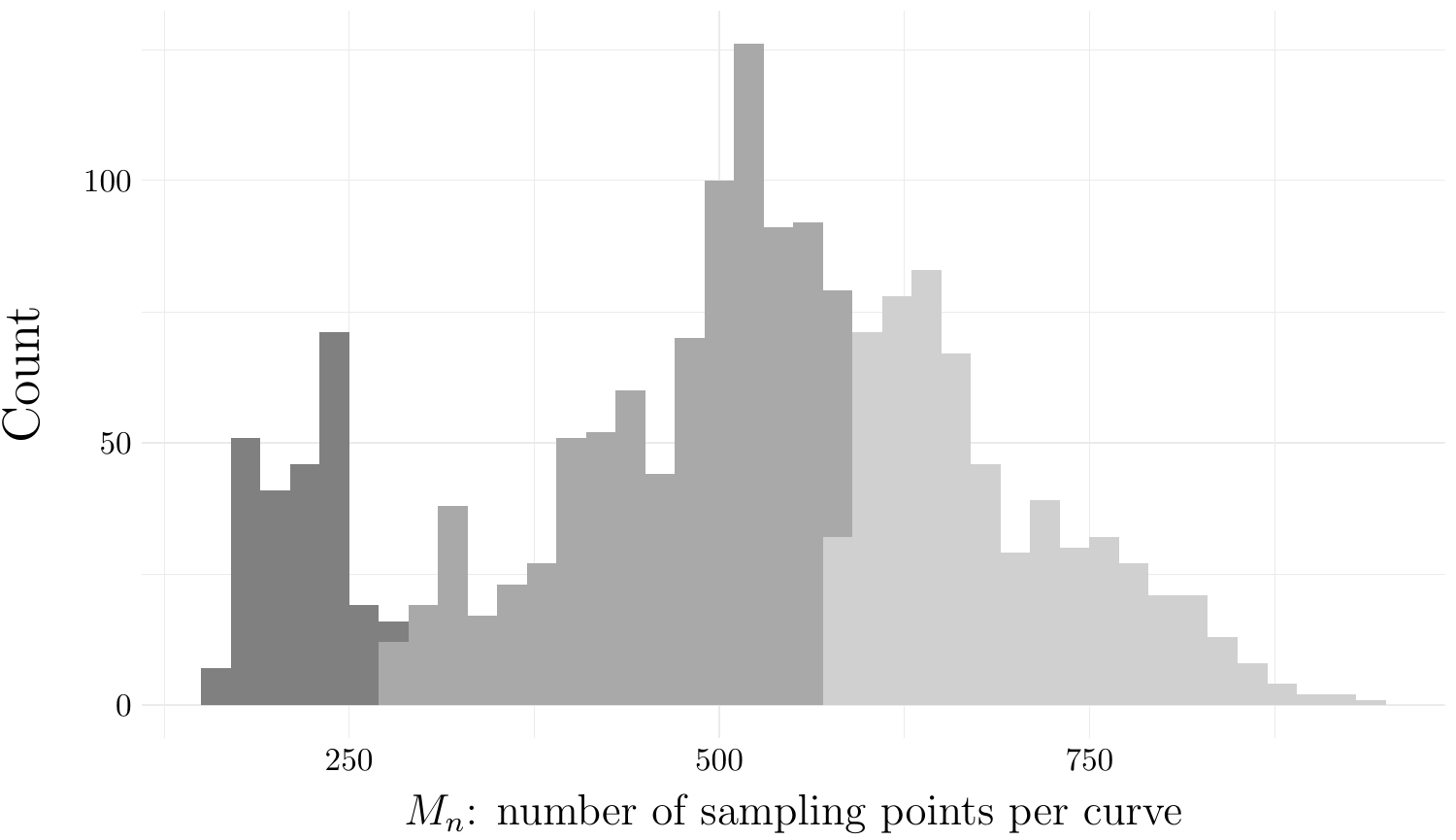}
    \caption{I-80 dataset clusters: density of sampling points for fluid (darkest gray), in-between fluid and dense, and dense traffic (lightest gray).}
    \label{fig:data_desc_b}
\end{figure}

Figure \ref{fig:H_ngsim} presents the results of the estimation of $\ST$ for values of $\T$ from $0.2$ to $0.8$, for each group. 
The evolution of $\ST$ is quite smooth, except for Group 1 (Figure \ref{fig:H_group1}).  A possible explanation could be the small number of curves and the average of $M_n$ in this group, which correspond to low values of $\N0$ and $\widehat\mu$. 
We also provide the estimation of the regularity using the whole sample of size 1714. The differences we notice between the estimates of $\ST$ from different groups support our preliminary clustering step. 

To compute the curve estimate we adopt a leave-one-curve-out procedure: each curve is smoothed using the local regularity estimates computed from the other curves in the group (or the other 1713 curves when the data is not split into groups). The densities of the resulting ratios $r (\widehat{X}, \widetilde{X}) $ are plotted in Figure \ref{fig:density_ngsim}. When the traffic is fluid and the speed is high (group 1), our method perform much better than that of Montanino and Punzo. When the traffic is dense with low speed (group 3), the smoothed values obtained with the two methods are more similar, though our method still exhibits better performance for the majority of the curves. 

\begin{figure*}
    \centering
    \begin{subfigure}[b]{0.475\textwidth}
        \centering{}
        \includegraphics[width=\textwidth]{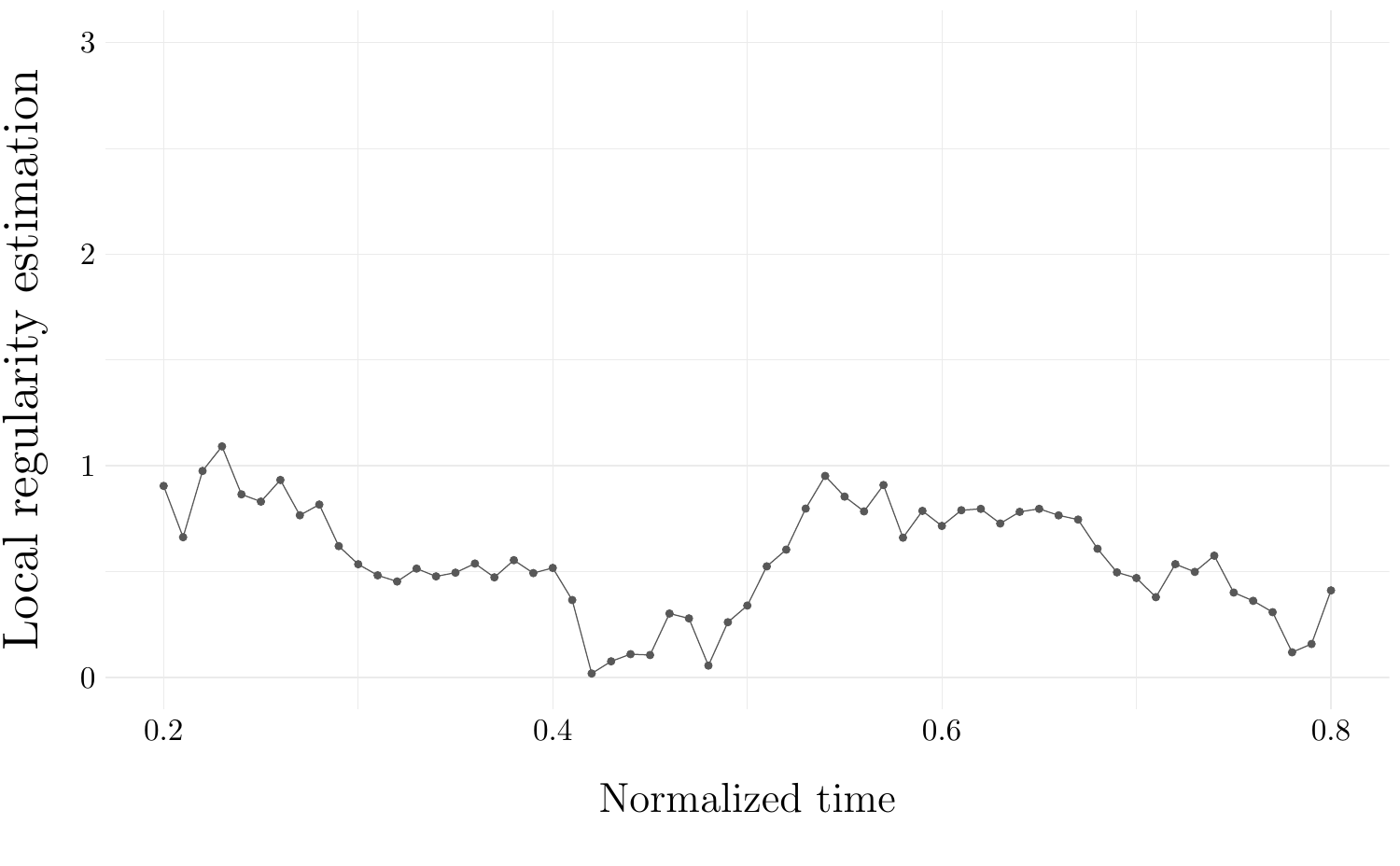}
         \vspace{-0.6cm}
        \caption{Complete sample}
        \label{fig:H_group0}
    \end{subfigure}
    \quad
    \begin{subfigure}[b]{0.475\textwidth}
        \centering
        \includegraphics[width=\textwidth]{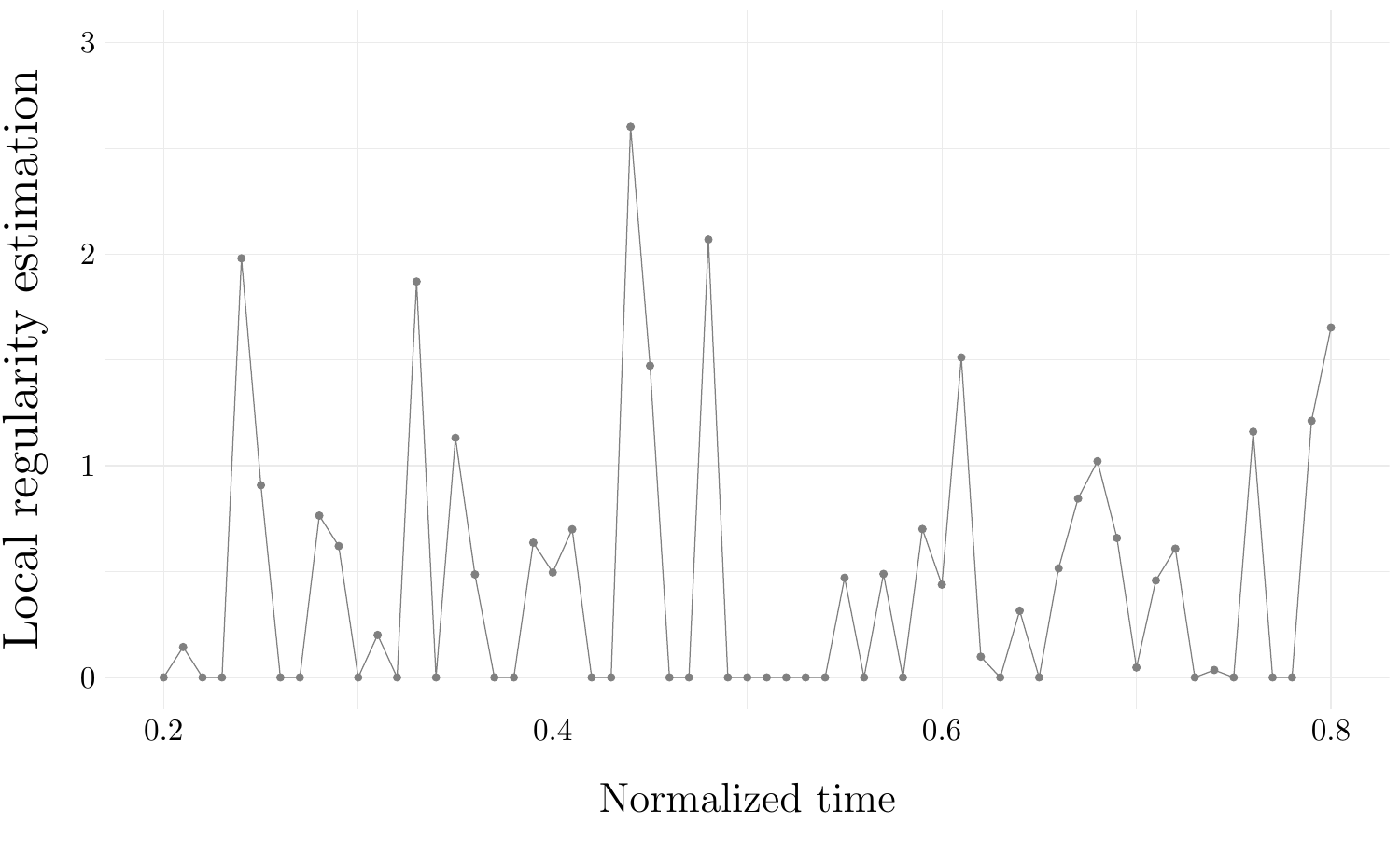}
         \vspace{-0.6cm}
        \caption{Fluid traffic/high velocity subsample}
        \label{fig:H_group1}
    \end{subfigure}
    \\
    \begin{subfigure}[b]{0.475\textwidth}
        \centering{}
        \includegraphics[width=\textwidth]{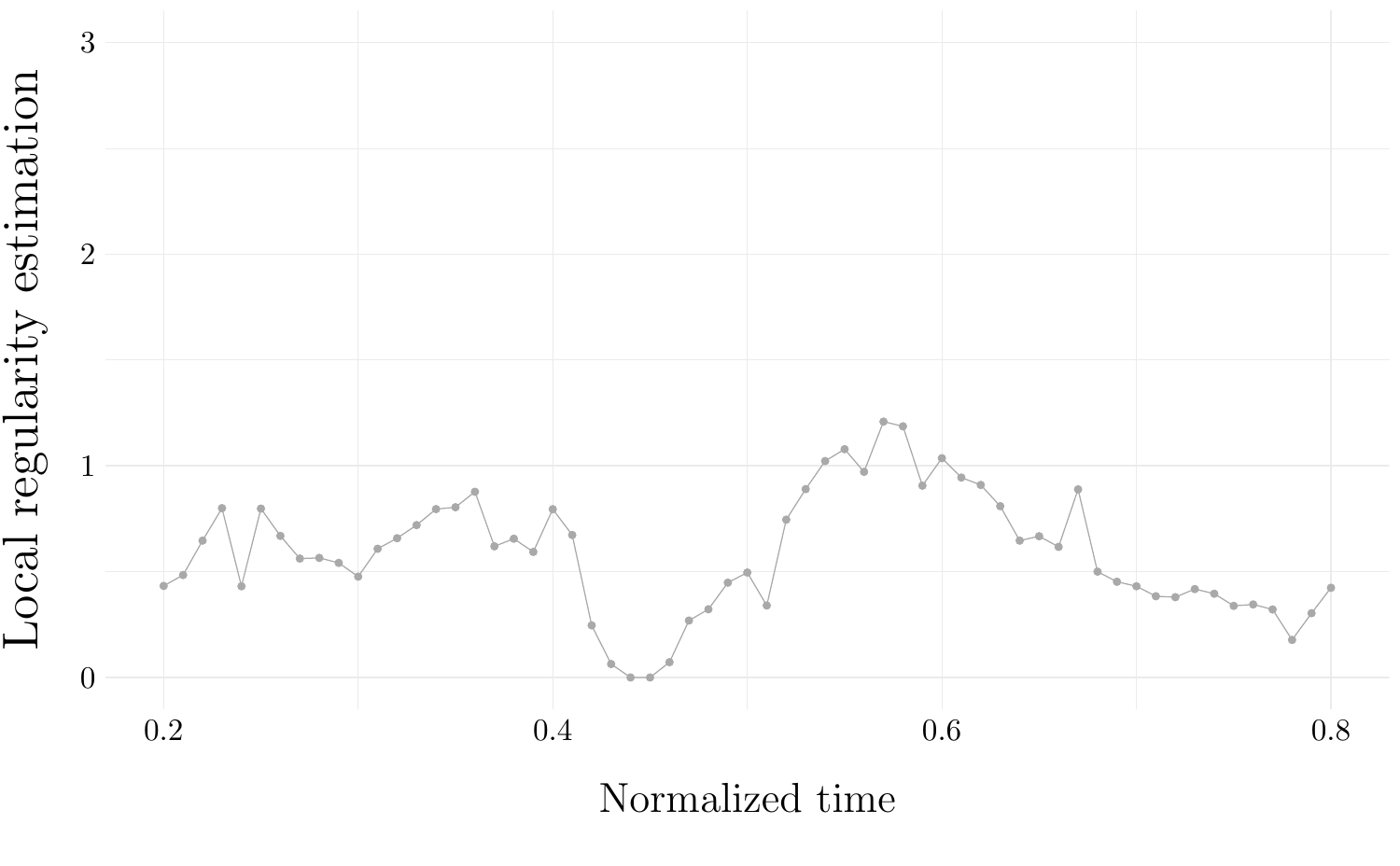}
         \vspace{-0.6cm}
        \caption{In-between group subsample}
        \label{fig:H_group2}
    \end{subfigure}
    \quad
    \begin{subfigure}[b]{0.475\textwidth}
        \centering
        \includegraphics[width=\textwidth]{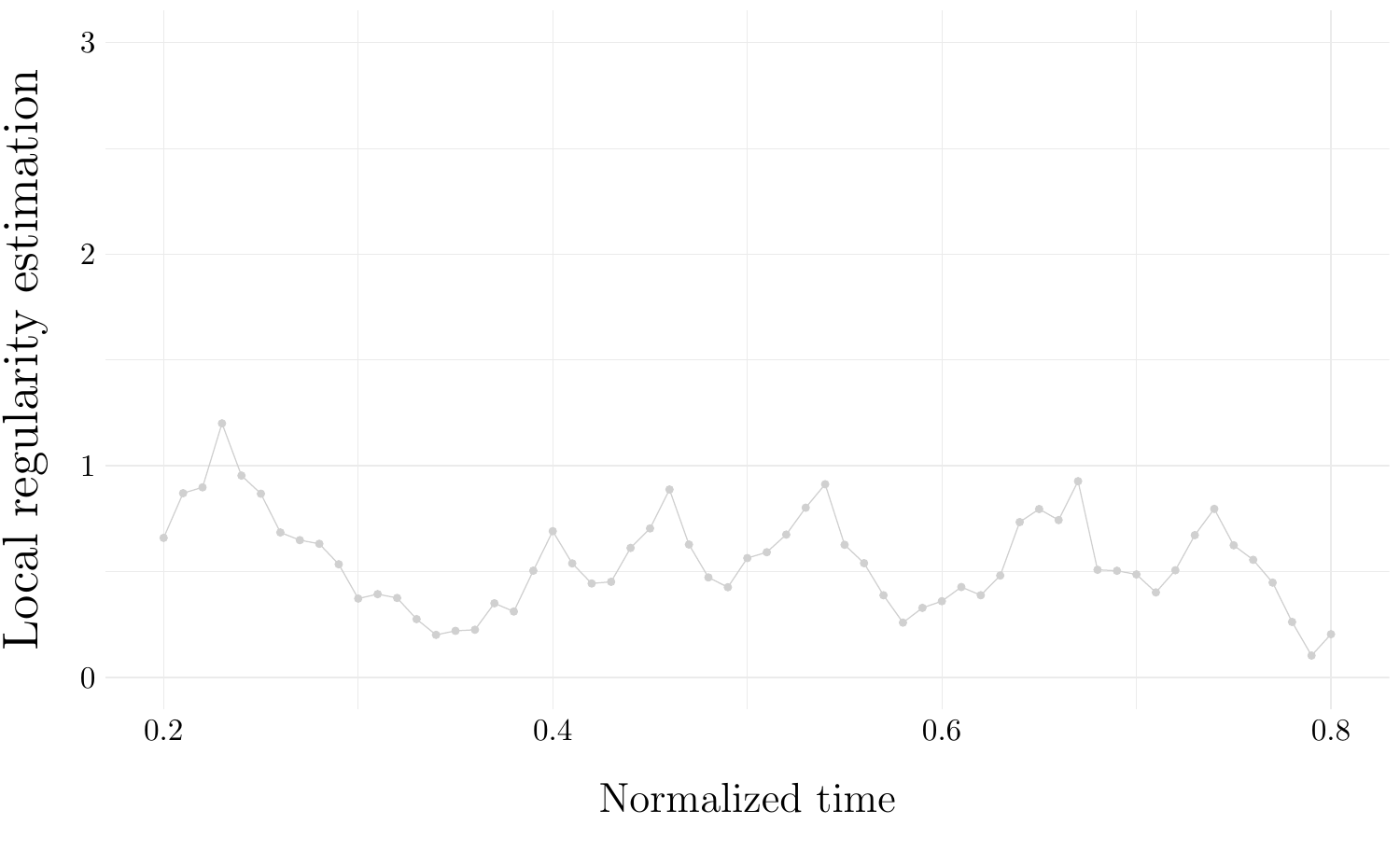}
         \vspace{-0.6cm}
        \caption{Dense traffic/low velocity subsample}
        \label{fig:H_group3}
    \end{subfigure}
    \caption{Estimation of the local regularity of the velocity curves for different $t_0$.}
    \label{fig:H_ngsim}

\end{figure*}

\begin{figure*}
    \centering
    \begin{subfigure}[b]{0.475\textwidth}
        \centering{}
        \includegraphics[width=\textwidth]{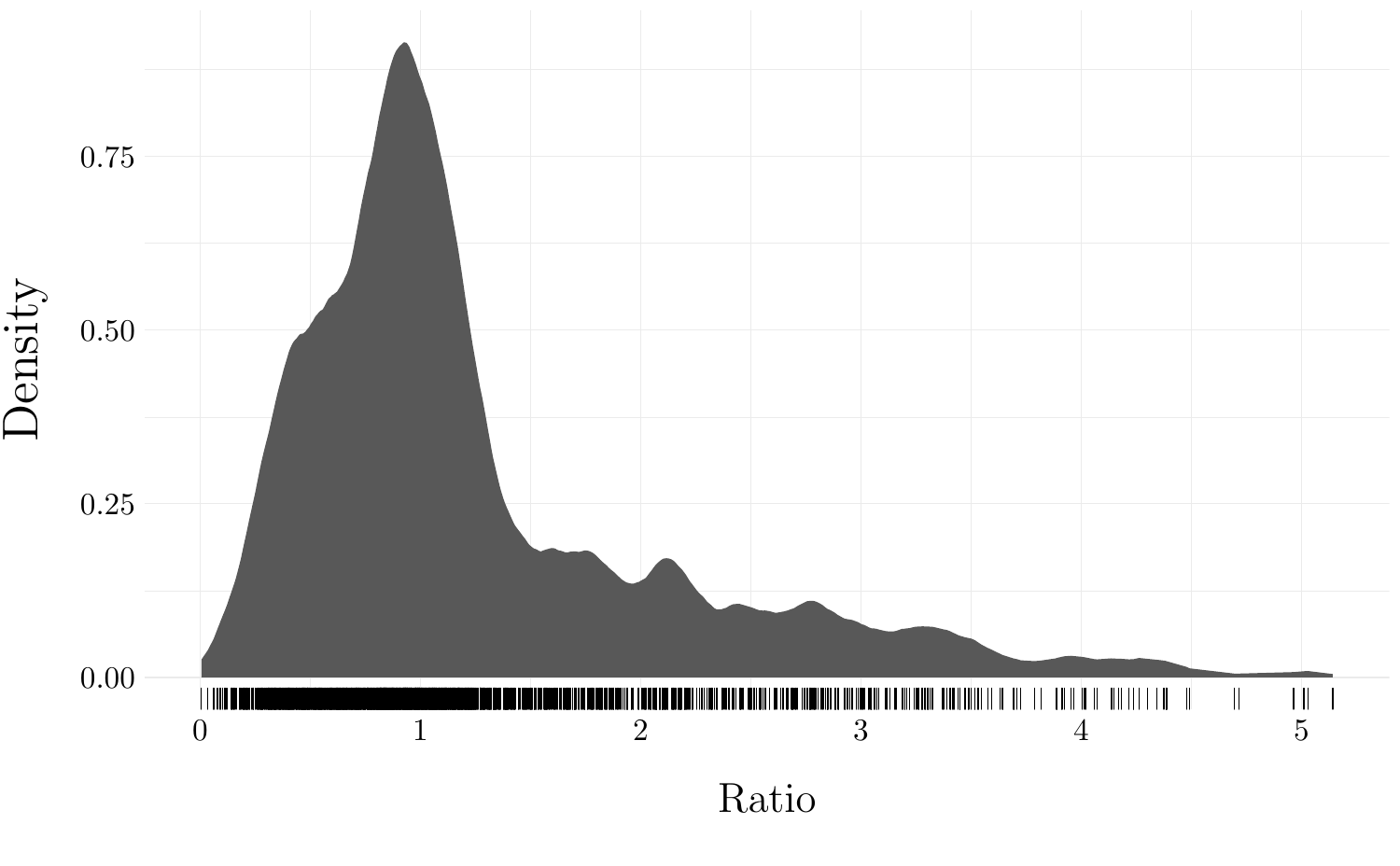}
      \vspace{-0.6cm}
        \caption{Complete sample}
        \label{fig:group0}
    \end{subfigure}
    \quad
    \begin{subfigure}[b]{0.475\textwidth}
        \centering
        \includegraphics[width=\textwidth]{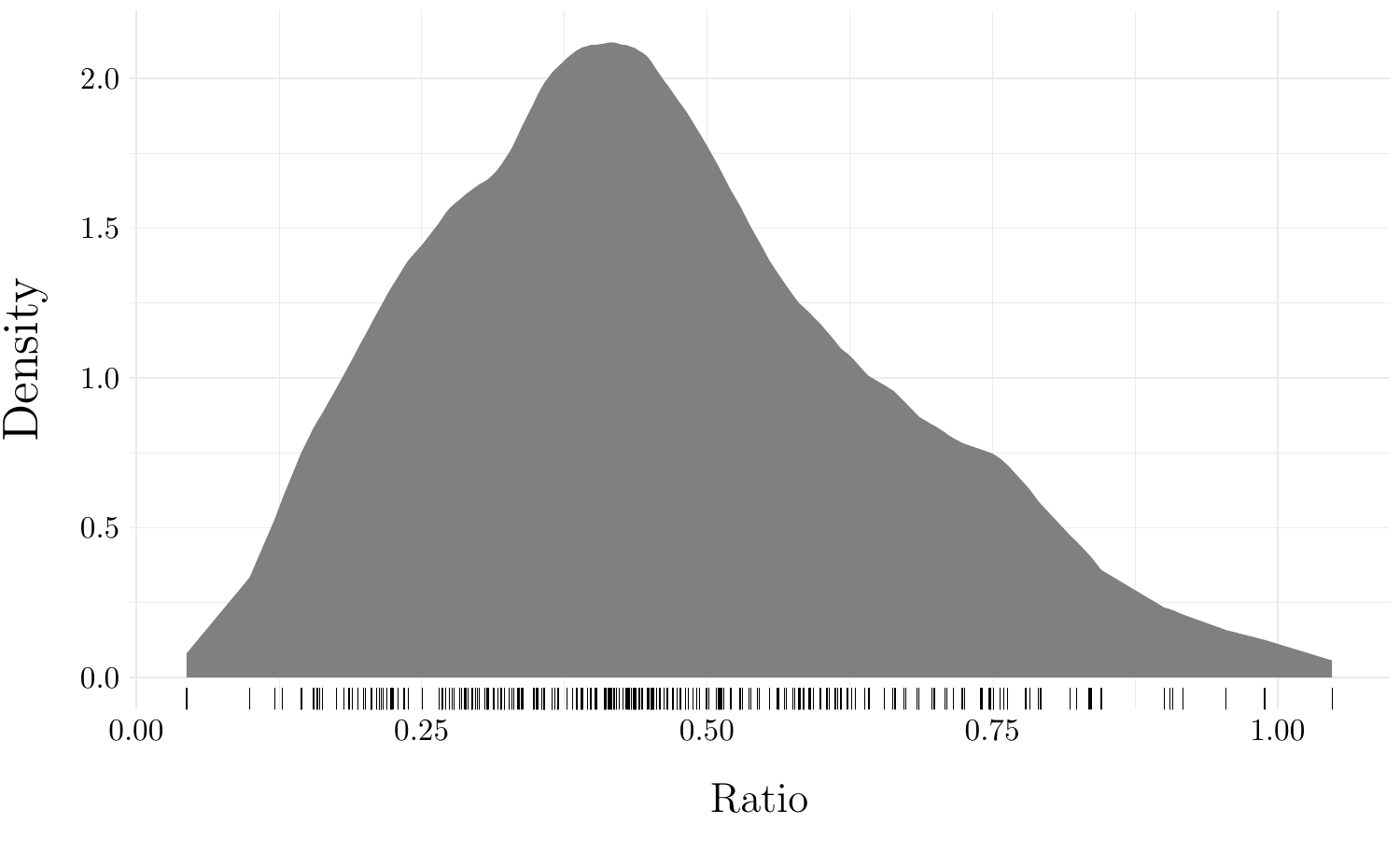}
        \vspace{-0.6cm}
        \caption{Fluid traffic/high velocity subsample}
        \label{fig:group1}
    \end{subfigure}
    \\
    \begin{subfigure}[b]{0.475\textwidth}
        \centering{}
        \includegraphics[width=\textwidth]{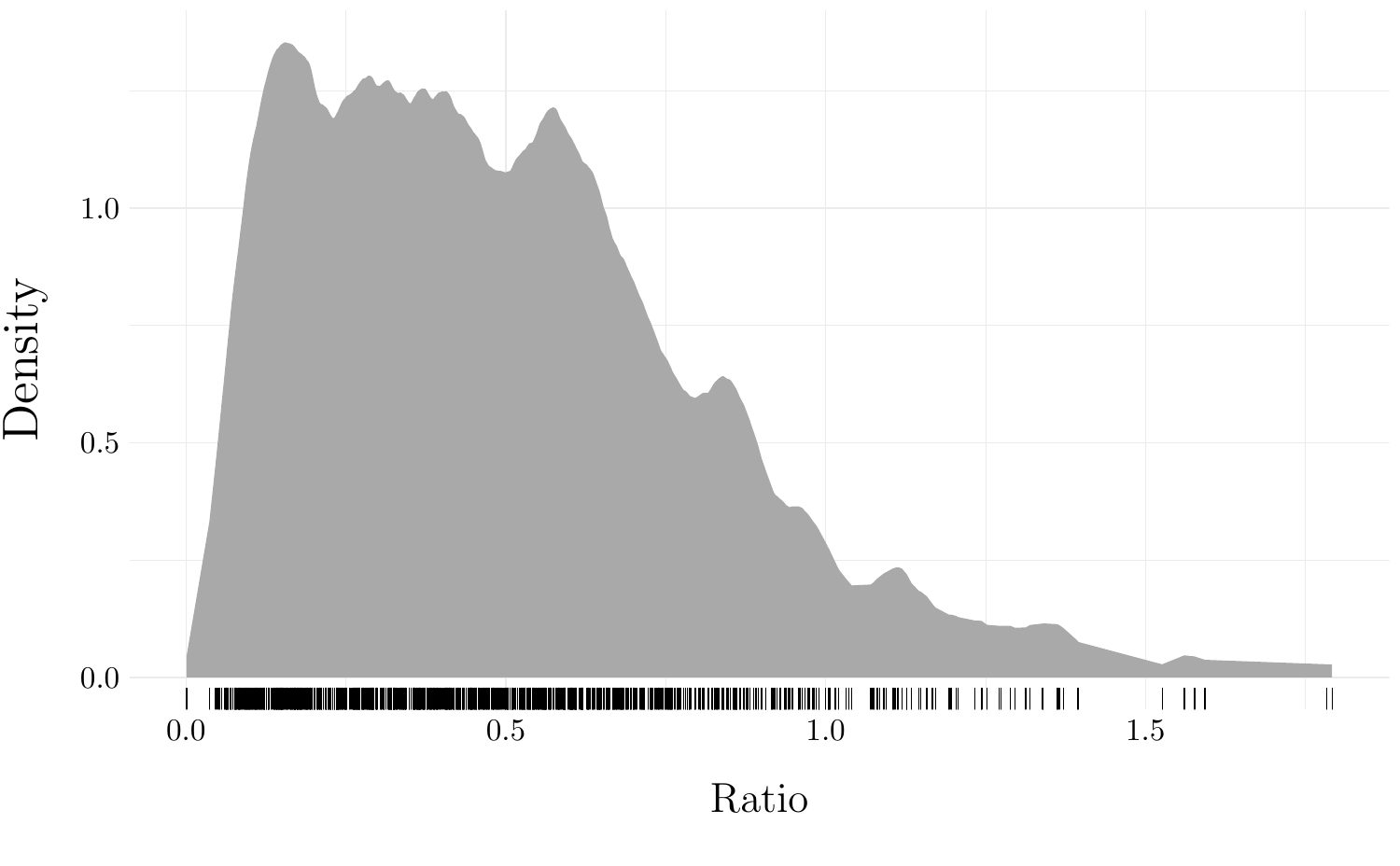}
        \vspace{-0.6cm}
        \caption{In-between group subsample}
        \label{fig:group2}
    \end{subfigure}
    \quad
    \begin{subfigure}[b]{0.475\textwidth}
        \centering
        \includegraphics[width=\textwidth]{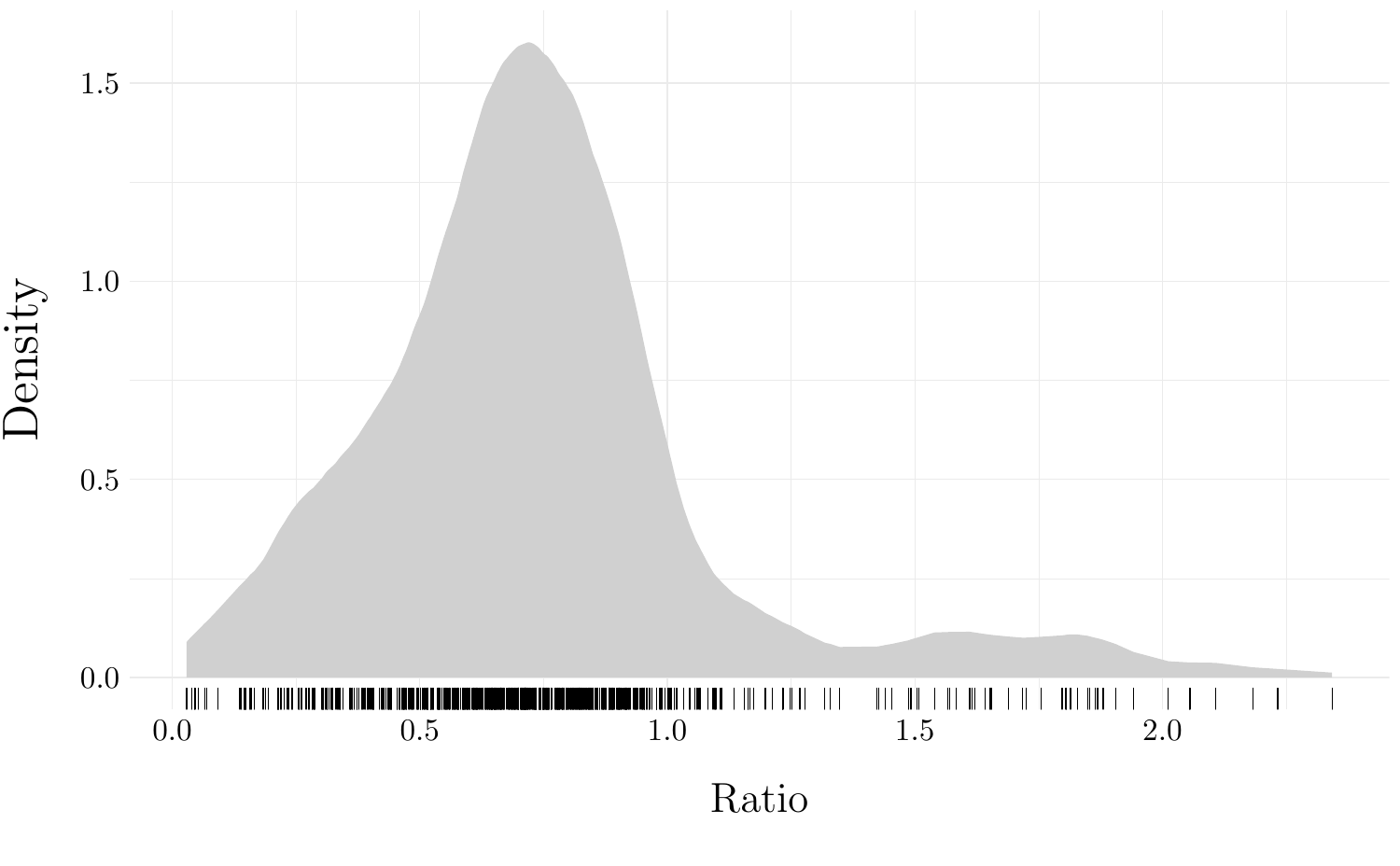}
        \vspace{-0.6cm}
        \caption{Dense traffic/low velocity subsample }
        \label{fig:group3}
    \end{subfigure}
    \caption{Densities of the ratio $r (\widehat{X}, \widetilde{X})  $ within different groups.}
    \label{fig:density_ngsim}
\end{figure*}

%% file: appendix/theorem_Ht0.tex
\section{Proof of Theorem~\ref{thm:Ht0}} \label{app:th_Ht0}

The proof of Theorem~\ref{thm:Ht0} is based on several lemmas that we present in the following. For these lemmas, we implicitly assume that the conditions of Theorem~\ref{thm:Ht0} are satisfied. 

\begin{lemma}\label{lem:Xl-Xk}
    Let $r$ be an integer such that
    \begin{equation*}
        (\mu+1)^{\frac{\beta_f(2\HT+\beta_\phi)}{4+\beta_f(2\HT+\beta_\phi)}}
        \leq  8 r
        \leq K_0.  
    \end{equation*}
    Let $\mathfrak{s}\in\{1,2,4,8\}$ and let $1 \leq k,l \leq K_0$ be such that $l-k = \mathfrak{s}r$.
    Then, for sufficiently large $\mu$, we have 
    \begin{multline*}
        \left|\EEB\!\left[\big|X_{T_{(l)}}\!-X_{T_{(k)}}\big|^{2}\right]\! - L_\T^2\!\left( \frac{l - k}{f(\T)(\mu+1)} \right)^{2\HT} \!\right|
        \\\leq  \mathfrak{c}\!\left( \frac{l - k}{f(\T)(\mu+1)} \right)^{2\HT+\min(\beta_\phi, \beta_f \HT /2)}\!,
    \end{multline*}
    where $\mathfrak{c} = \max(2L_\phi, \mathfrak{c}_1)$ and $\mathfrak{c}_1$ is a constant depending on $\HT$, $\beta_\phi$, $L_f ,\beta_f $ and $f(\T)$.
\end{lemma}

\begin{proof}[Proof of Lemma~\ref{lem:Xl-Xk}]
Note that, by the definition of $\EEB$, elementary properties of the conditional expectation, and Assumption~\assrefH{ass:L2},
\begin{align*}
    \EEB\left[\big|X_{T_{(l)}}-X_{T_{(k)}}\big|^{2}\right]
    &=\EE\left[\big|X_{T_{(l)}}-X_{T_{(k)}}\big|^{2}\1_{\mathcal{B}}\right]\\
    &=\EE\left\{\EE\left[\big|X_{T_{(l)}}-X_{T_{(k)}}\big|^{2}\1_{\mathcal{B}} \middle| M,T_{1},\dotsc, T_{M}\right]\right\}\\
    &= \EE\left\{\EE\left[\big|X_{T_{(l)}}-X_{T_{(k)}}\big|^{2} \middle| M,T_{1},\dotsc, T_{M}\right]\1_{\mathcal{B}}\right\}\\
    &= \EEB\left\{L_{t_0}^{2}|T_{(l)}-T_{(k)}|^{2\HT} \left[1+\phi_{\T}(T_{(k)},T_{(l)})\right]\right\}\\
    &=: (I) + (II),
\end{align*}
where $ (I)= L_{t_0}^{2}\EEB\left\{|T_{(l)}-T_{(k)}|^{2\HT} \right\}$.

By Lemma~\ref{lem:Tl-Tk_main} applied with $\alpha = 2\HT \leq 2$,
\begin{equation}\label{eq:reste1}
    (I) = L_{t_0}^{2}\left( \frac{l - k}{f(\T)(\mu+1)} \right)^{2\HT} (1+R_1)
\end{equation}
with
\begin{equation}
    |R_1| \leq \mathfrak{c}_1(2\HT)\left( \frac{l - k}{f(\T)(\mu+1)} \right)^{\beta_f \HT/2}.
\end{equation}
On the other hand, Assumption~\assrefH{ass:L2} implies that
\begin{equation*}
    |(II)| \leq L_{t_0}^{2}L_\phi\EEB\left(|T_{(l)}-T_{(k)}|^{2\HT+\beta_\phi}\right),
\end{equation*}
and using again Lemma~\ref{lem:Tl-Tk_main} with $\alpha = 2\HT + \beta_\phi\leq 3$ we obtain
\begin{equation}\label{eq:reste2}
    |(II)| \leq 2L_{t_0}^{2}L_\phi \left( \frac{l - k}{f(\T)(\mu+1)} \right)^{2\HT+\beta_\phi},
\end{equation}
for $\mu$ large enough such that $\mathfrak{c}_1(2\HT+ \beta_\phi)\leq 1$. Then, from~\eqref{eq:reste1} and~\eqref{eq:reste2}  we obtain
\begin{equation*}
    \EEB\left[\big|X_{T_{(l)}}-X_{T_{(k)}}\big|^{2}\right] = L_{t_0}^{2}\left( \frac{l - k}{f(\T)(\mu+1)} \right)^{2\HT} (1+R),
\end{equation*}
where $R$ is a remainder term such that, for sufficiently large $\mu$,
\begin{align*}
    |R| 
    &\leq 
    \max\left\{ 
    2L_\phi \left( \frac{l - k}{f(\T)(\mu+1)} \right)^{\beta_\phi},
    \mathfrak{c}_1(2\HT)\left( \frac{l - k}{f(\T)(\mu+1)} \right)^{\beta_f \HT/2}
    \right\}\\
    &\leq \mathfrak{c} \left( \frac{l - k}{f(\T)(\mu+1)} \right)^{\min(\beta_\phi,\beta_f \HT/2)},
\end{align*}
with $\mathfrak{c} = \max(2L_\phi, \mathfrak{c}_1(2\HT))$ with $\mathfrak{c}_1(\cdot)$ defined in Lemma \ref{lem:Tl-Tk_main}.
\end{proof}

For the sake of readability, we state below a technical lemma on the moments of the spacings $T_{(l)}-T_{(k)}$, for which the proof is given in the Appendix \ref{sec:tech_lem}. In Lemma \ref{lem:Tl-Tk_main}, we consider that $\mu$ is sufficiently large to ensure $(\mu+1)^{\beta_f\alpha/(4+\beta_f\alpha)} +1 \leq  \mu/\{2\log(\mu)\}$.

\begin{lemma}\label{lem:Tl-Tk_main}
    Let 
    $0<\alpha\leq3$ be a fixed parameter and let $r$ be an integer such that 
    \begin{equation*}
        (\mu+1)^{\frac{\beta_f\alpha}{4+\beta_f\alpha}}
        \leq  8 r
        \leq K_0 
        \quad\text{with}\quad
        K_0 \leq\frac{\mu}{2\log(\mu)}.
    \end{equation*}
    Let $\mathfrak{s}\in\{1,2,4,8\}$ and let $1 \leq k,l \leq K_0$ be such that $l-k = \mathfrak{s}r$.
    Then, for sufficiently large $\mu$, 
    \begin{equation*}
        \left|  \EEB\left[\big|T_{(l)}-T_{(k)}\big|^{\alpha}\right] - \left( \frac{l - k}{f(\T)(\mu+1)} \right)^{\alpha} \right|
        \leq  \mathfrak{c}_1 \left( \frac{l - k}{f(\T)(\mu+1)} \right)^{\alpha(1+\beta_f/4)},
    \end{equation*}
    with $\mathfrak{c}_1=\mathfrak{c}_1 (\alpha) = 8\mathfrak{c}_0 \{2f(\T)\}^{\beta_f\alpha/4}$ and $\mathfrak{c}_0$ a constant depending on $\alpha$, $L_f ,\beta_f $ and $f(\T)$.
\end{lemma}

\medskip

Lemma 2 is a generic result that will be applied with two values $\alpha \in(0,3]$, that are $\alpha =2H_{t_0}$ and $\alpha =2H_{t_0}+\beta_\phi$.

\medskip

\begin{lemma}\label{lem:hattheta-theta}
    Let $k$ be a positive integer such that $2 k - 1\leq K_0$.
    Then for any $\eta >0$,
    \begin{align*}
        q_k(\eta):= \max\left\{ \PP(\hat\theta_k - \theta_k \geq  \eta), \;\;\PP(\hat\theta_k - \theta_k \leq  -\eta) \right\} \leq \exp\left( -\mathfrak{e}N_0\eta^2 \right),
    \end{align*}
    where, using the notations introduced in Assumptions~\assrefH{ass:Lp} and~\assrefH{ass:eps},
    \begin{equation*}
        \mathfrak{e} = 1/(2\mathfrak{d}+2\mathfrak{D})
        \quad\text{with}\quad
        \mathfrak{d} = 27\big(
        \mathfrak{a}
+2\mathfrak{b}
        \big)
        \quad\text{and}\quad
        \mathfrak{D} = 9\max(\mathfrak{A}, 
\mathfrak{B})
        .
    \end{equation*}
\end{lemma}

\begin{proof}[Proof of Lemma \ref{lem:hattheta-theta} ]
By the definition in \eqref{eq:eventA} and \eqref{eq:def-hattheta}, 
\begin{equation*}
    \hat\theta_k = \frac{1}{\N0}\sum_{n=1}^{\N0}
    Z_n
    \qquad\text{where}\qquad
    Z_n = \big[Y^{(n)}_{(2k-1)}-Y^{(n)}_{(k)}\big]^{2}\mathbf{1}_{\mathcal{B}_n},
\end{equation*}
and
$
    \mathcal{B}_n
    =
    \left\{M_n \geq K_0, T^{(n)}_{(1)} \in J_\mu(\T), \dotsc,  T^{(n)}_{(K_0)} \in J_\mu(\T) \right\}.
$
Note that $\EE(\hat\theta_k) = \theta_k$. Moreover, for any $p\geq 2$, using Assumptions~\assrefH{ass:Lp} and~\assrefH{ass:eps}, we have
\begin{align*}
    \EE\big(|Z_n|^p\big)
    &= \EEB\big(|Y_{(2k-1)}-Y_{(k)}|^{2p}\big)\\
    &\leq 3^{2p-1} \EEB\left( |X_{T_{(2k-1)}} - X_{T_{(k)}}|^{2p} + |\varepsilon_{(2k-1)}|^{2p} + |\varepsilon_{(k)}|^{2p}  \right)\\
    &\leq 3^{2p-1}  \frac{p!}2 \left( 
    \mathfrak{a}
    \mathfrak{A}^{p-2} 
    + 2
    \mathfrak{b}
    \mathfrak{B}^{p-2}
    \right)\\
    &\leq  \frac{p!}2 \mathfrak{d} \mathfrak{D}^{p-2},
\end{align*}
where $\mathfrak{d}$ and $\mathfrak{D}$ are defined in the statement of this lemma. Bernstein's inequality implies
\begin{equation}\label{jan7}
     \PP(\hat\theta_k-\theta_k \geq  \eta) 
    \leq \exp\left( -\frac{N_0 \eta^2}{2\mathfrak{d}+2\mathfrak{D}\eta} \right)
    \leq \exp\left( -\mathfrak{e}N_0 \eta^2 \right),
\end{equation}
and the same bound is valid for $\PP(\hat\theta_k-\theta_k \leq - \eta)$. The  bound for $  q_k(\eta)$ follows. 
\end{proof}

\begin{lemma}\label{cor:hattheta-theta}
    Let $2 \leq k<l\leq (K_0+1)/2$ be two positive integers. For any $\eta>0$, define
    \begin{equation*}
        p_{k,l}^+(\eta) = 
        \PP(\hat\theta_l - \hat\theta_k \geq (1+\eta) (\theta_l - \theta_k))
        \quad\text{and}\quad
        p_{k,l}^-(\eta) = 
        \PP(\hat\theta_l - \hat\theta_k \leq (1-\eta) (\theta_l - \theta_k)).
    \end{equation*}
    Then, for sufficiently large $\mu$,
    \begin{equation*}
        \max\left\{p_{k,l}^+(\eta),p_{k,l}^-(\eta)  \right\} \leq 
        2\exp\left[ -\frac{\mathfrak{e}}{16} N_0 \eta^{2}\left( \frac{l-k}{\mu+1} \right)^{4\HT}\right],
    \end{equation*}
    with $\mathfrak{e}$ defined in Lemma \ref{lem:hattheta-theta}.
\end{lemma}

\begin{proof}[Proof of Lemma \ref{cor:hattheta-theta}]
Assume that $k$ and $l$ satisfy the assumptions stated in the Lemma and assume moreover that $\mu$ is large enough so that $\eta(\theta_l-\theta_k)/2 < 1$. Then
\begin{align*}
    p_{k,l}^+ (\eta)
    &= \PP\big[(\hat\theta_l  -\theta_l) - (\hat\theta_k - \theta_k) \geq \eta (\theta_l - \theta_k)\big]\\
    &\leq \PP\big[\hat\theta_l  -\theta_l \geq \eta (\theta_l - \theta_k)/2\big] + \PP\big[\hat\theta_k  -\theta_k \leq -\eta (\theta_l - \theta_k)/2\big]\\
    &\leq q_l\left( \eta (\theta_l - \theta_k)/2 \right) + q_k\left( \eta (\theta_l - \theta_k)/2 \right),
\end{align*}
and the same bound is valid for $p_{k,l}^- (\eta)$. 
By~\eqref{eq:theta:2k-1:k}, we have
$$
    \theta_{l} - \theta_k
    \geq 
 \left\{(l-k)/(\mu+1) \right\}^{2\HT} /2,
$$ 
provided  $\mu$ is sufficiently large.  We obtain the bound for $ \max(p_{k,l}^+(\eta),p_{k,l}^-(\eta)) $ after 
applying Lemma \ref{lem:hattheta-theta}.
\end{proof}

\begin{proof}[Proof  of Theorem~\ref{thm:Ht0}]
Let $\epsilon>0$. With the notation  from \eqref{eq:proxy} and~\eqref{eq:hatH}, we can write 
\begin{multline*}
\PP\left( \big|\hHT(k) - \HT\big| > \varepsilon \right)
\leq 
    \mathbf 1_{\left\{ \left|\HT(k) - \HT\right| > \epsilon/2 \right\} }
    + \PP\left( \big|\widehat H(k) - \HT(k)\big| > \varepsilon/2 \right) \\\eqqcolon B+V,
\end{multline*}
and thus it suffices to bound the terms $B$ and $V$.

\textbf{The term $B$.} 
The study of the set in the indicator function boils down to the study of the convergence of $\HT(k)$ to $\HT$.
Using Lemma~\ref{lem:Xl-Xk} with  $l-k = k-1 =  r$  we have
$$
    \theta_k - 2\sigma^2
    = \EEB\left[\left(X_{T_{(2k-1)}}-X_{T_{(k)}}\right)^2\right]
    = L_\T^2\left( \frac{k-1}{f(\T)(\mu+1)} \right)^{2\HT} (1+\rho_k),
$$
and 
\begin{equation*}
    |\rho_k| \leq 
    \mathfrak{c}\left( \frac{k-1}{f(\T)(\mu+1)} \right)^{\min(\beta_\phi, \beta_f \HT /2)}
    \eqqcolon \rho_k^*,
\end{equation*}
with $\mathfrak{c}$ a constant defined in Lemma~\ref{lem:Xl-Xk}. Using  again Lemma~\ref{lem:Xl-Xk} with $k=2k-1$, $l=4k-3$ and $\mathfrak{a}=2$ and taking the difference, we deduce that there exists $R_k$ such that
\begin{align}
    \theta_{2k-1} - \theta_k
    &= L_\T^2\left( \frac{2(k-1)}{f(\T)(\mu+1)} \right)^{2\HT} (1+\rho_{2k-1})
    \\&\qquad - L_\T^2\left( \frac{k-1}{f(\T)(\mu+1)} \right)^{2\HT} (1+\rho_k)\nonumber\\
    &= (4^{\HT}-1)L_\T^2 \left( \frac{k-1}{f(\T)(\mu+1)} \right)^{2\HT}
    \left( 1+ R_k\right)\label{eq:theta:2k-1:k},
\end{align}
where 
\begin{equation*}
    |R_k| = \left| \frac{4^{\HT}\rho_{2k-1} - \rho_{k}}{4^{\HT}-1}  \right| =
    \leq \frac{4^{\HT}+1}{4^{\HT}-1} \rho_{2k-1}^*
    \leq \frac{4^{\HT}+1}{4^{\HT}-1} \rho_{K_0}^*.
\end{equation*}
Similarly,  we obtain~:
\begin{equation}\label{eq:theta:4k-3:2k-1}
    \theta_{4k-3} - \theta_{2k-1}
    = (4^{\HT}-1)L_\T^2 \left( \frac{2(k-1)}{f(\T)(\mu+1)} \right)^{2\HT}
    \left( 1+ R_{2k-1}\right).
\end{equation}
Combining~\eqref{eq:theta:2k-1:k} and~\eqref{eq:theta:4k-3:2k-1}, we obtain
\begin{equation*}
    \log(\theta_{4k-3} - \theta_{2k-1}) 
    -\log(\theta_{2k-1} - \theta_k) 
    = \HT \log4 + \log(1+R_{2k-1}) - \log(1+R_{k}),
\end{equation*}
which leads, using the definition of $\HT(k)$ given by~\eqref{eq:proxy}, to:
\begin{equation*}
    \HT(k) = \HT + \eta_k
    \qquad\text{where}\qquad
    \eta_k = \frac{\log(1+R_{2k-1}) - \log(1+R_{k})}{2\log 2}.
\end{equation*}
Note that, for sufficiently large $\mu$, both $R_k$ and $R_{2k-1}$ are greater that $-1/2$. This implies 
\begin{equation*}
    |\eta_k| 
    \leq \frac{|R_{2k-1} - R_k|}{\log2} 
    \leq \left(\frac{2}{\log2}\frac{4^{\HT}+1}{4^{\HT}-1}\right) \rho_{4k-3}^*.
\end{equation*}
Thus, since $ \rho_{4k-3}^* \leq  \rho_{K_0}^*$, the condition $\big|\HT(k) - \HT\big| > \varepsilon/2$ fails and $B=0$ as soon as 
\begin{equation*}
    \epsilon > 
    \left(\frac{4}{\log2}\frac{4^{\HT}+1}{4^{\HT}-1}\right) \rho_{K_0}^*,
\end{equation*}
that is as soon as condition \eqref{eq:cdt_epsi} is satisfied, provided $\mu$ is sufficiently large.

\textbf{The term $V$.} Defining the event
$
    \mathcal{D} = \{\hat\theta_{4k-3}>\hat\theta_{2k-1}>\hat\theta_{k}\},
$
we can write
\begin{equation}\label{eq:second-term}
    \PP\left( \big|\widehat H(k) - \HT(k)\big| > \epsilon/2 \right) 
    \leq 
    \PP\left( \big|\widehat H(k) - \HT(k)\big| > \epsilon/2, \mathcal{D} \right) + \PP(\overline{\mathcal{D}}).
\end{equation}
First note that using Lemma~\ref{cor:hattheta-theta} we have, for sufficiently large $\mu$~:
\begin{align}
    \PP(\overline{\mathcal{D}})
    &\leq \PP(\hat\theta_k \geq  \hat\theta_{2k-1}) 
    + \PP(\hat\theta_{2k-1} \geq  \hat\theta_{4k-3})\nonumber\\
    &\leq
    p_{2k-1, k}^-(1) + p_{4k-3, 2k-1}^-(1)
    \leq
    4\exp\left[ -\frac{\mathfrak{e}}{16} N_0 \left( \frac{k-1}{\mu+1} \right)^{4\HT}\right]\label{eq:control-dbar}.
\end{align}
Now, it remains to bound the quantity
\begin{align*}
    \wp &= 
    \PP\left( \big|\widehat H(k) - \HT(k)\big| > \epsilon/2, \mathcal{D} \right) \\
    &=\PP\left[\left|\log\left( \frac{\hat\theta_{4k-3}-\hat\theta_{2k-1}}{\theta_{4k-3}-\theta_{2k-1}} \times\frac{\theta_{2k-1}-\theta_{k}}{\hat\theta_{2k-1}-\hat\theta_{k}}\right)
    \right| > \epsilon\log2, \mathcal{D} \right].
\end{align*}
Since both $\hat\theta_{4k-3}-\hat\theta_{2k-1}$ and $\hat\theta_{2k-1}-\hat\theta_{k}$ are positive under $\mathcal{D}$, we have
\begin{align*}
    \wp 
    &\leq 
    \begin{aligned}[t]
        \PP&\left[\frac{\hat\theta_{4k-3}-\hat\theta_{2k-1}}{\theta_{4k-3}-\theta_{2k-1}} \times\frac{\theta_{2k-1}-\theta_{k}}{\hat\theta_{2k-1}-\hat\theta_{k}} > 2^\epsilon, \mathcal{D} \right]\\
        &+
        \PP\left[\frac{\hat\theta_{4k-3}-\hat\theta_{2k-1}}{\theta_{4k-3}-\theta_{2k-1}} \times\frac{\theta_{2k-1}-\theta_{k}}{\hat\theta_{2k-1}-\hat\theta_{k}} < 2^{-\epsilon}, \mathcal{D} \right]
    \end{aligned}\\
    &\leq
    \begin{aligned}[t]
        \PP&\left[\frac{\hat\theta_{4k-3}-\hat\theta_{2k-1}}{\theta_{4k-3}-\theta_{2k-1}} > 2^{\frac\epsilon2} \right]
        +
        \PP\left[\frac{\hat\theta_{2k-1}-\hat\theta_{k}}{\theta_{2k-1}-\theta_{k}} < 2^{-\frac\epsilon2} \right]\\
        &+\PP\left[\frac{\hat\theta_{4k-3}-\hat\theta_{2k-1}}{\theta_{4k-3}-\theta_{2k-1}} < 2^{-\frac\epsilon2} \right]
        +
        \PP\left[\frac{\hat\theta_{2k-1}-\hat\theta_{k}}{\theta_{2k-1}-\theta_{k}} > 2^{\frac\epsilon2} \right].
    \end{aligned}
\end{align*}
Applying Lemma~\ref{cor:hattheta-theta}, we obtain~:
\begin{multline*}
    \wp
    \leq
    p_{4k-3, 2k-1}^+(2^{\frac\epsilon2}-1)
    +p_{4k-3, 2k-1}^-(1-2^{-\frac\epsilon2})
    +p_{2k-1, k}^+(2^{\frac\epsilon2}-1)
    \\+p_{2k-1, k}^-(1-2^{-\frac\epsilon2}).
\end{multline*}
Now remark that
\begin{align*}
    p_{2k-1, k}^+(2^{\frac\epsilon2}-1)
    &\leq
    2\exp\left[ -\frac{\mathfrak{e}}{16} N_0 \left( 2^{\frac\epsilon2}-1 \right)^2
    \left( \frac{k-1}{\mu+1} \right)^{4\HT}\right]\\
    &\leq 
    2\exp\left[ -\frac{\mathfrak{e}\log^2(2)}{64} N_0 \epsilon^2
    \left( \frac{k-1}{\mu+1} \right)^{4\HT}\right],
\end{align*}
and, as soon as $\epsilon<2/\log2$, we have
$1-2^{-\epsilon/2} \leq \epsilon/4$, which implies:
\begin{align*}
    p_{2k-1, k}^-(1-2^{-\frac\epsilon2})
    &\leq
    2\exp\left[ -\frac{\mathfrak{e}}{16} N_0 \left( 1-2^{-\frac\epsilon2} \right)^2
    \left( \frac{k-1}{\mu+1} \right)^{4\HT}\right]\\
    &\leq 
    2\exp\left[ -\frac{\mathfrak{e}\log^2(2)}{256} N_0 \epsilon^2
    \left( \frac{k-1}{\mu+1} \right)^{4\HT}\right].
\end{align*}
Using similar derivations for the others terms, we obtain:
\begin{equation}\label{eq:control-wp}
    \wp
    \leq
    8\exp\left[ -\frac{\mathfrak{e}\log^2(2)}{256} N_0 \epsilon^2
    \left( \frac{k-1}{\mu+1} \right)^{4\HT}\right].
\end{equation}
Combining~\eqref{eq:second-term} with~\eqref{eq:control-dbar} and~\eqref{eq:control-wp}, we obtain, for sufficiently large $\mu$ and $\epsilon < 2/\log2$:
\begin{equation*}
    \PP\left( \big|\widehat H(k) - \HT(k)\big| > \epsilon/2 \right)
    \leq 12
    \exp\left[ -\mathfrak{f} N_0 \epsilon^2
    \left( \frac{k-1}{\mu+1} \right)^{4\HT}\right],
\end{equation*}
where
  $  \mathfrak{f} = \mathfrak{e}\log^2(2)/ 256.$
\end{proof}

%% file: appendix/proof_Lp.tex
\section{Proofs of Theorems~\ref{thm:exponential-bound} and \ref{thm:local-poly}} \label{app:th-23}

The proofs of Theorems~\ref{thm:exponential-bound} and \ref{thm:local-poly} are based on the following lemmas for which the proofs are provided in the Appendix \ref{sec:tech_lem}. 
For the first lemma we consider the matrix  $A$ defined in~\eqref{eq:Anstar} with the bandwidth
 $\widehat h = M^{-1/(2\hST +1)}$. Let $\lambda$ be the smallest eigenvalue of this matrix. Let 
\begin{equation}\label{eq:AAmatrix}
    \mathbf{A} = f(\T) \int_\RR U(u) U^\top(u) K(u) du,
\end{equation}
and let $\lambda_0$ denote its smallest eigenvalue. In the following, we assume that $K(\cdot)$ satisfies~\eqref{eq:bounded-kernel}. Then $\mathbf{A}$ is positive definite \citep[see][for details]{tsybakov2009} and thus $\lambda_0>0$. 


\begin{lemma}\label{lem:eigenval}
  Let $K(\cdot)$ be a kernel such that, for any $t\in\RR$:
    \begin{equation}
        \label{eq:bounded-kernel_SM}
        \kappa^{-1}\1_{[-\delta, \delta]}(t)  \leq K(t) \leq \kappa \1_{[-1,1]}(t) ,\quad \text{for some } 0<\delta<1 \text{ and } \kappa\geq1.
    \end{equation}  
    Under Assumptions \assrefLP{ass:M2},  \assrefLP{ass:hatHt} and  \assrefLP{ass:hatkt0}, the matrix $A$ defined as in \eqref{eq:Anstar}, with $h=\widehat h$, is positive semidefinite. Moreover, there exists a positive constant $\mathfrak{g}$ that depends only on $K$, $\degree$, $f(\T)$
    and $\lambda _0$ such that, for $M$ sufficiently large,
    \begin{equation}\label{eq:E-beta}
        \PP( \lambda \leq \beta | M) 
        \leq
        2\exp(-\mathfrak{g} M\widehat h),\quad \forall 0<\beta\leq \lambda_0/2,
    \end{equation}
    and, for sufficiently large  $\mu$,
    \begin{equation}\label{eq:E2-beta}
      \sup_{0<\beta\leq \lambda_0/2}   \PP(\lambda \leq \beta) \leq \mathfrak{K}_2 \exp\left[-\frac{\mathfrak{g}}{2}\tau(\mu)\log^2(\mu)\right]
    \end{equation}  
    where 
    \begin{equation}
    \tau(\mu) = \frac{1}{ \log^{2}(\mu)}
    \left(\frac{\mu}{\log (\mu) }\right)^{\frac{2\ST}{2\ST+1}},
    \end{equation}
   with $\ST = \KT + \HT$.      Here, 
   $\mathfrak{K}_2$ is a universal constant.
\end{lemma}

Since the dimension of $A$ and $\mathbf A$ are given by $\hKT$, the probability $\PP(\cdot)$ in  Lemma \ref{lem:eigenval}  should be understood as the conditional probability given the estimator $\hKT$. 

\begin{lemma}\label{lem:technical-exp}
    Let $\xi$ be a positive random variable such that
    $
        c_1 \coloneqq \EE\left[ \exp\left( \eta_0 \xi^4 \right) \right]  < \infty,
    $
    for some positive constant $\eta_0$. Then, for any $\tau\geq 1$:
    \begin{equation*}
        \EE\left[ \exp\left( \tau\xi \right) \right]
        \leq
        c_1 \exp\left(c_2\tau^{4/3}\right)
        \qquad\text{where}\qquad
        c_2 = \left( 5/16\eta_0 \right)^{1/3}.
    \end{equation*}
\end{lemma}
    
\begin{proof}[Proof of Theorem~\ref{thm:exponential-bound}]
Without loss of generality, we could suppose that $f(\T)/2\leq f(t)\leq 2 f(\T)$, $\forall t\in J_\mu(\T)$.
We define the events 
\begin{equation}\label{eq:events-EF}
    \mathcal{E}= \{\lambda > \lambda_0/2\},  \quad \mathcal{F} = \big\{|\hHT - \HT| \leq \log^{-2}(\mu)\big\}\cap \left\{\hKT = \KT \right\},
\end{equation}
and 
$
  {\mathcal{G}} = \{ \left|X_\T\right| \leq \tau^{\tilde \alpha }(\mu)\},
  $
with $1/3< \tilde \alpha < \alpha=5/12$. 
Next, let $Z= \left| \widehat{X}_{t_0} - X_{t_0} \right|$. Assume that $\mu$ is such that $\log^{-1}(\mu/\log(\mu))<\lambda_0/2$, then using Assumption~\assrefH{ass:M2}, we have: 
\begin{align*}
    \EE[\varphi(\tau(\mu) Z^2)]
    &= (A) + (B) + (C) + (D),
\end{align*}
where $\varphi(x) = \exp(x^{1/4})$ and
\begin{align*}
    (A) &= \EE\left[\varphi\big(\tau(\mu) Z^2\big)\1_{\mathcal{E}}\1_{\mathcal{F}}\1_{\mathcal{G}}\right]\\
    (B) &= \EE\left[\varphi\big(\tau(\mu) Z^2\big) \1_{\overline{\mathcal{E}}}\right]\leq \EE^{1/2}\left[\varphi^2\big(\tau(\mu) Z^2\big) \right]\PP^{1/2}({\overline{\mathcal{E}}})\\
    (C) &= \EE\left[\varphi\big(\tau(\mu) Z^2\big) \1_{\overline{\mathcal{F}}}\right] \leq \EE^{1/2}\left[\varphi^2\big(\tau(\mu) Z^2\big) \right]\PP^{1/2}({\overline{\mathcal{F}}}) \\
     (D) &= \EE\left[\varphi\big(\tau(\mu) Z^2\big) \1_{\overline{\mathcal{G}}}\right]  \leq \EE^{1/2}\left[\varphi^2\big(\tau(\mu) Z^2\big) \right]\PP^{1/2}({\overline{\mathcal{G}}}).
\end{align*}
We show that $(A)$ is the main term, and it is bounded by a constant. 

By construction, $\left|\hat X_\T\right| \leq \tau^{ \alpha }(M)$ and, by \assrefLP{ass:M2}, 
$$
\tau^{ \alpha }(M) \geq \tau^{ \alpha }(\mu/\log (\mu)) \geq  \tau^{\tilde \alpha }(\mu),
$$
provided that $\mu$ is sufficiently large. Thus, for sufficiently large $\mu$,
$$
\left| \hat X_{\T} - X_{\T}\right| \1_{{\mathcal{G}}} \leq \left|U^\top(0) \hat{\vartheta} - X_{\T}\right| \1_{{\mathcal{G}}}\leq \left|U^\top(0) \hat{\vartheta} - X_{\T}\right| ,
$$
with  $\hat{\vartheta} = A^{-1} {a}$ and $A$ and $a$ defined in \eqref{eq:Anstar} and \eqref{eq:anstar}, respectively. 
Therefore,  
\begin{equation*}
    \sqrt{\tau(\mu)} Z \1_{{\mathcal{G}}} \leq \sqrt{\tau(\mu)} \left|\sum_{m = 1}^M (X_{T_m} - X_{\T})W_m\right| + \sqrt{\tau(\mu)}\left|\sum_{m = 1}^{M}\varepsilon_mW_m\right|,
\end{equation*}
where
\begin{equation*}
    W_m = \frac1{Mh} U^\top(0) A^{-1} 
    U\left( \frac{T_m - \T}{h} \right)
    K\left( \frac{T_m - \T}{h} \right).
\end{equation*}
This leads to
\begin{multline*}
    \left(\sqrt{\tau(\mu)} Z\right)^{1/2}  \1_{{\mathcal{G}}} \leq \left(\sqrt{\tau(\mu)} \left|\sum_{m = 1}^M (X_{T_m} - X_{t_0})W_m\right|\right)^{1/2} + \\\left(\sqrt{\tau(\mu)}\left|\sum_{m = 1}^{M}\varepsilon_mW_m\right|\right)^{1/2}.
\end{multline*}
Then, to show that $(A)$ 
is finite, it suffices to show that   
\begin{equation*}
    A_1 = \EE\left[
    \exp\left\{\left(4\sqrt{\tau(\mu)}\left|\sum_{m = 1}^M (X_{T_m} - X_{t_0})W_m\right|\right)^{1/2}\right\}\1_{\mathcal{E}}\1_{\mathcal{F}}  
    \right],
\end{equation*}
and
\begin{equation*}
    A_2 = \EE\left[
    \exp\left\{\left(4\sqrt{\tau(\mu)}\left|\sum_{m = 1}^{M}\varepsilon_mW_m\right|\right)^{1/2}\right\}
    \1_{\mathcal{E}}\1_{\mathcal{F}}  \right],
\end{equation*}
are finite and to apply Cauchy-Schwarz inequality. To control the stochastic term $A_2$, remark that: 
\begin{equation}\label{eq:def-A2}
    A_2 = 1+\sum_{p \geq 1} \frac{2^p[\tau(\mu)]^{p/4}}{p!}  B_p,
\end{equation}
where
\begin{equation*}
    B_p = \EE\left[\left|\sum_{m= 1}^{M}\varepsilon_m W_m\right|^{p/2}\1_{\mathcal{E}}\1_{\mathcal{F}}  \right].
\end{equation*}
By Jensen's inequality, $B_1\leq B_2^{1/2} \leq B_3^{1/3} \leq B_4^{1/4}$. Thus, it remains to control $B_p$ for any $p\geq 4$. For such values of $p$, we use Marcinkiewicz-Zygmund's inequality and obtain:
\begin{equation*}
    B_p \leq \left(\frac{p}{2}-1\right)^{p/2}
    \EE\left[ \left( \sum_{m = 1}^M \varepsilon_m^2 W_m^2 \right)^{p/4}\1_{\mathcal{E}}\1_{\mathcal{F}}  
    \right].
\end{equation*}
By a version of  Lemma 1.3 of \cite{tsybakov2009}, for $\kappa$ defined in \eqref{eq:bounded-kernel},
$$
\sup_{1\leq m\leq M}|W_m|  \1_{\mathcal{E}_{\lambda_0/2}} \leq  \frac{4\kappa}{\lambda_0 Mh} ,
$$
and
\begin{equation}\label{eq:bound-Wm-1m}
    \sum_{m = 1}^{M} |W_m|  \1_{\mathcal{E}_{\lambda_0/2}}  \leq  \frac{4\kappa}{\lambda_0 }  \frac{1}{Mh} \sum_{j = 1}^{M}\1_{ \{ \T-h \leq T_m \leq \T+h \} }.
\end{equation}
Let $\chi_m =h^{-1}\1_{ \{ \T-h \leq T_m \leq \T+h \} }$.     
By the Rosenthal inequality, there exists a universal constant $C$, such that, for any $q\geq 1$,
\begin{align}\label{eq:moments-W}
    \tilde \EE \left[  \left(\sum_{m = 1}^{M} \chi_m \right)^{q} \right] &\leq \frac{C^qq^q}{\log^q q}  \max\left\{   
    \sum_{m= 1}^{M}  \tilde \EE \chi _m^{q}   , \left(  \sum_{m= 1}^{M}  \tilde \EE \chi _m \right)^q     \right\}\\ &\leq \left(  \frac{4qC f(\T)
     M}{\log q} \right)^q .
\end{align}
See \cite{Jon85}.      
Since $W_m^2\leq |W_m | \sup_{1\leq j\leq M}|W_j|$, we deduce    
$$
\tilde\EE \left[ \left(\sum_{m = 1}^{M}W_m^2\right)^{q} \1_{\mathcal{E}} 
\right] \leq  \left( \frac{4\kappa}{\lambda_0 Mh}   \right)^q \left(  \frac{16q\kappa C
f(_T) }{\lambda_0 \log q} \right)^q =: \left( \frac{1}{Mh}   \right)^q \left(  \frac{c_0 q  }{\log q} \right)^q .
$$
Using Assumption~\assrefLP{ass:M2}, we deduce:
\begin{align}
    \EE\! \left[ \left(\sum_{m = 1}^{M}W_m^2\right)^{\!q} \1_{\mathcal{E}}\1_{\mathcal{F}}  \right] 
    &\leq
    \left(  \frac{c_0 q  }{\log q} \right)^{\!q }  
    \EE\left[ \left(\frac{\log\mu}{\mu}\right)^{q\frac{2\hHT}{2\hHT+1}} \1_{\mathcal{F}}  \right] \\
    &= 2 \left(  \frac{c_0 q  }{\log q} \right)^q \frac{1}{\left\{ \tau(\mu)\log^2\mu \right\}^q} .
\end{align}
The last line can be deduced using similar arguments to those used to obtain~\eqref{eq:hHT-HT} in the proof of Lemma \ref{lem:eigenval}. 
Next, let $\widetilde W_m = W_m^2/\sum_{j=1}^M W_j^2$. Since the error terms are independent on the $T_m$'s, and using \assrefH{ass:eps}, by Jensen's inequality
\begin{align*}
    B_p \left(\frac{p}{2}-1\right)^{-p/2} &\leq        \EE \left[ \left( \sum_{m= 1}^M |\varepsilon_m|^{p/2} \widetilde{W}_m \right)\left(\sum_{j = 1}^{M}W_j^2\right)^{p/4}\1_{\mathcal{E}}\1_{\mathcal{F}}  \right] 
    \\& \mkern-72mu=
    \EE \left\{  \EE \left[ \left( \sum_{m= 1}^M |\varepsilon_m|^{p/2} \widetilde{W}_m \right)\mid M,W_1,\ldots,W_M \right]\left(\sum_{j = 1}^{M}W_j^2\right)^{p/4}\1_{\mathcal{E}}\1_{\mathcal{F}}  \right\} \\
    &\mkern-72mu\leq 
    \left(\EE |\varepsilon|^{2p}\right)^{1/4} \EE \left[ \left(\sum_{m = 1}^{M}W_m^2\right)^{p/4} \1_{\mathcal{E}}\1_{\mathcal{F}}  \right] \\
    &\mkern-72mu\leq \left(\frac{p!}{2}\mathfrak b \mathfrak B^{p - 2}\right)^{1/4}  \left(  \frac{c_0 p  }{4\log (p/4)} \right)^{p/4} 
    {\left(\frac{1}{\tau(\mu)\log^2\mu}\right)^{p/4}}.
\end{align*}
Thus we have
\begin{align*}
    B_p &\leq \left(\frac{p}{2}-1\right)^{p/2}   \left(\frac{p!}{2}\mathfrak b \mathfrak B^{p - 2}\right)^{1/4}  \left(  \frac{c_0 p  }{4\log (p/4)} \right)^{p/4} {\left(\frac{1}{\tau(\mu)\log^2\mu}\right)^{p/4}}\\
    &=  \left(\frac{\mathfrak b}{2\mathfrak B^{2}} \right)^{1/4}  {\left(\frac{1}{\tau(\mu)\log^2\mu}\right)^{p/4}} D_p,
\end{align*}
where
$$
D_p =\left(\frac{p}{2}-1\right)^{p/2}     (p!)^{1/4}   \left(  \frac{c_0 p \mathfrak B }{4\log (p/4)} \right)^{p/4}  \leq p! \; \left(  \frac{c_1}{\log p} \right)^{p/4} ,
$$
for some constant $c_1$. 
For the last inequality, we use Stirling's formula. This implies that there exists a universal constant $c_2$ such that 
\begin{equation}\label{eq:bound-Bp}
    \frac{B_p}{p!} \leq  \left(  \frac{c_2}{\log p} \right)^{p/4}  {\left(\frac{1}{\tau(\mu)\log^2\mu}\right)^{p/4}}.
\end{equation}
Combining~\eqref{eq:def-A2} with~\eqref{eq:bound-Bp} we obtain:
\begin{align*}
    A_2 
    &=
    1+{\left\{
    2B_1\tau^{1/4}(\mu) +
    2B_2\tau^{1/2}(\mu) + 
    \frac{4B_3}3\tau^{3/4}(\mu)
    \right\}} + 
    \sum_{p \geq 4} \frac{2^p  \tau(\mu)^{p/4}}{p!}  B_p \\
    &\leq
    1+
    {\left\{
    2\big(B_4\tau(\mu)\big)^{1/4} +
    2\big(B_4\tau(\mu)\big)^{1/2} + 
    \frac{4\big(B_4\tau(\mu)\big)^{3/4}}3
    \right\}} \\&\mkern342mu + \sum_{p \geq 4} \left(  \frac{16 c_2}{\log p} \right)^{p/4} {\left( \frac{1}{\log^2\mu} \right)^{p/4}}\\
    &< \infty.
\end{align*}
The inequality on the last line comes from the fact that $B_4\tau(\mu)\log^2(\mu)$ is bounded.

To control the bias term $A_1$, let us define, for any $0<\beta<\HT$:
\begin{equation}\label{eq:lambda_RY}
    \Lambda_\beta
    = \sup_{\substack{u, v \in J_\mu(\T)\\u\neq v}}
    \frac{|X^{(\KT)}_u-X^{(\KT)}_v|}{|u-v|^\beta},
\end{equation}
where here $X^{(\KT)}_u$ denotes the $\KT$-th derivative of the trajectory $X_u$.
Applying Taylor's formula and using the basic properties satisfied by the weights $W_m$, we obtain:
\begin{align*}
    \left| \sum_{m = 1}^M X(T_m) W_m - X_\T \right|
    &\leq \left|\sum_{m=1}^M \sum_{k=1}^{\KT} \frac{X^{(k)}(\T)}{k!} (T_m-\T)^k W_m\right| \\
    &\qquad + \sum_{m=1}^M \frac{\left|X^{(\KT)}(\T)-X^{(\KT)}(\zeta_m)\right|}{\KT!}|T_m-\T|^{\KT} |W_m| \\
    &\leq \frac{\Lambda_\beta}{\KT!}\sum_{m = 1}^M  |T_m  - \T|^{\KT+ \beta} |W_m|,
\end{align*}
where $|\zeta_m-\T|\leq |T_m-\T|$. Note that this result is obtained using~:
\begin{equation*}
    \sum_{m=1}^M (T_m-\T)^k W_m = 0.
\end{equation*}
Since, under $\mathcal{E}$ we have, $W_m =0$ as soon as $|T_m-\T| > h$,  :
\begin{align*}
    \left| \sum_{m = 1}^M X(T_m) W_m - X_\T \right|\1_{\mathcal{E}}
    &\leq \frac{\Lambda_\beta h^{\KT+\beta}}{\KT!} \sum_{m=1}^M |W_m|\1_{\mathcal{E}}\\
    &\leq \frac{\Lambda_\beta}{\KT!}\frac{4\kappa}{\lambda_0}
    \frac{ h^{\KT+\beta}}{Mh} \sum_{m=1}^M \1_{\{\T-h \leq T_m \leq \T+h\}}.
\end{align*}
The last line follows from \eqref{eq:bound-Wm-1m}. Moreover, combining the result obtained by \citep[][p.~27]{MR1725357}, with \assrefH{ass:Lp}, for any  $0< H_{\T}- \beta<\beta_0$ where $\beta_0$ is some sufficiently small  fixed value, we have:
\begin{align*}
    \EE\Lambda_\beta^{p/2} &
    \leq 2^{\frac{1}{4} + \frac{p}{2}(H_{t_0} + 1)}\left(\frac{1}{1 - 2^{\beta - H_{t_0}}}\right)^{p/2}\left(\frac{p!}{2}\mathfrak{a}\mathfrak{A}^{p - 2}\right)^{1/4} 
    \\ &
    \leq \frac{\mathfrak{a}^{1/4}}{\mathfrak{A}^{1/2}}(p!)^{1/4}\left(\frac{8\log 2 \sqrt{\mathfrak{A}}}{H_{t_0} - \beta}\right)^{p/2}.
\end{align*}
Since, by definition, the random variable $\Lambda_\beta$ is independent of $\hHT$, $M$ and the $T_m$'s, by the last inequality above and inequality~\eqref{eq:moments-W}, we have:
\begin{multline*}
    \EE\left(\left| \sum_{m = 1}^M X(T_m) W_m - X_\T \right|^{p/2}\1_{\mathcal{E}}\1_{\mathcal{F}}   \right)
    \\\leq  \left(  \frac{2pCf(\T)
    }{\log (p/2)} \right)^{p/2} \EE\left[\left(h^{\KT+\beta} \right)^{p/2}\right] 
    \EE\left[ \left( \Lambda_\beta \right)^{p/2} \right].
\end{multline*}
We thus obtain:
\begin{align*}
    A_1
    &\leq \sum_{p \geq 0}\frac{\big(16\tau(\mu)\big)^{p/4}}{p!}
    \EE\left(\left| \sum_{m = 1}^M X(T_m) W_m - X_\T \right|^{p/2}\1_{\mathcal{E}}\1_{\mathcal{F}}   \right) \\
    &\leq \sum_{p\geq 0}\frac{\big(16\tau(\mu)\big)^{p/4}}{p!} \left(  \frac{2pCf(\T)
    }{\log (p/2)} \right)^{p/2} \EE\left[\left(h^{\KT+\beta} \right)^{p/2}\1_{\mathcal{F}}\right] 
    \EE\left[ \left( \Lambda_\beta \right)^{p/2} \right].
\end{align*}
Note that, 
on the event $\mathcal F$,
$$
\left(h^{\KT+\beta} \right)^{p/2} \leq C^{p/2} \left(\frac{\log\mu}{\mu}\right)^{\frac{p}{2}\frac{2(\KT+\beta)}{2(\KT+\HT)+1}}.
$$
Taking $\beta = \HT - \log^{-1}\mu$, since $(\mu /\log \mu)^{1/\log \mu}$ is bounded, we deduce that, for some constant $C>0$,  
$$
A_1 \leq  \sum_{p\geq 0}   \frac{C^{p/2}}{\log^{p/2} (p)}<\infty. 
$$

It remains to control $(B)$, $(C)$ and $(D)$. For this purpose, let us first note that, by the Assumption~\assrefLP{ass:Lp2},
$
c_1:=\EE [\exp(\eta_0 X^2_\T)]<\infty,
$
for $\eta_0= 1/(2\mathfrak{A})$. 
We deduce that 
\begin{align*}
    \EE \left[\varphi^2\big(\tau(\mu) Z^2\big) \right] 
    &\leq \EE \left[\exp \big( 2 \tau^{1/4} (\mu) \left|\hat X_\T - X_\T  \right|^{1/2}  \big) \right]\\
    &\leq \EE \left[\exp  \left(2 \tau^{1/4} (\mu)\big\{|\hat X_\T|^{1/2} +\left|X_\T  \right|^{1/2}  \big\} \right) \right]\\
    &\leq \exp  \left[ 2 \tau^{1/4} (\mu)\tau^{\alpha/1} (\mu \log(\mu)) \right] \EE \left[\exp  \left(2 \tau^{1/4} (\mu)\left|X_\T  \right|^{1/2}   \right) \right]\\
    &\leq c_1 \exp  \left[ 2 \tau^{1/4} (\mu)\tau^{\alpha/2} (\mu \log(\mu)) + 2^{4/3}c_2 \tau^{1/3} (\mu)\right],
\end{align*}
where for the last inequality, we apply Lemma \ref{lem:technical-exp} with and $\xi = |X_\T|^{1/2}$, and thus $c_2=(5\mathfrak{A}/8 )^{1/3}$. 
Now notice that, using Markov's inequality
\begin{align*}
    \PP(\overline{\mathcal{G}})
    &= \PP(|X_\T| > \tau^{\tilde\alpha}(\mu)) \\
    &\leq c_1 \exp\left( -\eta_0\tau^{2\tilde\alpha}(\mu) \right).
\end{align*}
Since $\tilde\alpha<1/2$, Assumptions~\assrefLP{ass:hatHt} and~\assrefLP{ass:hatkt0} imply that for sufficiently large $\mu$:
\begin{equation*}
    \PP(\overline{\mathcal{F}})
    \leq 2 \mathfrak{K}_1\exp(-\mu)
    \leq  \exp\left( -\eta_0\tau^{2\tilde\alpha}(\mu) \right).
\end{equation*}
Moreover, Lemma~\ref{lem:eigenval} also implies that, for sufficiently large $\mu$:
\begin{equation*}
    \PP(\overline{\mathcal{E}})
    \leq \mathfrak{K}_2 \exp\left(-\frac{\mathfrak{g}}{2} \tau(\mu)\log^2(\mu)\right)
    \leq  \exp\left( -\eta_0\tau^{2\tilde\alpha}(\mu) \right).
\end{equation*} 
Finally, if $\mathcal{H}$ denotes either $\mathcal{E}$, $\mathcal{F}$ or $\mathcal{G}$, we have
\begin{multline*}
    \EE \left[\varphi^2\big(\tau(\mu) Z^2\big) \right]
    \PP(\overline{\mathcal{H}})
    \\\leq
    C  \exp  \left[ 2 \tau^{1/4} (\mu)\tau^{\alpha/2} (\mu \log(\mu)) + 2^{4/3}c_2 \tau^{1/3} (\mu) - \eta_0\tau^{2\tilde\alpha}(\mu)\right],
\end{multline*}
where $C$ denotes a positive constant. The choice $\alpha = 5/12$ and $\tilde\alpha = 9/24$ allows us to deduce that $\EE \left[\varphi^2\big(\tau(\mu) Z^2\big) \right]\PP(\overline{\mathcal{H}})$ is bounded. This concludes the proof of Theorem~\ref{thm:exponential-bound}.
\end{proof}

\begin{proof}[Proof of Theorem~\ref{thm:local-poly}]
By Theorem~\ref{thm:exponential-bound},
\begin{equation}\label{eq:max-esp}
    \max_{1\leq \n1 \leq  \N1} 
    \EE\left[\varphi\left\{
    \tau(\mu) 
    \left|\hXT1 - \XT1\right|^2
    \right\}\right] 
    \leq \Gamma_0
    \end{equation}
    where
    \begin{equation}\tau(\mu) = \frac{1}{ \log^{2}(\mu)}
    \left(\frac{\mu}{\log (\mu) }\right)^{\frac{2\ST}{2\ST+1}},
    \end{equation}
and   $\varphi(x) = \exp(x^{1/4})$. 
Now, let $x_0 = 256$ and consider $\widetilde\varphi\leq\varphi$ defined by
\begin{equation*}
    \widetilde{\varphi}(x) =
    \begin{cases}
        \varphi^\prime(x_0)(x - x_0) + \varphi(x_0) &\text{if}~ x \leq x_0 \\
        \varphi(x) &\text{if}~ x \geq x_0
    \end{cases},
\end{equation*}
and note that $\widetilde\varphi$ is  nondecreasing and convex. Then, by Lemma~1.6 in \cite{tsybakov2009}, 
\begin{equation*}
    \EE\left(\max_{1\leq \n1 \leq  \N1} 
    \left|\hXT1 - \XT1\right|^2\right) 
    \leq \tau^{-1}(\mu)\widetilde{\varphi}^{\leftarrow}(\Gamma_0 \N1),
\end{equation*}
where $\widetilde{\varphi}^{\leftarrow}$ denotes the inverse function of $\widetilde\varphi$. Moreover, for $\N1$ sufficiently large, we have $\widetilde{\varphi}^{\leftarrow}(\Gamma_0\N1) = \log^4(\Gamma_0 \N1)$. \end{proof}

%% file: appendix/proof_sup.tex

\section{Proof of Theorem~\ref{thm:local-poly-sup}}

\begin{proof}[Proof of Theorem~\ref{thm:local-poly-sup}]
Assume, without loss of generality that 
$$K = \log^{4}(\N1)\exp(\log^2(\mu)),$$ 
is an integer and let $\min(I) = s_0 < s_1 < \dotsc < s_K < s_{K+1} = \max(I)$ be a regular grid of the interval~$I$. For any $t\in I$, let $k_t \in\{1,\dotsc, K\}$  be such that $|t-s_{k_t}| \leq 1/(2K-2)\eqqcolon \epsilon$. 
We have
\begin{equation*}
    \EE\left( 
        \max_{1\leq \n1 \leq \N1} 
        \sup_{t\in I} 
        \left|\hXt1 - \Xt1\right|^2
    \right)
    \leq 3 (A+B+C),
\end{equation*}
where
\begin{align*}
    A &= \EE\left(  \max_{1\leq \n1 \leq \N1} \max_{k=1,\dotsc,K}\left|\hXtemp1_{s_k} - \Xtemp1_{s_{k}}\right|^2 \right),\\ B &= \EE\left(  \max_{1\leq \n1 \leq \N1} \sup_{t\in I}\left|\hXt1 - \hXtemp1_{s_{k_t}}\right|^2 \right)\\
   C &= \EE\left(  \max_{1\leq \n1 \leq \N1} \sup_{t\in I}\left|\Xt1 - \Xtemp1_{s_{k_t}}\right|^2 \right).
\end{align*}

\textbf{Bound for $A$.} Using arguments similar to those of the proof of Theorem~\ref{thm:local-poly}, we obtain:
\begin{equation*}
    A
    \leq c\tau^{-1}(\mu)\log^4(\N1 K) 
    = 
    c \Psi(\mu, \N1),
\end{equation*}
where $c$ denotes an absolute positive constant and $\Psi$ is defined by~\eqref{eq:Psi}.

\textbf{Bound for $B$.} Note that using~\cite[][p.~45]{tsybakov2009} we obtain that there exists a positive constant $\ell$ such that, for any $1\leq\n1\leq\N1$, we have almost surely:
\begin{equation*}
    \sup_{|t-s|<\epsilon} |\hXtemp1_t - \hXtemp1_s| 
    \leq 
    \frac{\epsilon\ell}{\lambda_0 M_{\n1}h^2} \sum_{i=1}^{M_{\n1}} |Y_i^{[\n1]}|.
\end{equation*}
Using~\assrefH{ass:eps} and~\assrefLP{ass:Lp2}, this implies that there exists a positive constant $\mathfrak Y$ such that:
\begin{equation*}
    B 
    \leq \frac{\ell\mathfrak{Y}}{\lambda_0}\frac{\epsilon}{h^2}
    \leq \frac{\ell\mathfrak{Y}}{\lambda_0} \epsilon\mu^2
    \ll  \Psi(\mu,\N1).
\end{equation*}

\textbf{Bound for $C$.} Set $0<\eta<H/2$ and, for any $1\leq\n1\leq\N1$, define the random variable:
\begin{equation*}
    \Lambda_{\n1} = \sup_{s\neq t\in I} \frac{|\Xtemp1_t- \Xtemp1_s|}{|t-s|^\eta}.
\end{equation*}
We have:
\begin{align*}
    C \leq \EE\left(  \max_{1\leq \n1 \leq \N1} \sup_{|t-s|\leq \epsilon}\left|\Xt1 - \Xtemp1_{s}\right|^2 \right)
    \leq \epsilon^{2\eta} \EE\left( \max_{1\leq \n1 \leq \N1}  \Lambda_{\n1}^2\right).
\end{align*}
It remains to bound the last expectation. By~\assrefH{ass:Lp}, we have, for any $1\leq\n1\leq\N1$ and any $p\geq 2$:
\begin{equation*}
    \EE\left( \Lambda_{\n1}^{2p} \right) 
    \leq \left( \frac{p!}{2} \mathfrak{a}\mathfrak{A}^{p-2} \right) |I|^{2p(H-\eta)} 
    = p!\mathfrak{M}^{p} ,
\end{equation*}
where $|I|$ denotes the length of $I$ and
\begin{equation*}
    \mathfrak{M}=\max\left[\frac{\mathfrak{a}}2|I|^{4(H-\eta)},\mathfrak{A}|I|^{2(H-\eta)},1\right].   
\end{equation*}
Thus, we have
\begin{equation*}
    \EE\left[ \exp\left(\frac{\Lambda_{\n1}^2}{2\mathfrak{M}} \right) \right]\leq 2
    \qquad\text{which implies}\qquad
    \EE\left( \max_{1\leq \n1 \leq \N1}  \Lambda_{\n1}^2\right) 
    \leq
    2\mathfrak{M}\log(2\N1). 
\end{equation*}
We finally obtain that:
\begin{equation*}
    C \leq 2\mathfrak{M}\log(2\N1) \epsilon^{2\eta} \ll \Psi(\mu, \N1).
\end{equation*}
Gathering the bounds, we deduce that:
\begin{equation*}
    \EE\left( 
        \max_{1\leq \n1 \leq \N1} 
        \sup_{t\in I} 
        \left|\hXt1 - \Xt1\right|^2
    \right)
    \leq 3c \Psi(\mu, \N1)\{1+o(1)\}.
\end{equation*}
\end{proof}
%
%
%

%% file: supplement/presmoothing_regularity.tex

\section{Alternative local regularity estimator}\label{subsec:golo_new}

\subsection{Main assumptions}

In this section we propose an alternative approach to estimate the regularity $\ST = \KT + \HT$ of the process $X$ without any restriction on $\KT\in\NN$. The main idea is to replace the noisy observations in~\eqref{eq:def-hattheta} by smoothed versions of the sample paths of the process.
To construct this estimator of $\ST$ and to derive its theoretical properties we need a set of assumptions that slightly differ from the one presented in the main manuscript. In what follows we fix an open subinterval $\Ostar$ of $J_\mu(\T)$ with length $0 < \Delta\leq 1$ and, for the sake of homogeneity, we denote $H_\KT = H_\T$.  

\begin{definition}\label{def_loc_reg}
For any $d\in\NN$, $0<H_d\leq 1$ and $L_d>0$, the class $\mathcal{X}(d+H_d, L_d; \Ostar )$ is the set of stochastic processes indexed by  $t\in\Ostar $ for which the following conditions hold true. 
\begin{assumptionHt}
    \item\label{Ht:derivatives}   With probability 1, for any $\ell\in\{0,\dotsc,d\}$ the $\ell$-th order derivative $\nabla^\ell X_t$ of $X_t$ exists for all $t\in\Ostar $, and satisfies:
    \begin{equation*}
    0 < \underline{a}_\ell = \inf_{u\in\Ostar } \EE\left[(\nabla^\ell X_u)^2\right] 
    \leq \sup_{u\in\Ostar } \EE\left[(\nabla^\ell X_u)^2\right] = \overline{a}_\ell < \infty.    
    \end{equation*}
    \item\label{Ht:equivalent} Two positive constants $S_d$ and  $\beta_d$  exist such that:
    \begin{equation*}
        \left|
        \EE\left[ (\nabla^d X_t - \nabla^d X_s)^{2} \right]
        -  L_d^2|t-s|^{2H_d}
        \right|
        \leq
        S_d^2|t-s|^{2H_d}\Delta^{2\beta_d},
        \qquad
        s,t\in\Ostar .
    \end{equation*}
    \item\label{Ht:moments}  $\mathfrak{a}>0$ and $\mathfrak{A}>0$ exist such that, for any $\ell\in\{0,\dotsc, d\}$ and any $p\geq 1$:
    \begin{equation*}
        \EE\left[ |\nabla^\ell X_t - \nabla^\ell X_s|^{2p} \right] \leq 
        \frac{p!}{2} \mathfrak{a} \mathfrak{A}^{p-2},
        \qquad
        s,t\in\Ostar.
     \end{equation*}
\end{assumptionHt}
The quantity $d+H_d$ is the local regularity of the process on $\Ostar $,   while $L_d$ is the Hölder constant of the $d-$th derivative of the trajectories.  
\end{definition}

These classes of processes satisfy embedding properties that will be useful to construct an estimator of $\KT$.

\begin{lemma}\label{lem:regularity-1}
   Assume that $X$ restricted to $\Ostar$ belongs to $\mathcal{X}(\KT+H_\KT, L_\KT;\Ostar)$  for some $0<H_\KT\leq 1$ and $L_\KT>0$. Then, for any $d\in\{0,\dotsc,\KT-1\}$, two positive real numbers $L_d$ and $S_d$ exist such that
    \begin{equation*}
        \left|
        \EE\left[ (\nabla^d X_t - \nabla^d X_s)^{2} \right]
        -  L_d^2|t-s|^{2}
        \right|
        \leq
        S_d^2|t-s|^{2}\Delta^{H_{d+1}}, \quad s,t\in\Ostar,
    \end{equation*}
    with $H_{d+1}= \1_{\{d\neq \KT-1\}} + H_\KT\1_{\{d = \KT-1\}}$.  This  implies that, for any $0\leq d \leq \KT$, the process $X$ restricted to $\Ostar$ belongs to the class $\mathcal{X}(d+H_d, L_d ;\Ostar)$.
\end{lemma}

The three parameters $\KT\in\NN$, $0<H_\KT< 1$ and $L_\KT>0$ are fixed for the rest of the Section. We also  assume that $X$ restricted to $\Ostar$ belongs to $\mathcal{X}(\KT+H_\KT, L_\KT;\Ostar)$. 

\subsection{Heuristics on the definition of the alternative estimator}

Using Lemma~\ref{lem:regularity-1}, we remark that
\begin{equation*}
    \KT = \min\{d \in\NN : H_d < 1 \}.
\end{equation*}
A natural idea  to construct an estimator of $\ST$, is thus to find an estimator $\hat H_d$ of $H_d$, for any $d\in\NN$, and to define:
\begin{equation*}
    \widehat\KT =  \min\{d \in\NN : \hat H_d < 1 - \varphi(\mu) \}
    \quad\text{and}\quad
    \widehat{\ST} = \widehat\KT + \widehat H_{\widehat\KT},
\end{equation*}
for some decreasing function $\varphi(\cdot)$ which will be defined later.

Thus, our problem reduces to the construction of accurate estimators of $H_d$ for all $d\in\NN$. For simplicity, let us denote, for any $s,t\in\Ostar$:
\begin{equation*}
    \theta_d(s,t) = \EE\left[ (\nabla^d X_t - \nabla^d X_s)^{2} \right]
    \approx L_d^2 |t-s|^{2H_d} 
    \quad \text{if }  \Delta \text{ is small}.
\end{equation*}
Now, let $t_1$ and $t_3$ be such that $[t_1,t_3]\subset\Ostar$ and $t_3-t_1=\Delta/2$. Denote by $t_2$ the middle point of $[t_1, t_3]$. It is easily seen that
\begin{equation}\label{eq:tilde-Hd}
    H_d \approx \tilde H_d =  \frac{\log(\theta_d(t_1,t_3)) - \log(\theta_d(t_1,t_2))}{2\log(2)} \quad \text{if }  \Delta \text{ is small}.
\end{equation}
This suggests to define
\begin{equation*}
    \widehat H_d = \frac{\log(\hat\theta_d(t_1,t_3)) - \log(\hat\theta_d(t_1,t_2))}{2\log(2)},
\end{equation*}
where, for any $s,t\in\Ostar$
\begin{equation*}
    \hat\theta_d(s,t) = \frac1{\N0} \sum_{n=1}^{\N0} \left(\widetilde{\nabla^d X}^{(n)}_t - \widetilde{\nabla^d X}^{(n)}_s\right)^{2}.
\end{equation*}
Here, $\widetilde{\nabla^d X}^{(n)}$ denotes a pilot estimator of the curve $\nabla^dX^{(n)}$ that can be obtained by a presmoothing procedure.

\subsection{Concentration properties 
}

The quality of the estimator $\widehat{\ST}$ depends on the quality of the generic nonparametric estimators $\widetilde{\nabla^d X}$ of $\nabla^d X$. To quantify their behavior, we consider the local $\mathbb{L}^p$-risk 
\begin{equation*}
    R_p(d)    = R_p(d;\Ostar )  = \sup_{t\in\Ostar} \EE\left( |\xi_d(t)|^p \right),\quad \text{where } \quad \xi_d(t) = \widetilde{\nabla^d X}_t - {\nabla^d X}_t.
\end{equation*}
Our method applies with any type of nonparametric estimator $\widetilde{\nabla^d X}$ (local polynomials, splines,...) as soon as,  for any $p\in\NN$, its $\mathbb{L}^p$-risk  is  suitably bounded. The following mild condition is satisfied by common estimators, see for instance Theorem 1 in \cite{gaiffas2007} for the case of local polynomials. 

\begin{assumptionLP}
    \item\label{LP:1} There exist two positive constants $\mathfrak{c}$ and $\mathfrak{C}$ such that
    \begin{equation*}
        R_{2p}(d)
        \leq 
        \frac{p!}{2} \mathfrak{c} \mathfrak{C}^{p-2}, \qquad \forall p\geq 1,\; d\in\{0,\dotsc, \delta\}.
    \end{equation*}
\end{assumptionLP}

We can now derive an exponential bound for the concentration of all the estimators $\hat H_d$, $d\in\{0,\dotsc, \KT\}$. To make this exponential bound useful for deriving 
optimal rates for our estimators of the mean and covariance functions,  we will require  the largest risk among $R_2(0),\ldots,R_2(\KT)$ to tend to zero as   $\mu$ increases to infinity. 

\begin{theorem}\label{thm:estimation-alpha}
Assume that  $X$ restricted to $\Ostar$ belongs to $\mathcal{X}(\KT+H_\KT, L_\KT;\Ostar)$, for some integer $\KT \geq 0$ and $0<H_\KT< 1$, and that~\assrefLP{LP:1} holds.   
Assume also that there exists $\tau>0$ and $B>0$ such that:
    \begin{equation*}
        \rho(\mu) = \max_{d\in\{0,\dotsc,\KT\}}   R_2(d)   \leq B\mu^{-\tau}. 
\end{equation*}
Let $0<\gamma<1$ and $\Gamma>0$, and consider 
    \begin{equation*}
        \Delta(\mu) = 2 \exp\left( -\log^{\gamma} (\mu) \right) 
        \quad\text{and}\quad
        \varphi(\mu) = \log^{-\Gamma}(\mu).
    \end{equation*}
Then, for any $\mu$ larger than some constant $\mu_0$ depending on $B$, $\tau$, $\gamma$, $\Gamma$, $H_\KT$, $\beta_\KT$ and for  some positive constant $\mathfrak{f}$,
we have
\begin{equation*}
    \PP\left( |\widehat{\ST} - \ST| > \varphi(\mu) \right)
    \leq
    8(1+\KT)%
    \exp\left( -\mathfrak{f} \N0\varphi^2(\mu) 
    \big[\Delta(\mu)\big]^{4H_\KT} \right).
\end{equation*}
\end{theorem}

The three quantities $\rho(\mu) $, $\Delta(\mu) $ and $ \varphi(\mu)$ are required to decrease to zero, as $\mu$ tends to infinity, in such a way that $\rho(\mu)/\Delta(\mu) + \Delta(\mu)/\varphi(\mu)\rightarrow 0$.  We propose $\Gamma = 2$ and $\gamma = 1/2$. 
The choices of the rates for $\rho(\mu)$, $\Delta(\mu)$ and $ \varphi(\mu)$ satisfy some additional  requirements. First, it will be shown below that, in order to achieve 
optimal rates of convergence for the mean and covariance estimators, the local regularity has to be estimated with a concentration rate $\varphi(\mu)$ faster than  $\log^{-1}(\mu)$. This is a consequence of the identity $\mu^{1/\log(\mu)} = e$ for any $\mu>1$, and of a mild condition on $N$ and $\mu $, such as
\begin{equation}\label{mfka}
\limsup_{N,\mu \rightarrow \infty}  \{\log(N)/\log(\mu)\} <\infty.
\end{equation}
The  technical condition \eqref{mfka}  matches general  situations found in applications.
Second, we want to allow for reasonable rates of increase for $\N0$, the size of the learning set. In Theorem \ref{thm:estimation-alpha}, $\N0$ can increase as fast  as an arbitrary positive power of $\mu$. Third, since $\tau>0$ could be arbitrarily small, the rate imposed on the nonparametric estimators $\widetilde{\nabla^d X}$ of $\nabla^d X$ is a very mild requirement which could be achieved by the common estimators, with random or fixed design, under mild conditions, in particular on the distribution of the $M_i$ and the smoothing parameter. See, for instance, \cite{tsybakov2009} and \cite{BELLONI2015}. In particular, the required rate for the $\widetilde{\nabla^d X}$ can be obtained under general forms of heteroscedasticity.

\subsection{Proofs for the alternative local regularity estimator}

To prove Theorem~\ref{thm:estimation-alpha}, we state and prove some auxiliary lemmas where the following notations will be used:
\begin{equation*}
    \rhostar = \max_{d\in\{0,\dotsc,\KT\}} \left\{\left( R_2(d) \right)^{\frac{1}{4H_d}}\right\},
    \quad\text{and}\quad
    \deltastar = \Delta(\mu).
\end{equation*}
Remark that, for $\mu$ large enough, we have 
$\mathfrak{M} \max(\deltastar^{H_\KT}, \deltastar^{\beta_\KT})< \varphi(\mu)< 2$ and $A_1 \rhostar \leq \deltastar \leq A_2$, where:
\begin{equation*}
    A_1 = \max_{d\in\{0,\dotsc,\KT\}} \left[ \frac{2^{4H_d+3}}{L_d^2} \left( \sqrt\frac{\mathfrak{a}}{\mathfrak{A}}+1 \right) \right]^{\frac{1}{2H_d} },
    \;
    A_2 = \min_{d\in\{0,\dotsc,\KT\}}  \left[ \frac12\left( \frac{L_d}{S_d} \right)^2 \right]^{\frac{1}{H_{d+1}}} \wedge 1,
\end{equation*}
and $\mathfrak{M} = 4\max_{d\in\{0,\dotsc,\KT\}} \left({S_d}/{L_d}\right)^2$.


\begin{proof}[Proof of Lemma~\ref{lem:regularity-1}]
    Using Taylor's formula, there exists $\xi\in(s\wedge t, s\vee t)$ such that:
    \begin{align*}
        \EE\left[ |\nabla^d X_t - \nabla^d X_s|^{2} \right]
        &= (t-s)^2 \EE\left[\left(\nabla^{d+1} X_\xi\right)^2\right]\\
        &= (t-s)^2 \left\{ 
            L_{d}^2
            + 2E_1(d) 
            +  E_2(d)
            \right\},
    \end{align*}
    where
    \begin{align*}
        L_{d}^2 &= \EE\left[\left(\nabla^{d+1} X_{t_1}\right)^2\right]\\
        E_1(d) &= \EE\left[ \nabla^{d+1} X_{t_1} \left(\nabla^{d+1} X_\xi - \nabla^{d+1} X_{t_1}\right)\right]\\
        E_2(d) &= \EE\left[\left(\nabla^{d+1} X_\xi - \nabla^{d+1} X_{t_1}\right)^2\right].
    \end{align*}
    Remark that~\assrefHt{Ht:derivatives} implies that $\underline{a}_{d+1} < L_d^2 < \overline{a}_{d+1}$. Using the Cauchy-Schwartz inequality,  
    \begin{align*}
        |\EE\left[ (\nabla^d X_t - \nabla^d X_s)^{2} \right]) - L_d^2 (t-s)^2| 
        &\leq |2E_1(d)+E_2(d)| (t-s)^2\\
        &\leq \left(2L_d\sqrt{E_2(d)} + E_2(d)\right) (t-s)^2.
    \end{align*}
    Thus, it remains to bound $E_2(d)$. First, consider $d=\delta-1$. Then using~\assrefHt{Ht:equivalent} combined with the fact that $|\xi-t_1| \leq \Delta(\mu) \leq 1$, we have:
    \begin{equation*}
        E_2(\delta) 
        =  \EE\left[\left(\nabla^{\delta} X_\xi - \nabla^{\delta} X_{t_1}\right)^2\right]
        \leq (L_\delta^2+S_\delta^2) \big[\Delta(\mu)\big]^{2H_\delta}.
    \end{equation*}
    This implies that
    \begin{equation*}
        |2 E_1(\delta-1) + E_2(\delta-1)| 
        \leq S_{\delta-1}\big[\Delta(\mu)\big]^{H_\delta}
        \quad\text{with}\quad
        S_{\delta-1} = 2L_d\sqrt{L_\delta^2+S_\delta^2} + L_\delta^2 +S_\delta^2.
    \end{equation*}
    Next, consider the case of $d<\delta-1$. Using Taylor's formula and~\assrefHt{Ht:derivatives}, we have
    $$
        E_2(d) 
        = \EE\left[\left(\nabla^{d+1} X_\xi - \nabla^{d+1} X_{t_1}\right)^2\right]
        \leq \overline{a}_{d+2} (\xi-t_1)^2 \leq \overline{a}_{d+2} \Delta^2(\mu),
    $$
    which implies
    $|2 E_1(d) + E_2(d)| \leq S_{d} \Delta(\mu)$ with 
    $ S_{d} = 2L_d\sqrt{\overline{a}_{d+2}} + \overline{a}_{d+2}$. 
    Lemma~\ref{lem:regularity-1} is now proved. 
\end{proof}


\begin{lemma}\label{lem:regularity-3}
    Assume that the condition of Theorem~\ref{thm:estimation-alpha} hold true.  
    Assume also that $\deltastar^{2\beta_d} \leq (L_\KT/S_\KT)^2/2$.
    For  $d\in\{0,\dotsc,\KT\}$, define 
    \begin{equation*}
        \eta_*(d) = {4}\left(\sqrt{\frac{\mathfrak{a}}{\mathfrak{A}}}+1\right)
        \sqrt{R_2(d)}.
    \end{equation*}
    Let $s,t\in\Ostar$ such that $\theta_d(s,t) > \eta_*(d)$.  For any $\kappa > 0$, define
    \begin{equation*}
        p_d^+(s,t; \kappa) = \PP\left[\hat\theta_d(s,t)> (1+\kappa)\theta_d(s,t)\right]
    \end{equation*}
    and
    \begin{equation*}
        p_d^-(s,t;\kappa ) = \PP\left[\hat\theta_d(s,t)< (1-\kappa)\theta_d(s,t)\right].
    \end{equation*}
    There exists a  constant $\mathfrak{e}>0$ such that, for any $\kappa$ such that $\eta_*(d)<\kappa \theta_d(s,t)<1$, we have
    \begin{equation*}
        \max\big[p_d^+(s,t;\kappa ), p_d^-(s,t;\kappa )\big]
        \leq \exp\left( -\mathfrak{e}\N0 \kappa^2   \underline{a}_{d+1}^2  |t-s|^{4H_d}\right),
    \end{equation*}
    where $\underline{a}_{d+1}$  is defined in~\assrefHt{Ht:derivatives}, for $d<\KT+1$, and $\underline{a}_{\KT+1} = L^2_\KT/2 $.
\end{lemma}

\begin{proof}[Proof of Lemma~\ref{lem:regularity-3}]
    First, let us point out that, by the definition of the space $\mathcal{X}(\KT+H_\KT, L_\KT;\Ostar)$, the quantity   $\theta_d(s,t)$ could not be equal to zero for any $s,t$ in an open interval. Thus, the points $s,t\in\Ostar$ in the statement of Lemma \ref{lem:regularity-3} are well-defined. 

    Set $d\in\{0,\dotsc,\KT\}$ and $s,t\in\Ostar$. Let us decompose
    \begin{equation*}
        \hat\theta_d(s,t) - \theta_d(s,t)
        = \frac1{\N0} \sum_{n=1}^{\N0} \bar Z_n + \EE\left( \hat\theta_d(s,t)\right) - \theta_d(s,t) ,
    \end{equation*}
    where, for any $n=1,\dotsc,\N0$:
    \begin{equation*}
        \bar Z_n = Z_n - \EE(Z_n)
        \quad\text{with}\quad
        Z_n = \left(
        \widetilde{\nabla^d X}^{(n)}_t - \widetilde{\nabla^d X}^{(n)}_s
        \right)^2.
    \end{equation*}
    
    \textbf{Bounding the bias term.} 
    Note that
    \begin{multline*}
        \EE\left( \hat\theta_d(s,t)\right) -  \theta_d(s,t)
        = 2 
        \EE\left[
        (\xi_d(t)-\xi_d(s))
        \left({\nabla^d X}_t - {\nabla^d X}_s\right)
        \right]
        \\+ 
        \EE\left[(\xi_d(t)-\xi_d(s))^2\right]    .
    \end{multline*}
    Since $\EE(\xi_d(t)-\xi_d(s))^2 \leq 2R_2(d) $, using the Cauchy-Schwartz inequality and \assrefHt{Ht:moments}, we obtain:
    $$
    \left| \EE\left( \hat\theta_d(s,t)\right) -  \theta_d(s,t) \right|
    \leq
    2\sqrt{\frac{\mathfrak{a}}{\mathfrak{A}}R_2(d)} + 2R_2(d),
    $$
    Considering without loss of generality that $R_2(d)\leq 1$, we obtain
    \begin{equation*}
        \left| \EE\left( \hat\theta_d(s,t)\right) -  \theta_d(s,t) \right| \leq \eta_*(d)/2.
    \end{equation*}    
    
    \textbf{Moments of the stochastic term.} Let us note  that, for any $p\geq 1$,
    \begin{equation*}
        \EE\left( \left| \bar Z_n \right|^p \right)
        = \EE \left( \left| Z_n - \EE(Z_n)\right|^p \right) 
        \leq 2^p \EE\left( |Z_n|^p \right)
    \end{equation*}
    Moreover,
    \begin{align*}
        |Z_n|^p 
        &= \left| 
        \left({\nabla^d X}^{(n)}_t - {\nabla^d X}^{(n)}_s\right) 
        - 
        \left(\xi_d^{(n)}(t) - \xi_d^{(n)}(s)\right)
        \right|^{2p}\\
        &\leq \left(\left|{\nabla^d X}^{(n)}_t - {\nabla^d X}^{(n)}_s\right| 
        + 
        \left| \xi_d^{(n)}(t) \right|  + \left|\xi_d^{(n)}(s)\right| \right)^{2p}\\
        &\leq 3^{2p-1}\left\{  \left|{\nabla^d X}^{(n)}_t - {\nabla^d X}^{(n)}_s\right|^{2p} + \left|\xi_d^{(n)}(t)\right|^{2p} +  \left|\xi_d^{(n)}(s)\right|^{2p}\right\}.
    \end{align*}
    This implies that
    \begin{align*}
        \EE\left( \left| \bar Z_n \right|^p \right)
        &\leq 
        \frac{18^p}{3} 
        \left\{  
        \EE\left(\left|{\nabla^d X}^{(n)}_t - {\nabla^d X}^{(n)}_s\right|^{2p} \right)
        + 
        2 R_{2p}(d)
        \right\}\\
        &\leq
        {\frac{18^p}{3} \frac{p!}{2}
        \left\{\mathfrak{a} \mathfrak{A}^{p-2}
        +  4\mathfrak{c} \mathfrak{C}^{p-2}
        \right\}}\\
        &\leq { \frac{p!}{2} \mathfrak{d} \mathfrak{D}^{p-2}},
    \end{align*}
    {where $\mathfrak{d} = 108(\mathfrak{a}+4\mathfrak{c})$ and $\mathfrak{D}= 18\max(\mathfrak{A},\mathfrak{C})$.}
    The second line in the last display comes from \assrefHt{Ht:moments} and \assrefLP{LP:1}.
    
    \textbf{Exponential bounds.} Since $\eta_*{(d)}$ has the rate of $\sqrt{R_2(d)}$, we could consider 
    $\eta_*{(d)}< \eta<1$ and, using Bernstein's inequality, we obtain:
    \begin{align*}
        \PP\left(  \hat\theta_d(s,t) - \theta_d(s,t) > \eta \right)
        &\leq \PP\left(  \hat\theta_d(s,t) - \EE\left(\hat\theta_d(s,t)\right) > \eta/2 \right) \\
        &\leq \exp\left( -\frac{\N0 \eta^2}{8\mathfrak{d}+4\mathfrak{D}\eta} \right)\\
        &\leq \exp\left( -\mathfrak{e}\N0 \eta^2 \right),
    \end{align*}
    where $\mathfrak{e} = 1/(8\mathfrak{d}+4\mathfrak{D})$. 
    Since $\kappa\theta_d(s,t)>\eta_*(d)$, this quantity could replace $\eta$ in the above inequality. Hence:
    \begin{equation*}
        \PP\left(  \hat\theta_d(s,t) > (1+\kappa) \theta_d(s,t) \right)
        \leq \exp\left( -\mathfrak{e}\N0 \kappa^2 \theta_d^2(s,t)\right).
    \end{equation*}
    Assume first that $d<\KT-1$. Applying Talor's formula, there exists $\xi\in\Ostar$ such that $|\xi-s|\leq |t-s|$ and
    \begin{equation*}
        \theta_d(s,t) = (t-s)^2 \EE\left[ \left( \nabla^{d+1} X_\xi \right)^2 \right]
        \geq (t-s)^2 \underline{a}_{d+1} = (t-s)^{2H_d} \underline{a}_{d+1}.
    \end{equation*}
    The last inequality is a consequence of~\assrefHt{Ht:moments}. Assume now that $d=\KT$. Using~\assrefHt{Ht:equivalent}, we have:
    \begin{align*}
        \theta_\KT(s,t) 
        &\geq L_\KT^2|t-s|^{2H_\KT} - S_\KT^2|t-s|^{2H_\KT}\deltastar^{2\beta_d}\\
        &= |t-s|^{2H_\KT}\left( L_\KT^2 - S_\KT^2\deltastar^{2\beta_d}\right)\\
        &\geq \frac{L_\KT^2}2 |t-s|^{2H_\KT} = \underline{a}_{\KT+1}  |t-s|^{2H_\KT}.
    \end{align*}
    This implies that, for any $d\in\{0,\dotsc,\KT\}$:
    \begin{equation*}
        p_d^+(s,t;\kappa)
        \leq \exp\left( -\mathfrak{e}\N0 \kappa^2 \underline{a}_{d+1}^2  |t-s|^{4H_d}\right).
    \end{equation*}
    The same reasoning can be applied to bound the term $p_d^-(s,t;\kappa )$.
\end{proof}


\begin{lemma}\label{lem:regularity-4}
    Assume that the condition of Theorem~\ref{thm:estimation-alpha} hold true.  
    Let $d$ be an element of $\{0,\dotsc,\KT\}$.
    There then exists a positive constant $\mathfrak{f}_d$,  depending on $\mathfrak{a}$, $\mathfrak{A}$, $\mathfrak{c}$, $\mathfrak{C}$, $\underline{a}_{d+1}$ and $H_{d}$ such that, for any
    $\epsilon$ which satisfies
    \begin{equation*}
        4\left(\frac{S_{d}}{L_d}\right)^2   \deltastar^{H_{d+1}} < \epsilon \log(2) < 2,
    \end{equation*}
    the following inequality holds:
    \begin{equation*}
        \PP\left( |\hat H_d - H_d| > \epsilon \right)
        \leq
        4 \exp\left( -\mathfrak{f}_d \N0\epsilon^2 \deltastar^{4H_d} \right).
    \end{equation*}
\end{lemma}

\begin{proof}[Proof of Lemma~\ref{lem:regularity-4}]
    
    We first have to control the distance between $H_d$ and the proxy value $\tilde H_d$ defined in~\eqref{eq:tilde-Hd}. To do so, note that, for $k=2,3$, we have
    \begin{equation*}
        \theta_d(t_1, t_{k}) = L_{d+1}^2 |t_{k}-t_1|^{2H_d} (1+\rho_d(k)),
    \end{equation*}
    where, using~\assrefHt{Ht:equivalent} and Lemma~\ref{lem:regularity-1}
    \begin{equation*}
        |\rho_d(k)| \leq \left(\frac{S_{d}}{L_d}\right)^2 \Delta_{*}^{H_{d+1}} \leq \frac12 .
    \end{equation*}
    This implies 
    \begin{align*}
        |\tilde H_d - H_d| 
        &=  \left\vert\frac{\log(1+\rho_d(3) - \log(1+\rho_d(2)))}{2\log(2)}\right\vert\\
        &\leq \frac{|\rho_d(3) - \rho_d(2)|}{\log(2)}\\
        &\leq \frac{2}{\log(2)} \left(\frac{S_{d}}{L_d}\right)^2  \Delta_{*}^{H_{d+1}}\\
        &\leq \frac\epsilon2.
    \end{align*}
    We deduce that
    \begin{align*}
        \PP(|\hat H_d - H_d| > \epsilon)
        &\leq \PP(|\hat H_d - \tilde H_d| > \epsilon - |\tilde H_d - H_d|)\\
        &\leq  \PP(|\hat H_d - \tilde H_d| > \epsilon/2)\\
        &\leq \PP\left( \frac{\hat\theta_d(t_2,t_3)}{\theta_d(t_2,t_3)}\frac{\theta_d(t_1,t_2)}{\hat\theta_d(t_1,t_2)}  > 2^\epsilon \right)\\
        & \mkern180mu + \PP\left( \frac{\hat\theta_d(t_2,t_3)}{\theta_d(t_2,t_3)}\frac{\theta_d(t_1,t_2)}{\hat\theta_d(t_1,t_2)}  < 2^{-\epsilon} \right).
    \end{align*}
    By simple algebra and using the definition of the functions $p_d^+$ and $p_d^-$ introduced in Lemma~\ref{lem:regularity-3}, we obtain
    \begin{align*}
        \PP(|\hat H_d - H_d| > \epsilon)
        &\leq p^+_d(t_1,t_3; 2^{\epsilon/2}-1) + p^-_d(t_1,t_3; 1-2^{-\epsilon/2})\\
        &\qquad + p^+_d(t_1,t_2; 2^{\epsilon/2}-1) + p^-_d(t_1,t_2; 1-2^{-\epsilon/2}).
    \end{align*}
    Note that, using Lemma~\ref{lem:regularity-1} and~\assrefHt{Ht:equivalent}, we have for $k=2,3$: 
    \begin{align*}
        \theta_d(t_1,t_k) 
        &\geq \vert t_3-t_1\vert^{2H_d} \left(L_d^2 - S_d^2 \deltastar^{H_{d+1}}\right)\\
        &= \left(\frac{\deltastar}{2^{k-1}}\right)^{2H_d} \left(L_d^2 - S_d^2 \deltastar^{H_{d+1}}\right)\\
        &\geq \frac{L_d^2}{2}\left(\frac{\deltastar}{2^{k-1}}\right)^{2H_d} \\
        &\geq 4\left( \sqrt\frac{\mathfrak{a}}{\mathfrak{A}}+1 \right)\sqrt{R_2(d)} .
    \end{align*}
    Thus, Lemma~\ref{lem:regularity-3} can be used to write
    \begin{align*}
        p^+_d(t_1,t_k; 2^{\epsilon/2}-1) 
        &\leq \exp\left( -\mathfrak{e}\N0 (2^{\epsilon/2}-1)^2 \underline{a}_{d+1}^2  |t_k-t_1|^{4H_d}\right)\\
        &\leq \exp\left( -\frac{\mathfrak{e}\log^2(2)}{4}\N0 \epsilon^2 \underline{a}_{d+1}^2  \left(\frac{\deltastar}{2^{k-1}}\right)^{4H_d}\right)\\
        &\leq \exp\left( -\frac{\underline{a}_{d+1}^2\mathfrak{e}\log^2(2)}{2^{2+8H_d}}\N0 \epsilon^2   \deltastar^{4H_d}\right).
    \end{align*}
    The same reasoning can be applied to bound $p^-_d(t_1,t_k; 1-2^{-\epsilon/2})$. However, note that in this case $1-2^{-\epsilon/2} \leq \epsilon/4$. This implies that:
    \begin{align*}
        p^-_d(t_1,t_k; 1-2^{-\epsilon/2}) 
        &\leq \exp\left( -\frac{\underline{a}_{d+1}^2\mathfrak{e}\log^2(2)}{2^{4+8H_d}}\N0 \epsilon^2   \deltastar^{4H_d}\right).
    \end{align*}
    To complete the proof of Lemma \ref{lem:regularity-4}, if suffices to take $\mathfrak{f}_d = {\underline{a}_{d+1}^2\mathfrak{e}\log^2(2)}/{2^{4+8H_d}}$.
\end{proof}


\begin{proof}[Proof of Theorem~\ref{thm:estimation-alpha}]
    Note that:
    \begin{align*}
        \PP(|\hST - \ST| > \varphi(\mu))
        &\leq \PP(|\hST - \ST| > \varphi(\mu), \hKT = \KT) + \PP(\hKT \neq \KT) \\
        &\leq \PP(|\hat H_\KT- H_\KT| >\varphi(\mu)) +  \PP(\hKT < \KT) + \PP(\hKT > \KT) \\
        &\leq \PP(|\hat H_\KT- H_\KT| >\varphi(\mu)) +  
        \sum_{d=0}^{\KT-1}\PP(\hat H_d  <  1- \varphi(\mu)) \\& \mkern266mu + \PP(\hat H_\KT  >  1- \varphi(\mu)).
    \end{align*}
    Now, recall that, for $d<\KT$ we have $H_d = 1$. Note also that $H_\KT < 1$. This implies that:
    \begin{align*}
        \PP(|\hST - \ST| > \varphi(\mu))
        &\leq \PP(|\hat H_\KT- H_\KT| >\varphi(\mu)) \\
        &+  
        \sum_{d=0}^{\KT-1}\PP(|\hat H_d  -  H_d| > \varphi(\mu)) + \PP(|\hat H_\KT - H_\KT| > 1- H_\KT).
    \end{align*}
    Since $1-H_\KT>\varphi(\mu)$ for $\mu$ sufficiently large, such that $\varphi(\mu)$ could replace $\epsilon$ in Lemma \ref{lem:regularity-4}, we have:
    \begin{equation*}
        \PP(|\hST - \ST| > \varphi(\mu))
        \leq 
        8(1+\KT)
        \exp\left( -\mathfrak{f} \N0\varphi^2(\mu) \Delta_*^{4H_\KT} \right).
    \end{equation*}
    The Theorem \ref{thm:estimation-alpha} is proved.
\end{proof}

%% file: supplement/technical_lemma.tex
\section{Technical lemmas}\label{sec:tech_lem}


\begin{proof}[Proof of Lemma \ref{lem:Tl-Tk_main}]
Let $ \mathcal{C} = \{ M \geq K_0\} \setminus\mathcal B$. We have
\begin{equation*}
    \EEB\left[\big|T_{(l)}-T_{(k)}\big|^{\alpha}\right]
    =
    \EE\left[\big|T_{(l)}-T_{(k)}\big|^{\alpha}\1_{\mathcal{B}}\right]
    = (I) - (II),
\end{equation*}
where
$$
    (I) = \EE\left[\big|T_{(l)}-T_{(k)}\big|^{\alpha}\1_{M\geq K_0}\right]\quad \text{ and } \quad 
    (II) = \EE\left[\big|T_{(l)}-T_{(k)}\big|^{\alpha}\1_{\mathcal{C}}\right].
$$
We study separately the two terms of the right hand side of the above equation. 

\textbf{Study of $(II)$.} Note that
\begin{equation*}
    \EE\left[\big|T_{(l)}-T_{(k)}\big|^{\alpha}\1_{\mathcal{C}}\right]
    \leq |I|^{\alpha}  \PP(\mathcal{C}).       
\end{equation*}
The event $\mathcal{C}$ happens if, less than $K_0$ random times among $T_1,\dotsc, T_M$ fall into the interval $J_\mu(\T)$. This implies that
\begin{equation*}
    \PP(\mathcal{C})
    \leq \EE\left[ \PP(B_M < K_0 \mid M) \1_{\{M\geq K_0\}}\right]
\end{equation*}
where, for any integer $m\geq 1$, $B_m$ denotes a Binomial random variable defined as $\mathcal B (m, |J_\mu(\T)|/|I|)$, independent of $M$.
Using the Bernstein inequality, we obtain
\begin{equation*}
    \PP(B_M < K_0 \mid M)
    \leq 
    \exp\left( -\frac{2|J_\mu(\T)|}{|I|}M + 2K_0 \right).
\end{equation*}
Since $|J_\mu(\T)|/|I|\leq (\log(\mu))^{-1}$ and $K_0 \leq (2\log(\mu))^{-1}\mu$, we obtain
\begin{align*}
    \PP(\mathcal{C})
    &\leq \exp\left( -\frac{2|J_\mu(\T)|}{|I|}\mu + 2K_0 \right)
    \EE\left[ \exp\left( -\frac{2|J_\mu(\T)|}{|I|}(M-\mu) \right) \1_{\{M\geq K_0\}}\right]\\
 &\leq   \exp\left( -\frac{\mu}{\log(\mu)} \right)
    \EE\left[ \exp\left( -\frac{2|J_\mu(\T)|}{|I|}(M-\mu) \right)\1_{\{M\geq K_0\}} \right].
\end{align*}
To bound the last expectation, let $0<\epsilon<1$ be some real number. Then
\begin{align*}
    & \exp\left( -\frac{\mu}{\log(\mu)} \right)
    \EE\left[ \exp\left( -\frac{2|J_\mu(\T)|}{|I|}(M-\mu) \right)\1_{\{M\geq K_0\}} \right]\\
    &\;\leq \exp\left( -\frac{\mu}{\log(\mu)} \right)
    \left\{\exp\left( \frac{2|J_\mu(\T)|}{|I|}\mu\epsilon \right)\right. \\
    &\mkern180mu\left. + \EE\left[ \exp\left( -\frac{2|J_\mu(\T)|}{|I|}(M-\mu) \right)\1_{\{K_0\leq M\leq \mu - \mu\epsilon\}} \right]\right\}\\
    &\; \leq \exp\left( -\frac{\mu}{\log(\mu)} \right)
    \left\{\exp\left( \frac{4\mu\epsilon}{\log(\mu)} \right)+\exp\left(\frac{4\mu}{\log(\mu)}\right)\PP\left[ |M-\mu| >\mu\epsilon\right]\right\}\\
    &\; \leq \exp\left( -\frac{\mu}{\log(\mu)} \right)
    \left\{\exp\left( \frac{4\mu\epsilon}{\log(\mu)} \right)+\exp\left(\frac{2\mu}{\log(\mu)}\right)\exp(-\gamma_0\mu\epsilon)\right\}.
\end{align*}
This implies that:
\begin{align*}
    \PP(\mathcal{C})
    &\leq \exp\left[ -\frac{\mu}{\log(\mu)} 
    \left(
    1 -4\epsilon
    \right)\right]
    +
    \exp\left[ - \mu\epsilon 
    \left(
    \gamma_0 -
    \frac{1}{\epsilon\log(\mu)}
    \right)\right].
\end{align*}
Taking $\epsilon = 1/8$, we obtain, for sufficiently large $\mu$~:
\begin{align*}
    \PP(\mathcal{C})
    &\leq 2\exp\left[ -\frac{\mu}{2\log(\mu)} \right].
\end{align*}
We finally obtain, for sufficiently large $\mu$, 
\begin{equation}\label{eq:II-final}
    (II)= \EE\left[\big|T_{(l)}-T_{(k)}\big|^{\alpha}\1_{\mathcal{C}}\right]
    \leq 2|I|^{\alpha} \exp\left[ -\frac{\mu}{2\log(\mu)} \right].  
\end{equation}

\textbf{Study of $(I)$.} 
We define the random variable $\rho$ by the equation:
\begin{equation*}
    \EE\left[\big|T_{(l)}-T_{(k)}\big|^{\alpha} \mid M\right]
    =
    \left(
    \frac{l-k}{f(\T)(M+1)}
    \right)^{\alpha}
    \big( 
    1+\rho 
    \big)
\end{equation*}
Using Lemma \ref{lem:moment-spacing}, we have, almost surely:
\begin{equation*}
    |\rho| \leq \mathfrak{c}_0 
    \left\{
    \frac1M + \frac1{\mathfrak{s}r} + 
    \frac1{M\mathfrak{s}r} +
    \left(\frac{\mathfrak{s}r}{M+1}\right)^{\beta_f \alpha /4}
    \right\}.
\end{equation*}
Whenever $8 r\geq  (\mu+1)^{\beta_f \alpha/(4+\beta_f \alpha)}$, by bounding smaller terms by the dominant ones and balancing the dominant terms on the right hand side of the last inequality, we have for $\mu$ large enough:
\begin{equation}\label{jan7b}
    |\rho| \leq 3\mathfrak{c}_0 
    \left(\frac{\mathfrak{s}r}{M+1}\right)^{\beta_f \alpha /4}
    + \mathfrak{c}_0 \left(\frac{\mathfrak{s}r}{\mu+1}\right)^{\beta_f \alpha /4}.
\end{equation}
On the other hand we have
\begin{align}
    \EE\left[\big|T_{(l)}-T_{(k)}\big|^{\alpha}\1_{M\geq K_0}\right]
    &=
    \EE\left(\EE\left[\big|T_{(l)}-T_{(k)}\big|^{\alpha} \mid M\right]\1_{M\geq K_0}\right)\nonumber\\
    &\mkern-18mu=\EE\left[
    \left(
    \frac{l-k}{f(\T)(M+1)}
    \right)^{\alpha}
    \big( 
    1+\rho 
    \big)
    \1_{M\geq K_0}
    \right]\nonumber\\
    &\mkern-18mu= \left(
    \frac{l-k}{f(\T)(\mu+1)}
    \right)^{\alpha}
    \EE\left[
    \left(\frac{\mu+1}{M+1}\right)^{\alpha}
    \big( 1+\rho \big)
    \1_{M\geq K_0}
    \right].\label{eq:I-0}
\end{align}
Now, define: 
\begin{equation}\label{eq:value-of-t}
    t = \frac{\left(\log(\mu+1)\right)^{2}}{2} \leq (\mu+1)/2,   
\end{equation}
and consider the following decomposition:
\begin{align*}
    \EE\left[
    \left(\frac{\mu+1}{M+1}\right)^{\alpha}
    \big( 1+\rho \big)
    \1_{M\geq K_0}
    \right]
    &= \EE\left[
    \left(\frac{\mu+1}{M+1}\right)^{\alpha}
    \big( 1+\rho \big)
    \1_{M\geq K_0}\1_{|M-\mu| \leq t}
    \right]\\
    &\qquad
    +
    \EE\left[
    \left(\frac{\mu+1}{M+1}\right)^{\alpha}
    \big( 1+\rho \big)
    \1_{M\geq K_0}\1_{|M-\mu | > t}
    \right].
\end{align*}
Using Assumption~\assrefH{ass:M}, combined with the fact that $r\leq \mu$, the term of the right hand side can be roughly bounded as follows:
\begin{multline*}
    \EE\left[
    \left(\frac{\mu+1}{M+1}\right)^{\alpha}
    \big( 1+\rho \big)
    \1_{M\geq K_0}\1_{|M-\mu| > t}
    \right]\\
    \leq 4\mathfrak{c}_0 (\mu+1)^{\alpha(1+\alpha \beta_f /4)} 
    \PP\left(|M-\mu | > t\right)\\
    \leq 4\mathfrak{c}_0 (\mu+1)^{\alpha(1+\alpha \beta_f /4)} 
    \exp\left(-\frac{\gamma_0}2 \left(\log(\mu+1)\right)^{2}\right).
\end{multline*}
Thus, for sufficiently large  $\mu$, we have:
\begin{equation}\label{eq:I-1}
    \EE\left[
    \left(\frac{\mu+1}{M+1}\right)^{\alpha}
    \big( 1+\rho \big)
    \1_{M\geq K_0}\1_{|M-\mu | > t}
    \right]
    \leq \frac{1}{\mu+1}.
\end{equation}
It remains to study the term
\begin{equation*}
    \EE\left[
    \left(\frac{\mu+1}{M+1}\right)^{\alpha}
    \big( 1+\rho \big)
    \1_{M\geq K_0}\1_{|M-\mu| \leq t}
    \right].
\end{equation*}
To do so, let us define
\begin{equation*}
    \tilde\rho_\alpha = 
    \left( \frac{\mu+1}{M+1} \right)^{\alpha} - 1. 
\end{equation*}
Since $K_0 < (\mu+1)/2$, we have 
\begin{align}
    \left(\frac{\mu+1}{M+1}\right)^{\alpha}
    \big( 1+\rho \big)
    \1_{M\geq K_0}\1_{|M-\mu| \leq t}
    &=
    \big( 1+\tilde\rho_\alpha \big)
    \big( 1+\rho \big)
    \1_{|M-\mu| \leq t}\nonumber\\
    &\mkern-84mu= 1 +
    \big(
    \tilde\rho_\alpha + 
    \rho + 
    \tilde\rho_\alpha\rho
    \big)
    \1_{|M-\mu| \leq t}
    - \1_{|M-\mu| > t}.\label{eq:I-2-a}
\end{align}
Under the event $\{|M-\mu \leq t|\}$, since $t< (\mu+1)/2$, we have:
\begin{equation*}
    1-\frac{\alpha t}{\mu+1} \leq \left( \frac{\mu+1}{M+1} \right)^{\alpha} \leq 1+\frac{2(2^{\alpha}-1)t}{\mu+1},
\end{equation*}
which leads to
\begin{equation}\label{eq:control-tilde-rho}
    \left| \tilde\rho_\alpha \right|
    \1_{\{|M-\mu \leq t|\}}
    \leq \frac{2(2^{\alpha}-1)t}{\mu+1}
    = (2^{\alpha}-1)\frac{\log^2(\mu+1)}{\mu+1}.
\end{equation}
Note also that by~\eqref{jan7b}:
\begin{equation}\label{eq:control-rho}
    |\rho| \1_{|M-\mu| \leq t}
    \leq 4\mathfrak{c}_0 
    \left(\frac{2\mathfrak{s}r}{\mu+1}\right)^{\beta_f \alpha/4}.
\end{equation}
Gathering~\eqref{eq:I-2-a}, \eqref{eq:control-tilde-rho} and~\eqref{eq:control-rho} we obtain, for sufficiently large  $\mu$~:
\begin{align}
    \left|
    \EE\left[
    \left(\frac{\mu+1}{M+1}\right)^{\alpha}
    \big( 1+\rho \big)
    \1_{M\geq K_0}\1_{|M-\mu| \leq t}
    \right] - 1
    \right|\\
    &\mkern-72mu\leq
    5\mathfrak{c}_0 \left( \frac{2\mathfrak{s}r}{\mu+1} \right)^{\beta_f \alpha/4}
    + \PP(|M-\mu|>t)\nonumber\\
    &\mkern-72mu\leq 6\mathfrak{c}_0 \left( \frac{2\mathfrak{s}r}{\mu+1} \right)^{\beta_f \alpha/4}. \label{eq:I-2}
\end{align}
Combining~\eqref{eq:I-0} with~\eqref{eq:I-1} and~\eqref{eq:I-2}, we obtain, for sufficiently large  $\mu$~:
\begin{equation}\label{eq:I-final}
    \EE\left[
    \big|T_{(l)}-T_{(k)}\big|^{\alpha}\1_{M\geq K_0}
    \right]
    =
    \left(
    \frac{l-k}{f(\T)(\mu+1)}
    \right)^{\alpha}
    \left(1+\tilde R\right),
\end{equation}
where
\begin{equation*}
    |\tilde R| 
    \leq
    7\mathfrak{c}_0 \left( \frac{2\mathfrak{s}r}{\mu+1} \right)^{\beta_f \alpha/4}.
\end{equation*}
From \eqref{eq:II-final} and~\eqref{eq:I-final}, we obtain, for $\mu$ large enough:
\begin{equation*}
    \EEB\left[\big|T_{(l)}-T_{(k)}\big|^{\alpha}\right]
    =
    \left(
    \frac{l-k}{f(\T)(\mu+1)}
    \right)^{\alpha}
    \left(1+R\right),
\end{equation*}
where
\begin{equation*}
    |R| 
    \leq
    8\mathfrak{c}_0 \left( \frac{2\mathfrak{s}r}{\mu+1} \right)^{\beta_f \alpha/4}
    =
    8\mathfrak{c}_0 (2f(\T))^{\beta_f \alpha/4}
    \left(
    \frac{l-k}{f(\T)(\mu+1)}
    \right)^{\beta_f \alpha/4} .
\end{equation*}
This ends the proof.
\end{proof}

Let us recall the definitions
\begin{equation}\label{eq:Anstar_SM}
  A = A_{M,h}=  \frac{1}{Mh} \sum_{m=1}^{M} U\left( \frac{T_m-\T}{h} \right)U^\top\left( \frac{T_m-\T}{h} \right)K\left( \frac{T_m-\T}{h} \right),
 \end{equation}
  and
  $$
     \mathbf{A} = f(\T) \int_\RR U(u) U^\top(u) K(u) du,
  $$
  with $U(u) = (1, u, \dotsc, u^{ \hKT}/\hKT!)$. Moreover, $\lambda$ and $\lambda_0$ are the smallest eigenvalues of $A$ and $\mathbf{A}$, respectively.   The matrix $\mathbf{A}$ is positive definite and thus $\lambda_0>0$. See \cite{tsybakov2009}. The following result shows that, with high probability, $\lambda$ stays away from zero. Let us recall that in our context, $\hKT$ is a generic estimator of $\KT$, independent of the online set of curves. Since dimension of the matrices $A$ and $\mathbf A$ are given by this estimator, the probability $\PP(\cdot)$ in  Lemma \ref{lem:eigenval}  should be understood as the conditional probability given the estimator $\hKT$. Finally, recall that 
  \begin{equation*}
   \widehat h = \left(\frac{1}{M}\right)^{1/(2\hST +1)}.
\end{equation*}


\begin{proof}[Proof of Lemma~\ref{lem:eigenval}]
 Without loss of generality, we could work on the set $\{\hKT = \KT\}$.  Moreover, for simplicity, we write $h$ instead of $\widehat h$ below in this proof. 
 
Note that, using Assumption~\assrefLP{ass:M2}, for any $1\leq i \leq j \leq \degree$, the element  $A_{i,j}$ tends almost surely to the element $\mathbf{A}_{i,j}$ as $\mu$ goes to infinity. This implies that the matrix $A$ tends to the matrix $\mathbf{A}$. This also implies that, for sufficiently large  $\mu$, we have $\lambda>0$. More precisely, we have:
    \begin{align*}
        |\lambda - \lambda_0| \leq \|A - \mathbf{A}\|_2 \leq 
       (\KT+1)  \|A -\mathbf{A}\|_{\infty},
    \end{align*}
    where $\|\cdot\|_2$ denotes the norm induced by the Euclidean norm whereas $\|\cdot\|_{\infty}$ denotes the entrywise sup-norm. 
    Let $\tilde\PP(\cdot)$ and $\tilde\EE(\cdot)$ denote the conditional probability $\PP(\cdot | M)$ and conditional expectation $\EE(\cdot | M)$, respectively. Then:
    $$
        \tilde\PP(\lambda \leq \beta)
        \leq \tilde\PP(|\lambda - \lambda_0| \geq \lambda_0/2)
        \leq \sum_{0\leq i,j\leq \degree}\tilde\PP\left(|(A_n)_{i,j}- \mathbf{A}_{i,j}| \geq \lambda_0/\{2(\KT+1) \} \right).
   $$
    Next, we decompose  
    \begin{equation*}
        A_{i,j}- \mathbf{A}_{i,j} = A_{i,j}- \tilde\EE(A_{i,j})
        +\tilde\EE(A_{i,j}) -\mathbf{A}_{i,j}.
    \end{equation*}
   Using Assumption~\assrefH{ass:T} and the fact that $K(\cdot)$ has the support $[-1,1]$, we have:
    \begin{align*}
        \left|\tilde\EE(A_{i,j})-\mathbf{A}_{i,j}\right|
        &\leq \left|\int_\RR \left[U(u) U^\top(u)\right]_{i,j} K(u) \big\{f(\T+hu) - f(\T)\big\}du\right|\\
        &= L_f h^{\beta_f}\int_\RR  \left|\left[U(u) U^\top(u)\right]_{i,j}u K(u)\right| du\\
        &\leq L_f h^{\beta_f}\int_\RR  \left|u \right| K(u)du\\
        &\eqqcolon L_f\|K\|_1 h^{\beta_f}. 
    \end{align*}
    This implies that, for $h$ sufficiently small, that is for $M$ sufficiently large, 
    \begin{align*}
        \tilde\PP(\lambda \leq \beta)
        &\leq 
        \sum_{0\leq i,j\leq \degree}\tilde\PP(|A_{i,j}- \tilde\EE(A_{i,j})| \geq  \lambda_0/\{2(\KT+1) \} - L_f\|K\|_1 h^{\beta_f})\\
        &\leq \sum_{0\leq i,j\leq \degree}\tilde\PP(|A_{i,j}- \tilde\EE(A_{i,j})|  > \lambda_0/\{4(\KT+1) \}).
    \end{align*}
    Let us define
    \begin{align*}
        \xi_{m,i,j}
        &= \left[
        U\left( \frac{T_m - \T}{h} \right)
        U^\top\left( \frac{T_m - \T}{h} \right)
        \right]_{i,j} 
        K\left( \frac{T_m - \T}{h} \right)\\
        &= \frac1{i!j!}\left( \frac{T_m-\T}{h} \right)^{i+j}  
        K\left( \frac{T_m - \T}{h} \right).
    \end{align*}
    By property~\eqref{eq:bounded-kernel}, we have 
    \begin{equation*}
        \left|\xi_{m,i,j} - \tilde\EE(\xi_{m,i,j})\right|
        \leq 2 \kappa.
    \end{equation*}
    Moreover, for $h$ sufficiently small, that is for $M$ sufficiently large, $f(t)\leq 2 f(\T)$, $\forall |t-\T|\leq h$, and thus 
    \begin{align*}
        \sum_{m=1}^M \widetilde{\Var}(\xi_{i,j}^{(m)})
        &\leq  \sum_{m=1}^M \tilde{\EE}\left[\{\xi_{i,j}^{(m)}\}^2\right]\\
        &\leq 
        2 f(\T) M h  \int_\RR 
       \left| \left[U(u)U^\top(u)\right]_{i,j} \right| K^2(u) du\\
        &\leq 
       2 f(\T)  \|K\|_2^2 \, Mh .
    \end{align*}
    Applying the Bernstein inequality \citep[see][p.~95]{MR595165}, we obtain, for any $x>0$:
    \begin{equation*}
        \tilde\PP\left(
        \frac{1}{Mh} \sum_{m=1}^M 
        \left|\xi_{m,i,j} - \tilde\EE(\xi_{m,i,j})\right| > x
        \right)
        \leq  2 \exp\left( -\frac{M^2x^2}{\frac{2\|f\|_\infty\|K\|_2^2M}{h}+\frac{4\kappa xM}{3h}} \right).
    \end{equation*}
    Then equation~\eqref{eq:E-beta} follows if we define:
    \begin{equation*}
        \mathfrak{g} = \psi\left(\frac{\lambda_0}{4(\KT+1)
        }\right)
        \quad\text{with}\quad
        \psi(x) = \frac{x^2}{2\|f\|_\infty\|K\|_2^2+\frac{4\kappa x}{3}}.
    \end{equation*}

    It remains to prove \eqref{eq:E2-beta}. Let us define the events
    \begin{equation}\label{eq:events-EF_SM}
    \mathcal{E}_\beta= \{\lambda > \beta\}  \quad \text{and } \quad \mathcal{F} = \big\{|\hHT - \HT| \leq \log^{-2}(\mu)\big\}\cap \left\{\hKT = \KT \right\}.
\end{equation}
    Using~\eqref{eq:E-beta}, we have
    \begin{align*}
        \PP(\overline{\mathcal{E}}_\beta)
        &= 2\EE[\exp(-\mathfrak{g} Mh)\1_\mathcal{F}] + \PP(\overline{\mathcal{F}})\\
        &\leq 2\EE[\exp(-\mathfrak{g} Mh)\1_\mathcal{F}] + 2\mathfrak{K}_1\exp(-\mu).
    \end{align*}
    The last line comes from Assumption~\assrefLP{ass:hatHt}. Note that under $\mathcal{F}$ 
    \begin{equation*}
        Mh = M^{\frac{2 \{\hKT+ \hHT\}}{2\{\hKT+ \hHT\}+1}} = M^{\frac{2 \{\KT+ \hHT\}}{2\{\KT+ \hHT\}+1}}
        = M^{\frac{2 \{\KT+ \HT\}}{2 \{\KT+ \HT\}+1}+\eta},
    \end{equation*}
    with
    \begin{equation*}
        |\eta| = \left|\frac{2(\hHT - \HT)}{(2\{\hKT+ \hHT\}+1)(2 \{\KT+ \HT\}+1)}\right|
        \leq 2|\hHT - \HT|.
    \end{equation*}
    Assumption~\assrefLP {ass:M2} implies that, under $\mathcal{F}$ and, for sufficiently large  $\mu$,
    \begin{equation}\label{eq:hHT-HT}
        Mh \geq
        \left( \frac{\mu}{\log(\mu)} \right)^{\frac{2 \{\KT+ \HT\}}{2 \{\KT+ \HT\}+1}-\frac{2}{\log^2(\mu)}}
        \geq \frac12 \left( \frac{\mu}{\log(\mu)} \right)^{\frac{2\{\KT+ \HT\}}{2\{\KT+ \HT\}+1}}.
    \end{equation}
    Thus, we have
    \begin{align*}
        \PP(\overline{\mathcal{E}}_\beta)
        &\leq 2 \exp\left(-\frac{\mathfrak{g}}{2}
       \tau(\mu)\log^2(\mu)\right)
         +2 \mathfrak{K}_1\exp(-\mu)\\
        &\leq \mathfrak{K}_2 \exp\left[- \frac{\mathfrak{g}}{2}\tau(\mu)\log^2(\mu)\right],
    \end{align*}
    for some positive constant $\mathfrak{K}_2$, that does not depend on $0<\beta\leq\lambda_0/2$.
    \end{proof}
        
    \begin{lemma}\label{lem:technical-exp_SM}
        Let $\xi$ be a positive random variable such that
        \begin{equation*}
            c_1 \coloneqq \EE\left[ \exp\left( \eta_0 \xi^4 \right) \right]  < \infty,
        \end{equation*}
        for some positive constant $\eta_0$. Then, for any $\tau\geq 1$:
        \begin{equation*}
            \EE\left[ \exp\left( \tau\xi \right) \right]
            \leq
            c_1 \exp\left(c_2\tau^{4/3}\right)
            \qquad\text{where}\qquad
            c_2 = \left( \frac{5}{16\eta_0} \right)^{1/3}.
        \end{equation*}
    \end{lemma}
    
    \begin{proof}[Proof of Lemma \ref{lem:technical-exp_SM}]
    Defining $\zeta = (16\eta_0/5)^{1/4}\xi$, we can assume, without loss of generality that $\eta_0=5/16$.
    Let $\gamma\geq \tau$.
    Remark that, since
    \begin{equation*}
        1-\frac{\tau\xi}{\gamma}
        = \left( 1-\frac{\tau\xi}{4\gamma} \right)^4
        - 6\left( \frac{\tau\xi}{4\gamma} \right)^2
        + 4\left( \frac{\tau\xi}{4\gamma} \right)^3
        -\left( \frac{\tau\xi}{4\gamma} \right)^4,
    \end{equation*}
    we obtain:
    \begin{align*}
        \EE\left[ \exp\left( \tau\xi \right) \right]
        &= \exp(\gamma) \EE\left[ \exp\left( -\gamma \left(1-\frac{\tau\xi}{\gamma} \right) \right)\right] \\
        &\leq \exp(\gamma) \EE\left[ 
            \exp\left( \frac{3\gamma}{8}\left( \frac{\tau\xi}{\gamma} \right)^2 \right)
            \exp\left( \frac{\gamma}{256}\left( \frac{\tau\xi}{\gamma} \right)^4 \right)\right]\\
        &\leq 
        \exp(\gamma+\eta) 
        \EE\left[ 
            \exp\left( -\eta \left( 1- \frac{3\gamma}{8\eta}\left( \frac{\tau\xi}{\gamma} \right)^2 \right)  \right)
            \exp\left( \frac{\gamma}{256}\left( \frac{\tau\xi}{\gamma} \right)^4 \right)
        \right].
    \end{align*}
    Using the fact that
    \begin{equation*}
        1- \frac{3\gamma}{8\eta}\left( \frac{\tau\xi}{\gamma} \right)^2
        = \left[ 1- \frac{3\gamma}{16\eta}\left( \frac{\tau\xi}{\gamma} \right)^2 \right]^2 - \left[\frac{3\gamma}{16\eta}\left( \frac{\tau\xi}{\gamma} \right)^2 \right]^2,
    \end{equation*}
    we obtain:
    \begin{equation*}
        \EE\left[ \exp\left( \tau\xi \right) \right]
        \leq 
        \exp(\gamma+\eta) 
        \EE\left[ 
            \exp\left( \frac{9}{256}\frac{\tau^4\xi^4}{\eta\gamma^2} + \frac{1}{256} \frac{\tau^4\xi^4}{\gamma^3}\right)
        \right].
    \end{equation*}
    Taking $\gamma=\eta=\tau^{4/3}/2$, we obtain:
    \begin{equation*}
        \EE\left[ \exp\left( \tau\xi \right) \right]
        \leq
        \exp(\tau^{4/3}) \EE\left[ \exp\left(\frac{5}{16} \xi^4\right)\right].
    \end{equation*}
    This completes the proof.
    \end{proof}

%% file: supplement/moment_bound.tex
    \section{Moment bounds for spacings}\label{mom_spac} 
    
    We need to find an accurate approximation for moments like 
    $$
    \mathbb{E}[(T_{(k)} - T_{(l)} )^\alpha\mid M =m],
    $$
    where $1\leq l<k\leq K_0\leq m$, $\alpha >0$. 
    Here, $T_{(1)}\leq\dotsc\leq T_{(K_0)}$ are defined as in Section \ref{sec:local-regularity}, that is the subvector of the $K_0$ closest values to $\T$. 
    We assume that $T$ admits a density $f$. Such moments will be considered with 
    $k$ and $l$ such that, for some fixed value $t_0\in [0,1]$ such that $f(t_0)>0$, 
    \begin{equation}\label{context}
        \frac{\max(\left|\lfloor t_0m\rfloor-k\right| , \left|\lfloor t_0m\rfloor -l\right| )}{m+1}  \leq 8 \frac{k-l}{m+1}
    \end{equation}
    and
    \begin{equation}\label{context2}
        \frac{k-l}{m+1} \quad \text{is small} ,
    \end{equation}
    and converges to zero when $m\rightarrow \infty. $ Herein, for any real number $a$, $\lfloor a \rfloor$ denotes the largest integer smaller than or equal to $a$. These conditions on $k$ and $l$ allows for $(k-l)$  increasing slower than $m$.
    
    Let us point out that $T_{(1)}\leq\dotsc\leq T_{(K_0)}$ defined in section \ref{sec:local-regularity} is not the order statistics from a random sample of $T$. In fact,  $T_{(k)}$, with $1\leq k \leq K_0$, is the  $(G+k)-$th order statistics of the sample $T_1,\ldots,T_m$. Here $G$ is a random variable and its value is determined by the way the subvector of $K_0$ closest values to $\T$ is built. It is important to notice that  $G$ 
    depends of the smallest and the largest values in this subvector, but is independent of the other components of the subvector.  
    In particular, this means that in the case where $T$ has a uniform distribution, the law of the spacings between  $T_{(1)}\leq\dotsc\leq T_{(K_0)}$ coincides with the law of the same type of spacings between the order statistics of a uniform sample of size $m$ on $[0,1]$. In particular, in the uniform case, the law of $T_{(k)} - T_{(l)} $ depends only on $m$ and $k-l$. For this reason, first we consider the case of $T$ with uniform law. In the general case, we use the transformation by the distribution function in order to get back to the uniform case.

    \subsection{The uniform case}
    
    Consider $U$ a uniform random variable on $[0,1]$. Let $U_1,\ldots,U_m$ be an independent sample of $U$ and let 
    $U_{(1)},\ldots,U_{(m)}$ be the order statistics. In the case of a uniform sample,  $U_{(k)} - U_{(l)}$ and $U_{(k-l)} $ have the same distribution, that is a beta distribution Beta$(k-l,m-(k-l)+1)$. Hence in this case, it is equivalent to study the moments of $U_{(r)} $ with $1\leq r = k-l  \leq m-1$. The variable $U_{(r)} $ has a Beta$(r,m-r+1)$ distribution. It also worthwhile to notice that $U_{(k)} - U_{(l)}$ and $U_{(l)}$ are independent, and the same is true for $U_{(k)} - U_{(l)}$ and $U_{(k)}$.
    
    By elementary calculations, we have 
    $$
    \mathbb{E}\left[ U_{(r)} ^\alpha  \right] = \frac{ B(\alpha+r,m-r+1)}{B(r,m-r+1)} = \frac{\Gamma (\alpha+r) }{\Gamma (r)} \frac{\Gamma (m+1) }{\Gamma (m+\alpha +1)},
    $$
    where $B(\cdot,\cdot)$ denotes the beta function and $\Gamma(\cdot)$ the gamma function. To derive the bounds for the moments of interest, we use some existing results on the approximation of the gamma functions and the ratios of the gamma functions. The results are recalled in Section \ref{wendel_app} below. 
    
    
    Let $M$ be a random variable taking positive integer values. In the following proposition we assume that, given the realization of $M\geq K_0$, $T_1 , \ldots, T_M$ be an independent sample with uniform distribution on $[0,1]$.  
    
    \begin{lemma}\label{firt1}
        Consider $0 < \alpha \leq 3$ and $1\leq l<k\leq m$, and let $r=k-l$. Then, for any $m\geq K_0$ in the support of $M$, 
        $$
        \left| \mathbb{E}\left[ (T_{(k)} - T_{(l)} )^\alpha \mid M =m \right] -  \frac{\Gamma (\alpha+r) }{\Gamma (r)} \frac{1}{(m+1)^\alpha}\right| \leq \frac{3}{m} \frac{\Gamma (\alpha+r) }{\Gamma (r)} \frac{1}{(m+1)^\alpha} ,
        $$
        and
        $$
        \left| \mathbb{E}\left[ (T_{(k)} - T_{(l)} )^\alpha \mid M =m \right] -  \left(\frac{r}{m+1}\right)^\alpha\right| \leq \left(\frac{r}{m+1}\right)^\alpha \left[\frac{3}{m} + \frac{4}{r} + \frac{12}{mr} \right].
        $$
    \end{lemma}
    
    \begin{proof}[Proof of Lemma \ref{firt1}] Given that $M=m$, $T_{(k)} - T_{(l)}$ is distributed as $U_{(r)}$, the $r-$th order statistic, with $1\leq r=k-l\leq m-1$, of an independent sample of size $m$ from the uniform law on $[0,1]$. 
        Using inequality \eqref{b7} with $x=m+1$ and $s=\alpha$,  we can write
        \begin{multline}
            \left| \mathbb{E}\left[ U_{(r)} ^\alpha \mid M =m \right] - \frac{\Gamma (\alpha+r) }{\Gamma (r)} \frac{1}{(m+1)^\alpha} \right| \\  = \frac{\Gamma (\alpha+r) }{\Gamma (r)}\frac{1}{(m+1)^\alpha}   \left|   \frac{(m+1)^\alpha \Gamma (m+1) }{\Gamma (m+\alpha +1)} - 1 \right| \\
            \leq \frac{3}{m} \frac{\Gamma (\alpha+r) }{\Gamma (r)} \frac{1}{(m+1)^\alpha}.
        \end{multline}
        Next, using inequality \eqref{b6} twice,  with $x=r$ and $s=\alpha$, and triangle inequality
        \begin{multline*}
            \left| \mathbb{E}\left[ U_{(r)} ^\alpha \mid M =m \right] -  \left(\frac{r}{m+1}\right)^\alpha\right| \\
            \leq \left| \mathbb{E}\left[ U_{(r)} ^\alpha \mid M =m \right] -  \frac{\Gamma (\alpha+r) }{\Gamma (r)} \frac{1}{(m+1)^\alpha}\right|
            + \left(\frac{r}{m+1}\right)^\alpha  \left| \frac{\Gamma (\alpha+r) }{r^\alpha \Gamma (r)}  - 1\right| \\
            \leq \left(\frac{r}{m+1}\right)^\alpha \left[\frac{3}{m} \frac{\Gamma (\alpha+r) }{r^\alpha \Gamma (r)} + \frac{4}{r}  \right] \\
            \leq \left(\frac{r}{m+1}\right)^\alpha \left[\frac{3}{m} \left(1+ \frac{4}{r}\right) + \frac{4}{r}  \right].
        \end{multline*}
    \end{proof}

    \subsection{The general case}
    
    Given the realization of $M$, let $T_1,T_2,\ldots$ be an independent sample from $T$, a random variable independent of $M$, with an absolute continuous distribution on $[0,1]$. Let $f$ (resp. $F$) (resp. $Q$) denote the density (resp. distribution function)  (resp. quantile function) of $T$.  We assume that $F$ is strictly increasing on $[0,1]$ and thus $Q$ is the inverse function for $F$, and $Q$ is differentiable with $Q^\prime = 1/f$. 
    Then, given $M=m$, for any $1\leq l<k\leq m$, the joint distribution of the order statistics $(T_{(k)} ,T_{(l)})$ is the same as the joint distribution of  $(Q(U_{(k)}) , Q(U_{(l)}))$, where $U_{(1)},\ldots,U_{(m)}$ is the order statistics of an independent uniform sample on $[0,1]$.

    
    Assume $\inf_{t\in[0,1]}f(t) >0$ and $f$ is H\"older continuous around $t_0$, i.e. there exists $L_f >0$, $0<\beta_f  \leq 1$,  and a neighborhood of $t_0$ in $[0,1]$ such that for any $u,v$ in this neighborhood, $|f(u)-f(v)|\leq L_f  |u-v|^{\beta_f }$. 
    
    \begin{lemma}\label{lem:moment-spacing}
        Let $m$ be an integer value in the support of $M$.
        Let $t_0\in[0,1]$,  assume that $k$ and $l$ are satisfying the conditions \eqref{context}-\eqref{context2}, and let $r=k-l$. The for any $0 < \alpha \leq 3$, 
        \begin{multline*}
            \left| \mathbb{E}\left[ (T_{(k)} - T_{(l)} )^\alpha \mid M =m \right] -  \frac{\Gamma (\alpha+r) }{\Gamma (r)} \left(\frac{1}{f(t_0) (m+1) }\right)^\alpha \right| \\
            \leq  \frac{\Gamma (\alpha+r) }{\Gamma (r)} \left( \frac{1}{f(t_0)(m+1)}\right)^\alpha\left[ \frac{3}{m} + C \left(\frac{r}{m+1} \right)^{\alpha \beta_f /4}\right] , \end{multline*}
            and
            \begin{multline}
                \left| \mathbb{E}\left[ (T_{(k)} - T_{(l)} )^\alpha \mid M =m \right] -  \left(\frac{r}{f(t_0)(m+1)}\right)^\alpha\right| \\
                \leq \left(\frac{r}{f(t_0)(m+1)}\right)^\alpha \left[\frac{3}{m} + \frac{4}{r} + \frac{12}{mr}  + C \left(\frac{r}{m+1} \right)^{\alpha \beta_f /4} \right]\\
                \leq \mathfrak{c}_0 \left(\frac{r}{f(t_0)(m+1)}\right)^\alpha \left[\frac{1}{m} + \frac{1}{r} + \frac{1}{mr}  + \left(\frac{r}{m+1} \right)^{\alpha \beta_f /4} \right],\label{eq:c0} 
            \end{multline}
            with $C$ and $\mathfrak{c}_0$ are two constants depending only on $\alpha$ and $L_f ,\beta_f $ and 
            $f(\T)$. 
        \end{lemma}
        
        \begin{proof}[Proof of Lemma \ref{lem:moment-spacing}] In the following, we use several times the following property: for any $a,b,\alpha \geq 0$, 
            $$
            (a+b)^\alpha \leq \max(1,2^{\alpha - 1}) \left( a^\alpha + b^\alpha  \right).
            $$  
            Next, given $M=m$, 
            $$\mathbb{E}\left[ (T_{(k)} - T_{(l)} )^\alpha \mid M =m \right] =\mathbb{E}\left[ \{Q(U_{(k)})  - Q(U_{(l)}) \}^\alpha \mid M =m \right].$$ 
            By a first order Taylor expansion of $Q(U_{(k)})$  around the point $U_{(l)}$, we get
            \begin{equation}\label{eqa}
                Q(U_{(k)}) - Q(U_{(l)}) =\frac{1}{f(t_0)} \left[ U_{(k)} - U_{(l)}\right] \left[1+ r(m,k,l) \right],
            \end{equation}
            with 
            $$
            r(m,k,l)=  \int_0^1 \frac{f(t_0) - f(U_{(l)}+ t[U_{(k)} - U_{(l)}])}{f(U_{(l)}+ t[U_{(k)} - U_{(l)}])}dt.
            $$
            Note that due to the fact the $Q$ is increasing and almost surely$U_{(k)} > U_{(l)}$, the identity \eqref{eqa} implies that $1+r(m,k,l)>0$ almost surely.  
            Using the triangle inequality and the properties of $f$, 
            $$
            | r(m,k,l)| \leq \frac{L_f }
            {f(\T)/2}
            \left(| U_{(l)} - t_0 |^{\beta_f } + |  U_{(k)} - U_{(l)}|^{\beta_f }\right).
            $$
            Let 
            $$
            t_m = \frac{\lfloor t_0(m-1)\rfloor+1}{m+1}. 
            $$ 
            Note that $1/(m+1) \leq t_m \leq m/(m+1) $ and 
            $$
            t_m = \mathbb{E}[  U_{(t_m (m+1) )}] .
            $$
            Next, we can bound
            \begin{multline*}
                | U_{(l)} - t_0 | \leq | U_{(l)} - t_m |  + | t_m - t_0 | \leq | U_{(l)} - \mathbb{E}[ U_{(t_m (m+1) )}] | + \frac{2}{m+1} \\
                \leq  | U_{(l)} -  U_{(t_m (m+1) )} | + | U_{(t_m (m+1) )}- \mathbb{E}[  U_{(t_m (m+1) )}] | + \frac{2}{m+1} .
            \end{multline*}
            Thus, with the convention $U_{(0)} = 0$, 
            \begin{multline*}
                \mathbb{E}\left[  | U_{(l)} - t_0 | ^{\beta_f } \mid M =m \right] \leq  \mathbb{E}\left[  U_{(|l-t_m(m+1)|)}  ^{\beta_f } \mid M =m \right] \\ +  \mathbb{E}\left[  \left|U_{(t_m (m+1) )}- \mathbb{E}[  U_{(t_m (m+1) )}] \right| ^{\beta_f } \mid M =m \right]  +  \left( \frac{2}{m+1} \right)^{\beta_f }.
            \end{multline*}
            By the facts presented in the uniform case, when $l\neq t_m(m+1) $, 
            $$
            \mathbb{E}\left[  U_{(|l-t_m(m+1)|)}  ^{\beta_f } \mid M =m \right] = \frac{\Gamma (\beta_f + | l-t_m(m+1)|) }{\Gamma (|l-t_m(m+1)|)} \frac{\Gamma (m+1) }{\Gamma (m+ \beta_f  +1)},
            $$
            and using Wendel's double inequality \eqref{wen1}
            with $s=\beta_f $, and  \eqref{context}, the product of the ratios of the gamma functions is bounded from above by
            $$
            \left(\frac{|l-t_m(m+1)|}{m+1+\beta_f } \right)^{\beta_f } \left(1+\frac{\beta_f }{m+1} \right)\leq 
            9\left(\frac{r}{m+1}\right)^{\beta_f }.
            $$
            On the other hand, using Jensen's inequality and the variance of a beta distribution with parameters $t_m(m+1)$ and $(1-t_m)(m+1)$, 
            \begin{multline*}
                \mathbb{E}\left[  \left|U_{(t_m (m+1) )}- \mathbb{E}[  U_{(t_m (m+1) )}] \right| ^{\beta_f } \mid   M = m \right] \\ \leq \mathbb{E}^{\beta_f /2}\left[  \left|U_{(t_m (m+1) )}- \mathbb{E}[  U_{(t_m (m+1) )}] \right| ^{2} \mid  M  = m \right]  =\left(\frac{t_m(1-t_m)}{m+2} \right)^{\beta_f /2}.
            \end{multline*}
            Gathering facts and using Lemma \ref{firt1}, there exists a constant $c$ such that 
            \begin{equation*}
                \mathbb{E}\left[  | U_{(l)} - t_0 | ^{\beta_f } \mid M =m \right] \leq c \left(\frac{r}{m+1} \right)^{\beta_f /2}.
            \end{equation*}
            On the other hand, since $ U_{(k)} - U_{(l)}$ is independent of $U_{(l)}$, from above and Lemma \ref{firt1} we deduce that for any $0 <\alpha^\prime \leq \alpha \leq 3$, 
            \begin{equation}\label{rt1}
                \mathbb{E}\left[ \{U_{(k)} - U_{(l)}\}^\alpha | r(m,k,l)|^{\alpha^\prime} \mid M =m \right] \leq C \left(\frac{r}{m+1} \right)^{\alpha + \alpha^\prime \beta_f /2 } 
            \end{equation}
            for some constant $C$ depending on 
            $L_f ,\beta_f $ and $f(\T)$.
            
            Coming back to relationship (\ref{eqa}), taking power $\alpha$ on both sides of the identity, we can write
            $$
            \mathbb{E}\left[  \{Q(U_{(k)}) - Q(U_{(l)})\}^\alpha   \mid M =m \right]  =\frac{1}{f^\alpha (t_0)} \mathbb{E}\left[  U_{(r)}^\alpha \mid M =m\right]+ R(m,k,l)
            $$
            with 
            $$
            R(m,k,l)=\mathbb{E}\left[ \{U_{(k)} - U_{(l)}\}^\alpha \{\left[1+ r(m,k,l) \right]^\alpha -1 \}\mid M =m\right] .
            $$
            Since for any $a>-1$ and $0<\alpha \leq 3$,
            $$
            |(1+a)^\alpha - 1| =  |(1+a)^{\alpha/2} - 1| |(1+a)^{\alpha/2} + 1|\leq 2|a|^{\alpha/2} (|a|^{\alpha/2}+2),
            $$
            using the bound \eqref{rt1} with $\alpha^\prime =\alpha$ and $\alpha^\prime = \alpha/2$, 
            $$
            |R(m,k,l)| \leq c_R \left(\frac{r}{m+1} \right)^{\alpha (1+\beta_f /4)},
            $$
            for some constant $c_R$ depending on  
            $L_f ,\beta_f $ and $f(\T)$.
            It remains to apply Lemma \ref{firt1} to complete the proof. 
        \end{proof}

        \subsection{Wendel's type inequalities for gamma function ratios}\label{wendel_app}
        
        Since in our case, we only need to consider  $\alpha \in (0,3]$, we could use the sharp bounds for the ratio of two gamma functions, as deduced by \cite{MR29448}. For any $x>0$ and $s\geq 0$, let
        $$
        R(x,s) = \frac{\Gamma (x+s) }{\Gamma (x) }.
        $$
        \cite{MR29448} proved that when $0\leq s \leq 1$, 
        \begin{equation}\label{wen1}
        \left(\frac{1}{1+ s/x}\right)^{1-s} \leq \frac{R(x,s)}{x^s} \leq 1. 
        \end{equation}
        Since 
        $$
       1-\frac{s}{x} \leq  \left(\frac{1}{1+ s/x}\right)^{1-s} ,\qquad \forall x\geq 1, 0\leq s \leq 1,
        $$
        we can deduce that, when $0\leq s \leq 1$,
        \begin{equation}\label{b1}
            1-\frac{1}{x} \leq  1-\frac{s}{x} \leq \frac{R(x,s)}{x^s} \leq 1,\qquad \forall x\geq 1. 
        \end{equation}
        Next, using the recurrence formula for the gamma function, when $1\leq s \leq 2$ we can write
        $$
        \frac{R(x,s)}{x^s}  = \left(1+\frac{s-1}{x}\right) \frac{R(x,s-1)}{x^{s-1}} 
        $$
        and deduce 
          \begin{equation}
              1-\frac{1}{x} \leq  \left(1+\frac{s-1}{x}\right)  \left(1-\frac{s-1}{x}\right) \leq \frac{R(x,s)}{x^s} \leq   1+\frac{s-1}{x}  \leq   1+\frac{1}{x} ,\quad \forall x\geq 1.
        \end{equation}
        For our purpose, we could deduce the following bounds: for any $0\leq s \leq 2$, 
        \begin{equation}\label{b4}
            1-\frac{1}{x} \leq  \frac{R(x,s)}{x^s} \leq      1+\frac{1}{x} ,\qquad \forall x\geq 1,
        \end{equation}
        and
        \begin{equation}\label{b5}
            1-\frac{1}{x-1} \leq  \frac{x^s}{R(x,s)} \leq      1+\frac{1}{x-1} ,\qquad \forall x\geq 2.
        \end{equation}
        Finally, using again the recurrence formula for the gamma function, when $2 \leq s \leq 3$, we can write
        $$
        \frac{R(x,s)}{x^s}  = \left(1+\frac{s-1}{x}\right)  \left(1+\frac{s-2}{x}\right) \frac{R(x,s-2)}{x^{s-2}} 
        $$
        and deduce, for $2 \leq s \leq 3$, and $x\geq 2$,
        $$
        \frac{R(x,s)}{x^s} \leq  \left(1+\frac{s-1}{x}\right)  \left(1+\frac{s-2}{x}\right)=1+\frac{3}{x}+ \frac{2}{x^2}\leq 1+\frac{4}{x},\quad \forall x\geq 2,
        $$
        and
        $$
        \frac{x^s}{R(x,s)} \geq 1 - \frac{3x+2}{(x+2)(x+1)} \geq 1 - \frac{3}{x+2}\geq 1 - \frac{3}{x},\quad \forall x\geq 1.
        $$
        On the other hand, 
        $$
        \frac{R(x,s)}{x^s} \geq \left( 1+\frac{1}{x} \right)  \frac{R(x,s-2)}{x^{s-2}} \geq \left( 1+\frac{1}{x} \right)  
        \left( 1-\frac{s-2}{x} \right) \geq 1-\frac{1}{x}
        $$
        and
        $$
        \frac{x^s}{R(x,s)} \leq \frac{x}{x+1} \frac{x^{s-2}}{R(x,s-2)} \leq \frac{x}{x+1} \; \frac{x}{x-(s-2)}
        \leq \frac{x^ 2}{x^2 -1} \leq      1+\frac{1}{x-1}.
        $$
        Gathering facts,  for $0 \leq s \leq 3$
        \begin{equation}\label{b6}
            \left| \frac{R(x,s)}{x^s} -1 \right| \leq     \frac{4}{x}, \qquad \forall x\geq 2,
        \end{equation}
        and
        \begin{equation}\label{b7}
            \left|  \frac{x^s}{R(x,s)} -1 \right| \leq     \frac{3}{x-1} ,\qquad \forall x\geq 2.
        \end{equation}

%% file: supplement/additional_simulation_results.tex
\section{Additional simulation results}\label{sec_add_simu}
       
       \subsection{The settings} 
In our simulations, we use three types of stochastic processes to generate the trajectories of $X$ that we recall in the following. 
\begin{itemize}
    \item \textbf{Setting 1}: \emph{Fractional Brownian motion.} The curves are generated using a classical fractional  Brownian motion with constant Hurst parameter $H\in (0,1)$. In this case, the local regularity of the process is the same at every point. Figure \ref{fig:bm} illustrates one realization of this setting.
    \item \textbf{Setting 2}: \emph{Piecewise fractional Brownian motion.} The curves are generated as a concatenation of multiple fractional Brownian motions with different regularities, that is with different Hurst parameters for different time periods. In this case, the local regularity is no longer constant. Figure \ref{fig:pbm} illustrates one realization of this setting.
    \item \textbf{Setting 3}: \emph{Integrated fractional Brownian motion.} The curves $X_t$ are obtained as  integrals $\int_0^t W_H(s)ds$, $t\in[0,1]$,  of the paths of a fractional Brownian motion process $W_H$ with constant Hurst parameter $H$. Here, the local regularity of the process is the same at each point but will be greater than $1$, thus this setting corresponds to the case of smooth trajectories. Figure \ref{fig:ibm} illustrates one realization of this setting.
\end{itemize}

\begin{figure}
    \centering
    
    \begin{subfigure}{\textwidth}
        \centering
        \includegraphics[scale=0.225]{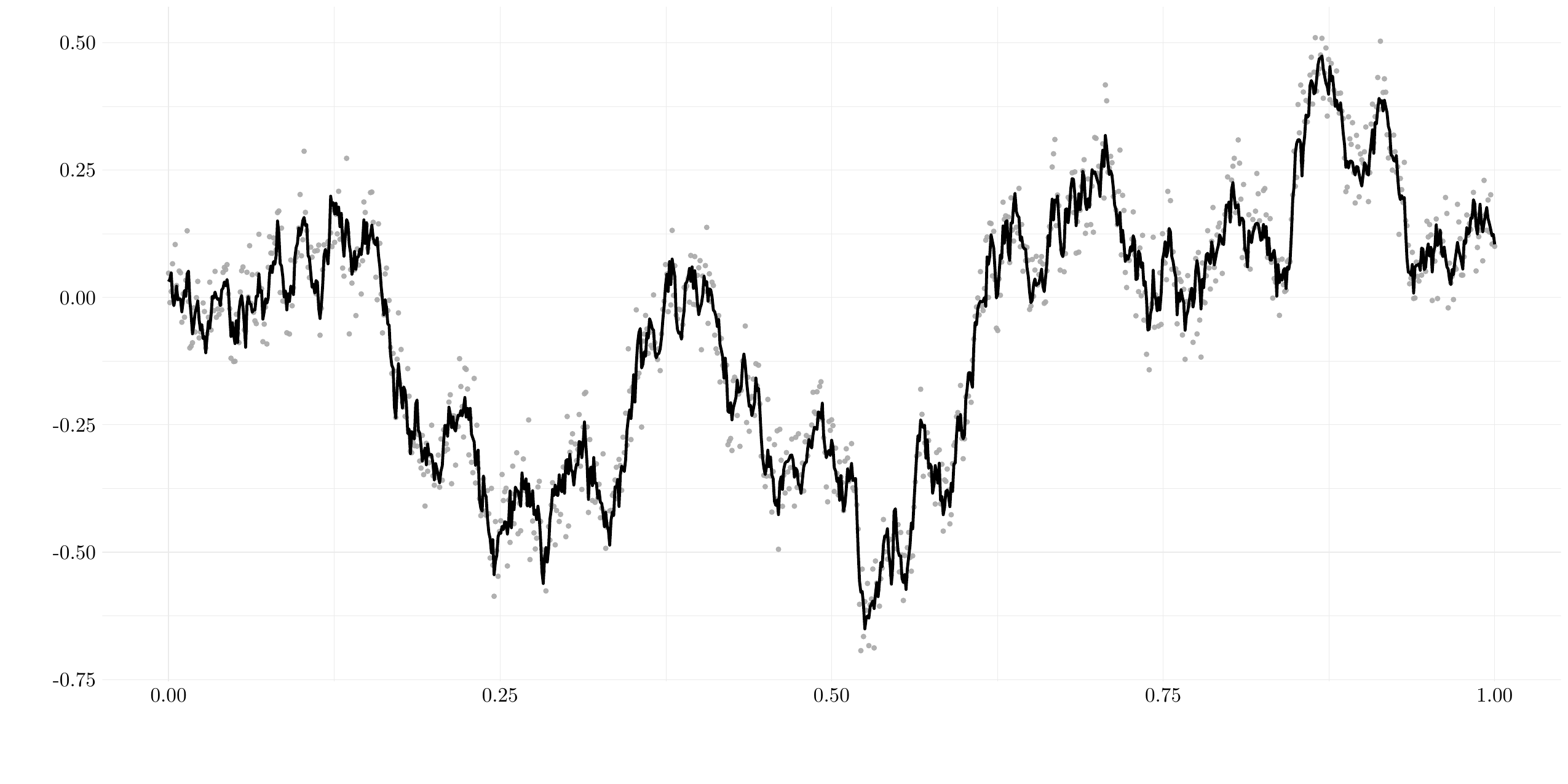}
        \caption{Brownian motion}
        \label{fig:bm}
    \end{subfigure}
    
    \vspace{0cm}
    
      \begin{subfigure}{\textwidth}
        \centering
        \includegraphics[scale=0.225]{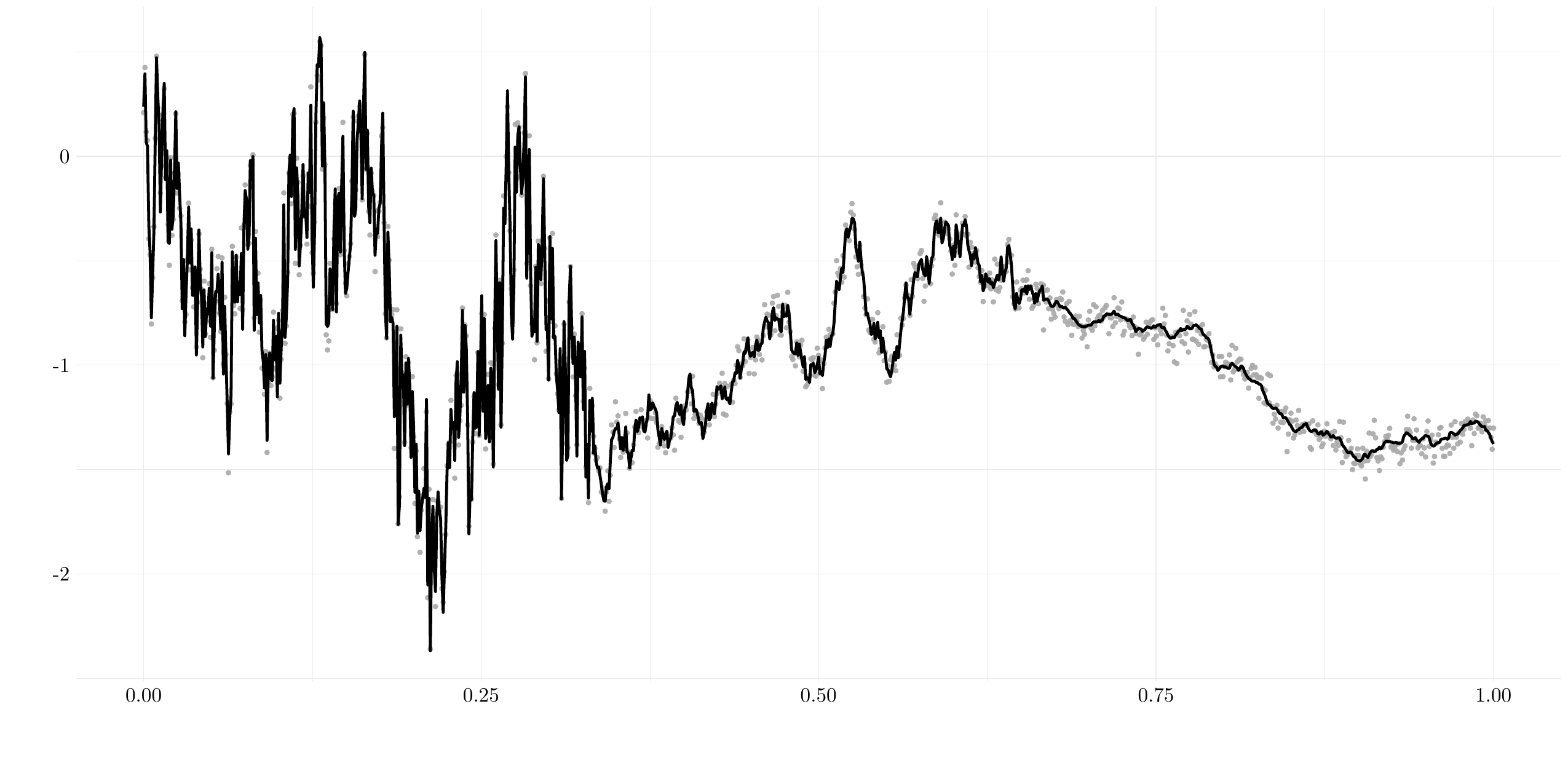}
        \caption{Piecewise Brownian motion}
        \label{fig:pbm}
    \end{subfigure}
    
    \vspace{0cm}
    
  \begin{subfigure}{\textwidth}
        \centering
        \includegraphics[scale=0.225]{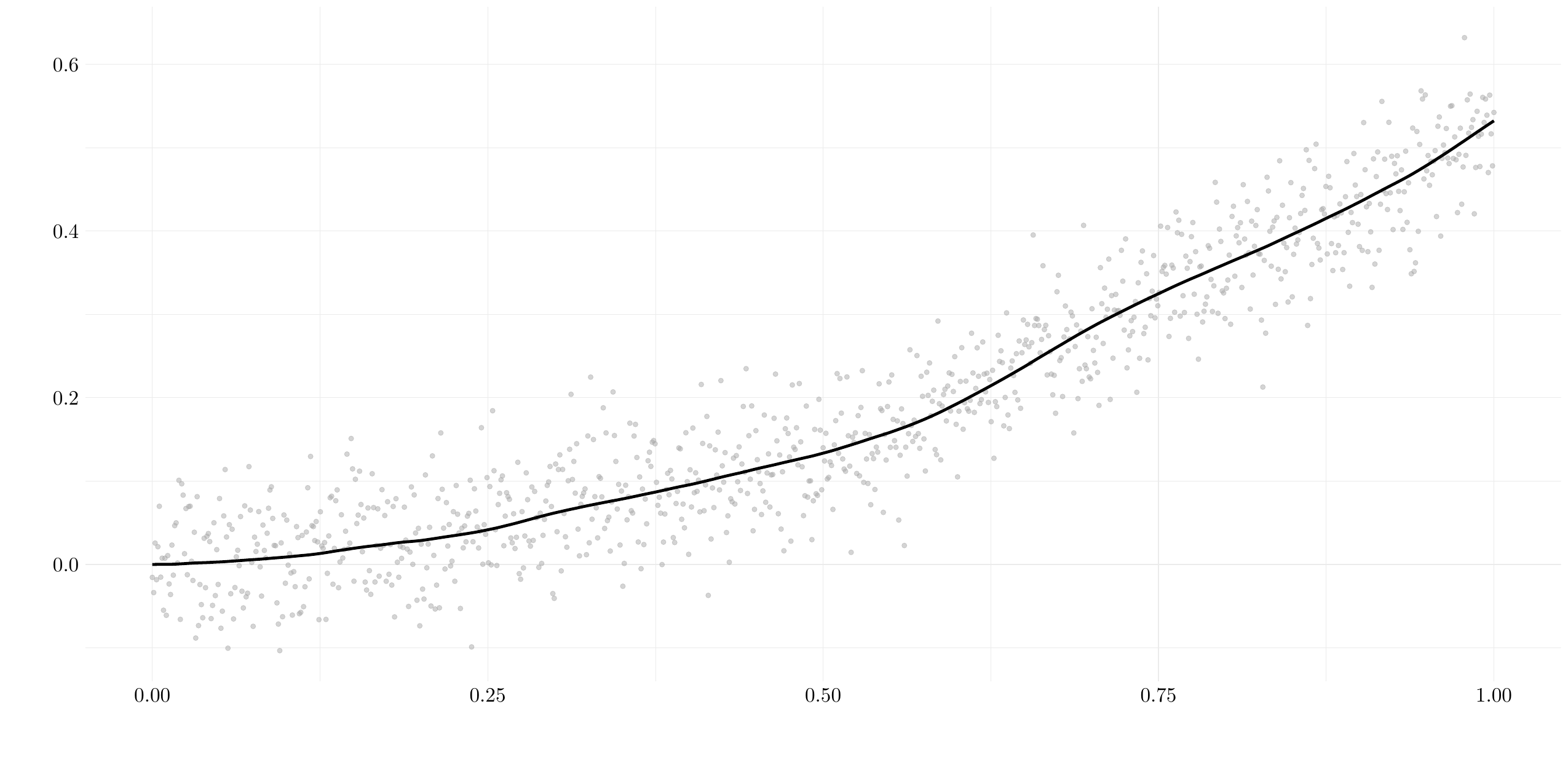}
        \caption{Integrated Brownian motion}
        \label{fig:ibm}
    \end{subfigure}
    
    \caption{Illustrations of simulated data generated according to the different settings. The curves correspond to the generated trajectories without noise that we aim to recover, and the grey points correspond to the noisy measurements.}
    \label{fig:examples}
\end{figure}

\subsection{On the computation time}

Figure \ref{fig:set1compu} presents the violin plots of the needed time to smooth $N_1 = 1000$ curves. The results are obtained with the parameters of the simulation $(1, 1000$, $1000, 300, \texttt{equi}, 0.5, 0.05)$. They correspond to the total CPU time (system time and user time) to estimate the bandwidth $h_n$ and then estimate the curves at their sampling points. We perform these computations on a personal computer equipped with a processor Intel Core i7-6600U, CPU: 2.60\si{\giga\hertz}, RAM: 24\si{\giga\octet} and rerun the estimation $10$ times. We observe that our smoothing device outperforms cross-validation and \emph{plug-in} in terms of computation time: about $1000$ times faster than the cross-validation. Let $\mathcal{H}_n$ be a set of bandwiths. For the cross-validation, we may explain these differences because of the computation of the estimator for each bandwidth in $\mathcal{H}_n$ and each curve $X^{(n)}$ of the sample ($N_1 \times \text{Card}(\mathcal{H}_n)$ calls to the estimation function) while our estimator requires only one estimation of the regularity of the functions and one evaluation of the estimator per curve ($N_1$ calls to the estimation function). 
In a similar way, figure \ref{fig:set3compu} presents the violin plots of the time necessary to smooth $N_1 = 1000$ curves with the parameters of the simulation $(3, 1000, 1000, 1000, \texttt{equi}, 1.7, 0.005)$. The same personal computer is used and the simulation is also run $10$ times. For   setting $3$, our procedure is slower than for   setting $1$, which can be easily explained by the computation of the derivatives of  each curve $X^{(n)}$. However, the computation time for the cross-validation is still not comparable with ours.

\begin{figure*}
    \centering

    \begin{subfigure}{.475\textwidth}
        \centering
        \includegraphics[scale=0.45]{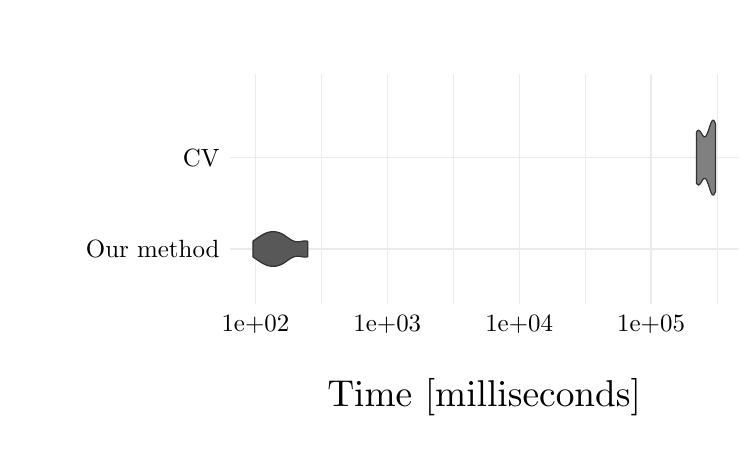}
        \caption{For setting 1}
        \label{fig:set1compu}
    \end{subfigure}
    \begin{subfigure}{.475\textwidth}
        \centering 
        \includegraphics[scale=0.45]{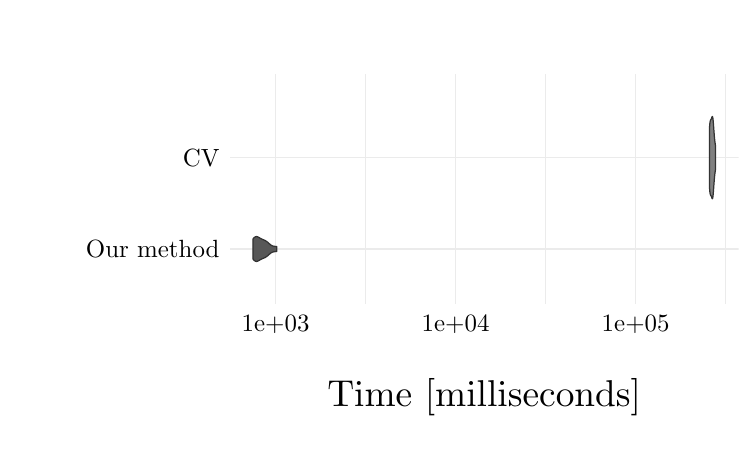}
        \caption{For setting 3}
        \label{fig:set3compu}
    \end{subfigure}
    
    \caption{Computational times (log scale)}
    \label{fig:compu_time}
\end{figure*}

\subsection{On the estimation of the local regularity}

Figure \ref{fig:set1_H} presents the results for the local regularity estimation for fBm with homoscedastic noise. The local esitmation of $\HT$ is performed at $\T = 1/2$ which correspond to the middle of the interval. The true value of $\HT$ is $0.5$. The results show an accurate estimator $\hHT$, except, maybe, for the simulation $(1, 250, 500, 1000, \texttt{equi}, 0.5, 0.05)$ where there is not enough curves compared to the number of sampling points.
\begin{figure}
	\centering
	\includegraphics[scale=0.45]{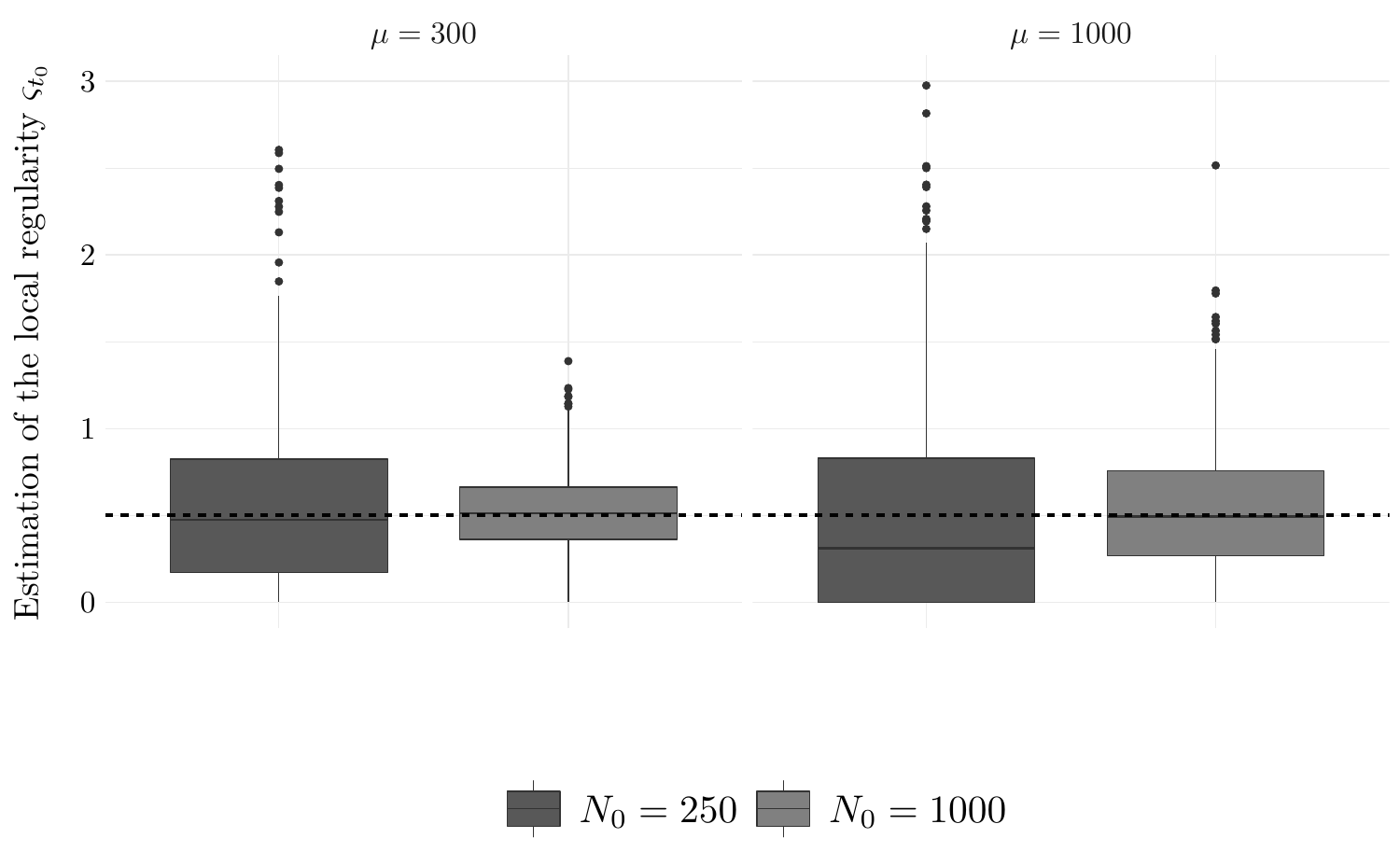}
	\caption{Estimation of the local regularity for fBm, with constant noise variance $\sigma^2 = 0.05$, at $t_0 = 1/2$. True value: $\varsigma_{t_0} = 0.5$.}
	\label{fig:set1_H}
\end{figure}

Figure \ref{fig:set2_H_hetero} presents the results for the local regularity estimation for piecewise fBm with heteroscedastic noise. The local estimations of $\HT$ are performed at $\T = 1/6, 1/2$ and $5/6$ which correspond to the middle of the interval for each regularity. The true values of $\HT$ are $0.4, 0.5$ and $0.7$, respectively. The true values of $\sigma^2$ are $0.04, 0.05$ and $0.07$, respectively. The results show an accurate estimator $\hHT$.
\begin{figure}
	\centering
	\includegraphics[scale=0.475]{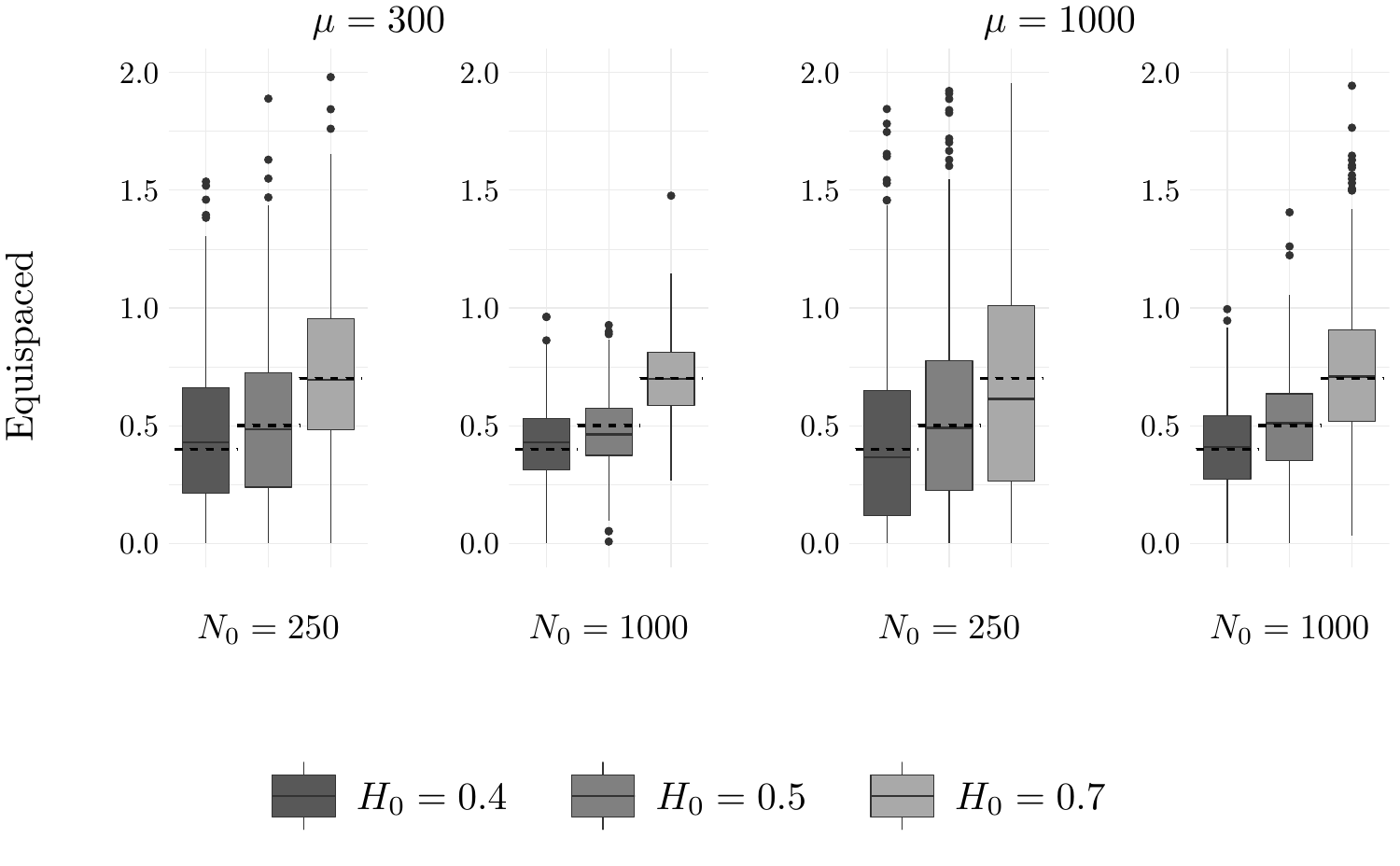}
	\caption{Estimation of the local regularity for piecewise fBm, with non-constant noise variance $\sigma^2 = 0.04, 0.05$ and $0.07$, at $\T = 1/6, 1/2$ and $5/6$, respectively. True values: $\varsigma_{\T} = \HT$ equal to $0.4, 0.5$ and $0.7$, respectively.}
	\label{fig:set2_H_hetero}
\end{figure}

\subsection{On the pointwise risk}

For technical convenience, in our theoretical study, we only considered the case where the regularity estimator $\ST$ is applied with an independent sample. If one wants to smooth the curves in the learning set, one can use a leave-one-out method. That is, for each curve, one can estimate the local regularity without that curve, and smooth the curve with the estimate obtained. Our method for calculating $\hHT$ is very fast, and such a leave-one-curve-out procedure is feasible. This idea was used to analyze the NGSIM data. However, one could also simply smooth the \emph{learning} set curves using the same local regularity estimates obtained from this dataset.
Figure \ref{fig:set2_same} presents the estimation of the risks $\mathcal R (\widehat X;1/6)$, $\mathcal R (\widehat X;0.5)$ and $\mathcal R (\widehat X;5/6)$ for piecewise fBm, with constant noise variance $\sigma^2=0.05$, when the training and the test set are the same. The simulation results indicate that our theoretical results could be extended to the case where the \emph{online} set is taken equal to the \emph{learning} set, though the concentration deteriorates. The theoretical investigation of this issue is left for future work. 
\begin{figure}
	\centering
	\includegraphics[scale=0.45]{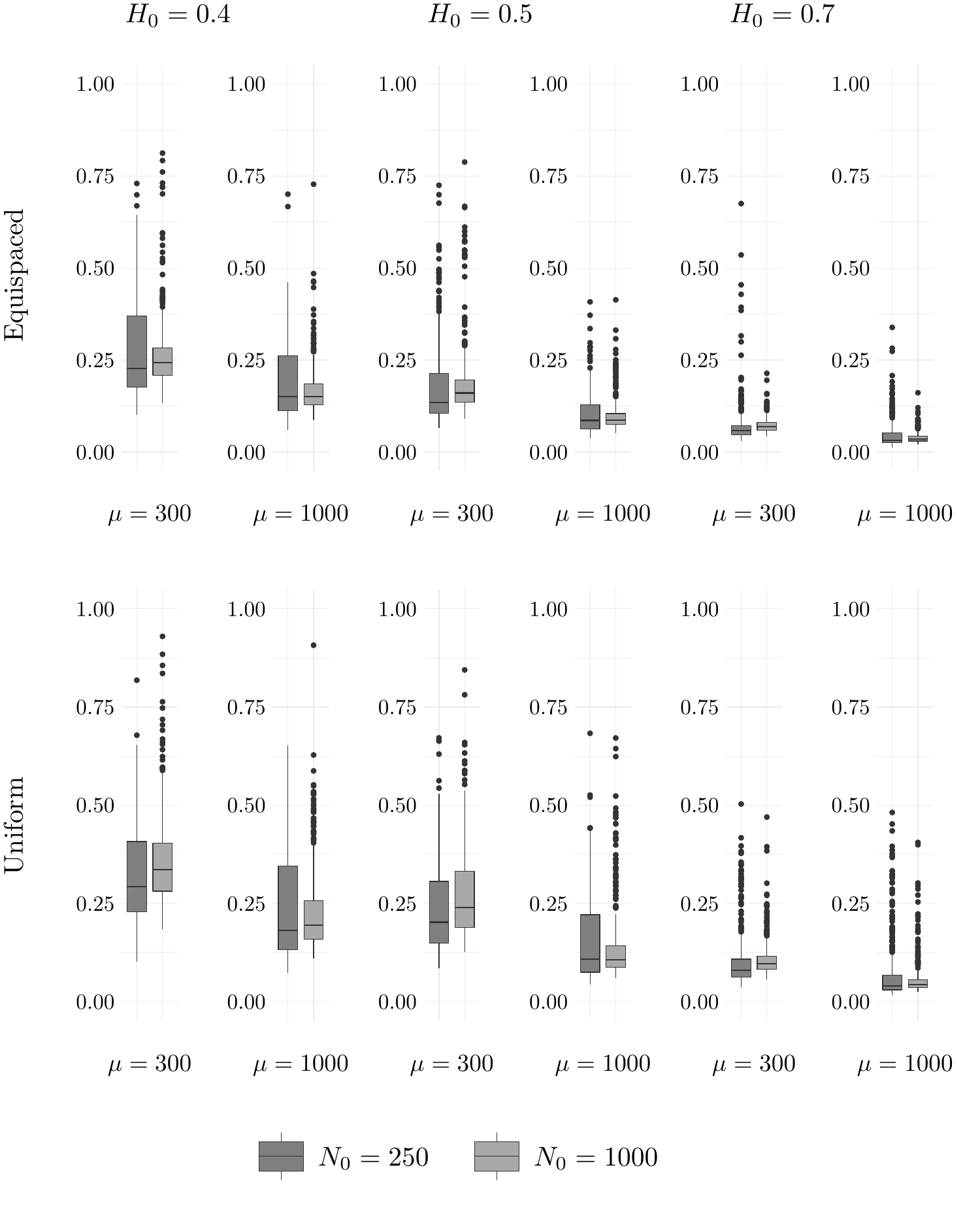}
	\caption{Estimation of the risks $\mathcal R (\widehat X;1/6)$, $\mathcal R (\widehat X;0.5)$ and $\mathcal R (\widehat X;5/6)$
     for piecewise fBm, with constant noise variance $\sigma^2=0.05$, when the training and the test set are the same.}
	\label{fig:set2_same}
\end{figure}
Figure \ref{fig:set1_risk} presents the estimation of the risks $\mathcal R (\widehat X;0.5)$ for fBm, with constant noise variance $\sigma^2=0.05$.
\begin{figure}
	\centering
	\includegraphics[scale=0.45]{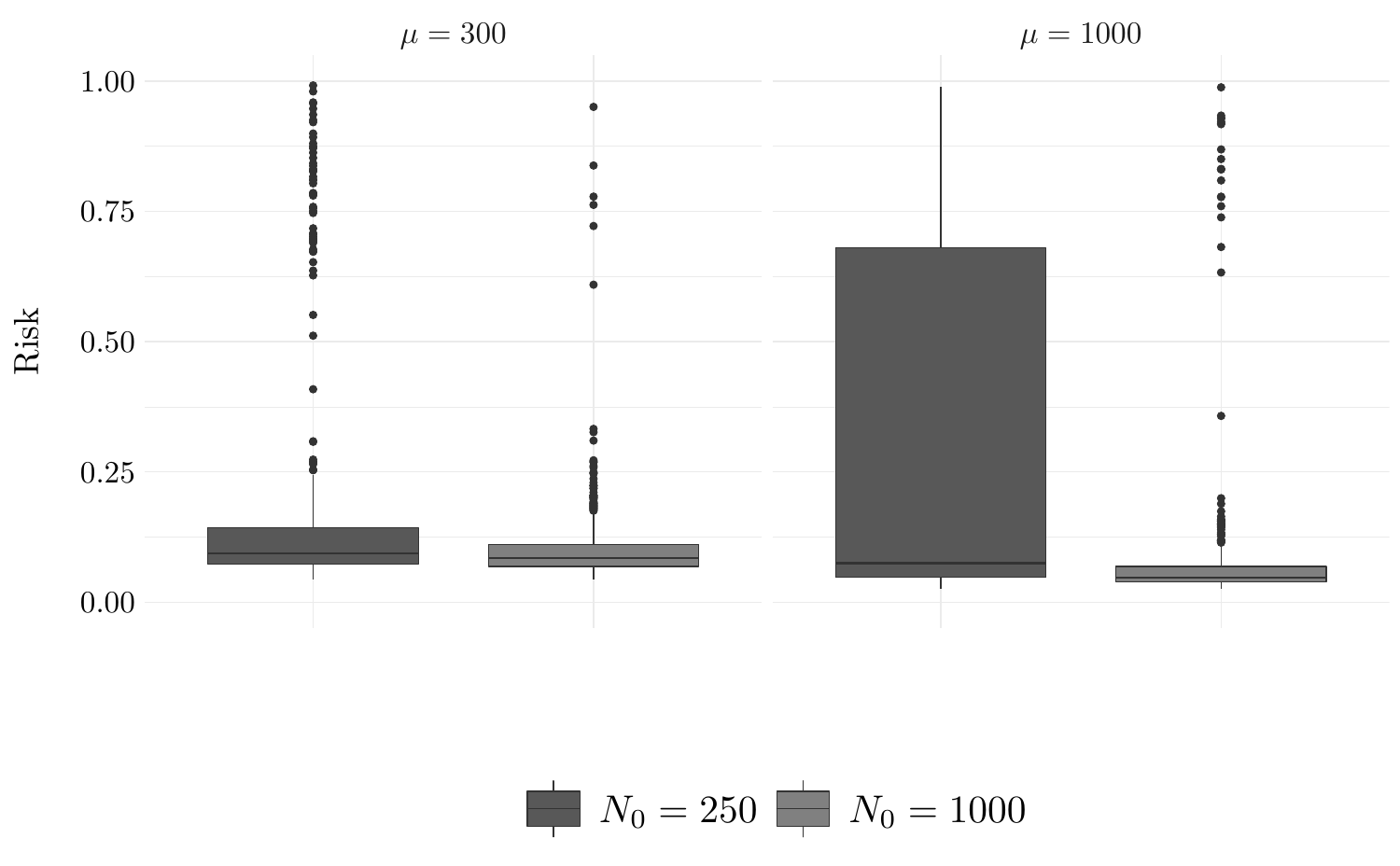}
	\caption{Estimation of the risk $\mathcal R (\widehat X;0.5)$ for smoothing the noisy trajectories of a fBm, with constant noise variance $\sigma^2 = 0.05$.}
	\label{fig:set1_risk}
\end{figure}
Figure \ref{fig:set2_risk_hetero} presents the estimation of the risks $\mathcal R (\widehat X;1/6)$, $\mathcal R (\widehat X;0.5)$ and $\mathcal R (\widehat X;5/6)$ for piecewise fBm, with heteroscedastic noise. The conclusion are the same than the homoscedastic case.
\begin{figure}
	\centering
	\includegraphics[scale=0.45]{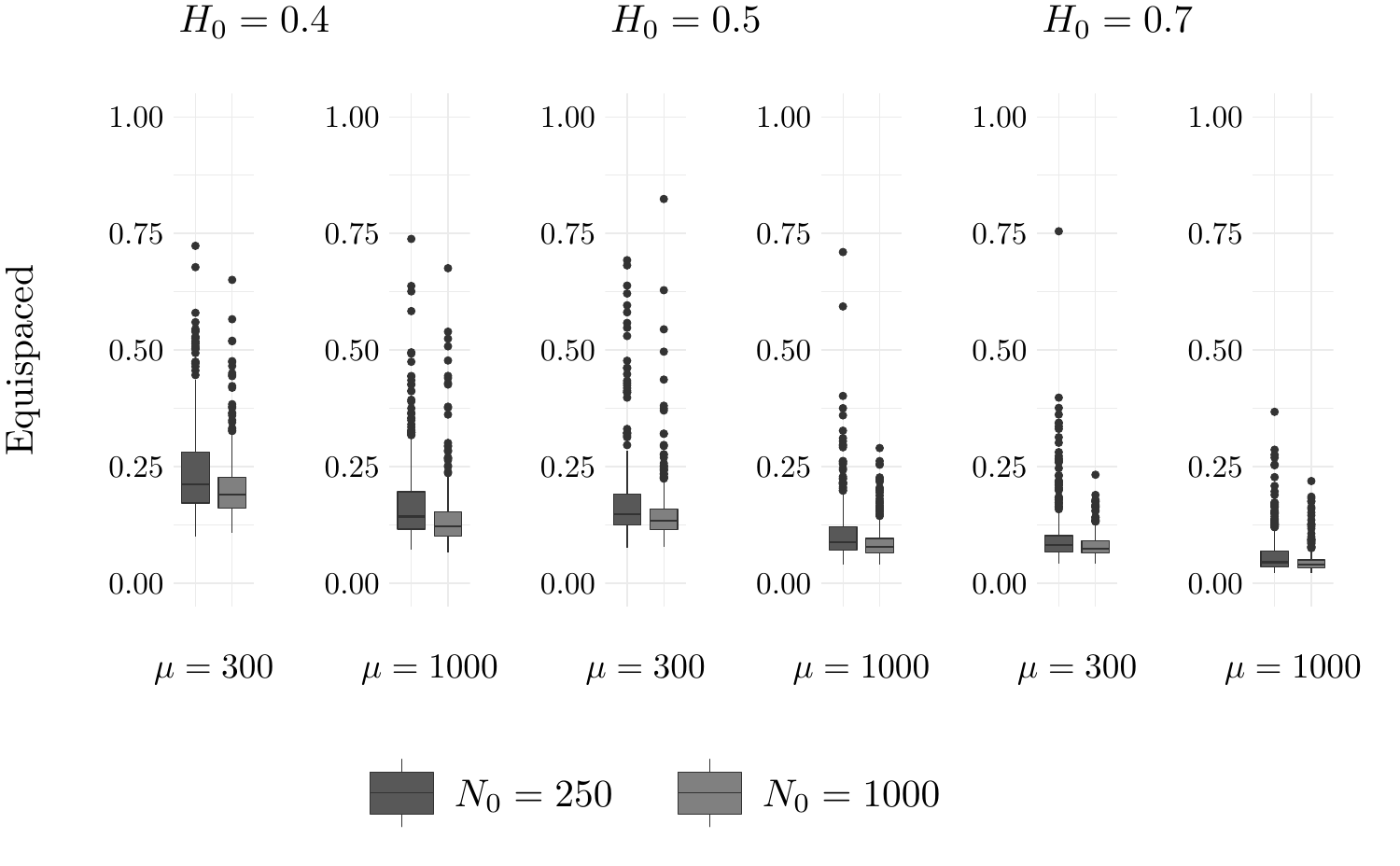}
	\caption{Estimation of the risks $\mathcal R (\widehat X;1/6)$, $\mathcal R (\widehat X;0.5)$ and $\mathcal R (\widehat X;5/6)$
     for piecewise fBm, with non-constant noise variance $\sigma^2 = 0.04, 0.05$ and $0.07$.}
	\label{fig:set2_risk_hetero}
\end{figure}

\subsection{Details on the constant of the bandwidth $h_{opt}$}\label{det_bws}

When the regression function admits a  derivative of order $\KT$ which is Hölder continuous in a neighborhood of $\T$, with exact exponent $\HT$ and  local Hölder constant $L_\T$, 
the optimal bandwidth for local polynomial smoothing proposed by \cite{tsybakov2009} is
\begin{equation}\label{eq:b_SM}
h_{opt} = \left(\frac{C}{M}\right)^{1/(2\ST + 1)}
\quad \text{with}
\quad C = C_\T = \frac{q_2}{2\ST q_1^2}
\end{equation} 
with 
$$
q_1 = C_* L_\T / \lfloor \ST \rfloor !  \quad \text{ and }\quad 
q_2 = \sigma^2_\T C_*^2.
$$ 
Here, $C_*$ is the constant defined on page 39 of \cite{tsybakov2009}. 
Let us recall the notation used by  \cite{tsybakov2009} for the local polynomial estimator of a regression function $r(\cdot)$, at the point $t$, using a sample $(Y_1,T_1),\ldots, (Y_M,T_M)$: 
$$
\hat r(\T) = \sum_{m=1}^M Y_m W_{Mm}(\T).
$$
In the case of the Nadaraya-Watson (NW) estimator, 
$$
W_{Mm}(\T) = \frac{1}{Mh}  \frac{K((T_m - \T)/h)}{\hat f (t)},
$$
where
$$ 
\hat f (\T) =\frac{1}{Mh}   \sum_{j=1}^M K((T_j - \T)/h) \approx f(\T).
$$
A closer look at the proof of Proposition 1.13  of  \cite{tsybakov2009} reveals that the absolute value of the bias is bounded by 
$$
h^{\ST}\frac{L_\T }{\lfloor \ST \rfloor ! } \frac{1}{\hat f (\T)} \sum_{m=1}^M |(T_m - \T)/h|^{\ST} K((T_m - \T)/h)\approx  \frac{h^{\ST}L_\T }{\lfloor \ST \rfloor ! } \int   K(v)   \lvert v \rvert^{\ST}dv.
$$
Meanwhile, the conditional variance of the NW estimator given the $T_m$ can be bounded  by
$$
\sigma^2_\T\frac{1}{Mh \hat f^2 (\T)}  \frac{1}{Mh} \sum_{m=1}^M K^2((T_m - \T)/h) \approx \sigma^2_\T\frac{1}{Mh f (\T)}\int   K^2(v)   dv.
$$
Given that $f (\T)$  can be estimated using the data points from all the curves,   the density of the $\Tnm$ can be estimated with high accuracy. We therefore use the true value $f(\T)$ in our simulations, which in the case of a uniform design is equal to 1.

\section{Traffic flow: Montanino and Punzo \cite{montanino_making_2013} methodology}\label{add_real_data_det}

Montanino and Punzo \cite{montanino_making_2013} presents a four steps methodology to make the NGSIM data usable. For a complete description of the steps, we let the reader refer to their article \cite{montanino_making_2013}. We briefly summarize their method here. The four steps below are applied for each trajectory separately. 

\begin{step} \emph{Removing the outliers}

They remove the measurements that lead to unreliable values of the acceleration by cutting all the records above a deterministic threshold of $30$ \si[per-mode=symbol]{\meter\per\second\squared}. The missing points are interpolated using a natural cubic spline with $10$ reference points before and after the outliers.
\end{step}

\begin{step}  \emph{Cutting off the high- and medium-frequency responses in the speed profile}

They remove the noise from the signal by linear smoothing of the signal with low-pass filter. The considered one is a first-order Butterworth filter \cite{butterworth_theory_1930} with cutoff frequency of $1.25$ \si{\hertz}.
\end{step}

\begin{step}  \emph{Removing the residual unphysical acceleration values, keeping the consistency requirements}

They remove residual peaks that exceed defined thresholds (varying with speed levels). For that, they move the position of the vehicle when the peak in acceleration appears in order to fulfill the thresholds. In order to prevent inconsistency, a $5$th-degree polynomial interpolation with constraint on the space traveled plus minor conditions was applied on a $1$\si{\second} window around the peak points.
\end{step}

\begin{step} Cutting off the high- and medium-frequency reponses generated from step 3

This step is the same as the step 2 but using the results of the step 3.

\end{step}

The methodology of  \cite{montanino_making_2013} seems  very specific to the NGSIM dataset, or at least some trajectory dataset, and by extension can not be easily applied to others. For using the algorithm on other trajectory datasets, their method requires some fine-tuning of the parameters. 

As explained in the main text, the 1714 observation units from the  I-80 dataset, available in the NGSIM study, have been recorded at moments of the day when traffic is evolving, it goes from fluid to dense traffic. Therefore, we consider that there are three groups in the data: a first group corresponding to a  fluid (high-speed) traffic, a second one for in-between fluid and dense  traffic, and a third groups corresponding to the dense (low-speed) traffic. Our local regularity approach, and the kernel smoothing induced, are applied for each group separately. The three group clustering was performed using a Gaussian mixture model estimated by an EM algorithm initialized by hierarchical model-based agglomerative clustering as proposed by Fraley and Raftery \cite{fraley_model_2002} and implemented in the \textbf{\textsf{R}} package \texttt{mclust} \cite{scrucca_mclust_2016}. The optimal model is then selected according to BIC. The three resulting classes have 239, 869 and 606 velocity trajectories, respectively.   Plots of randomly selected subsamples of trajectories from each groups are provided in Figure \ref{fig:subgroup_samples}.

\begin{figure}
    \centering
    \begin{subfigure}{\textwidth}
        \centering
        \includegraphics[scale=0.3]{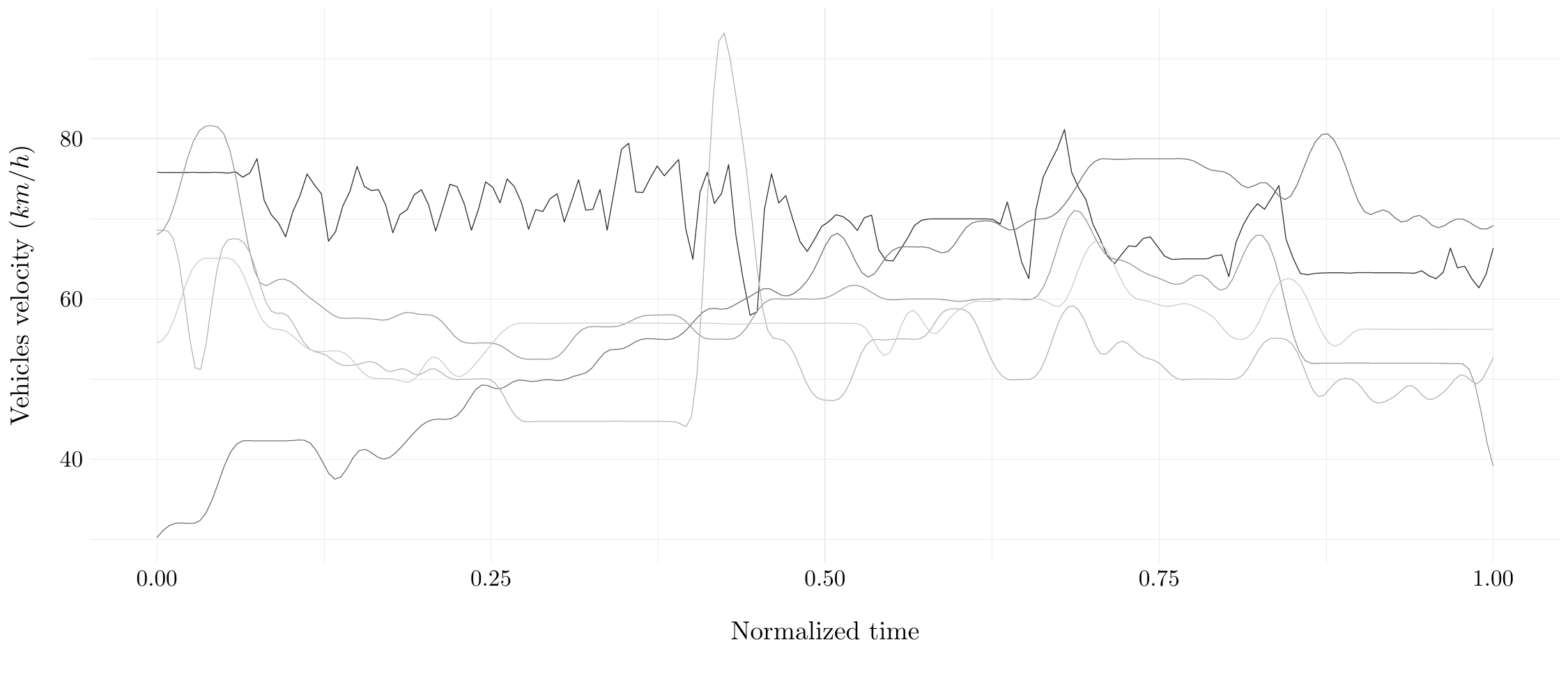}
        \caption{Fluid/high-speed traffic}
    \end{subfigure}
      \begin{subfigure}{\textwidth}
        \centering
        \includegraphics[scale=0.3]{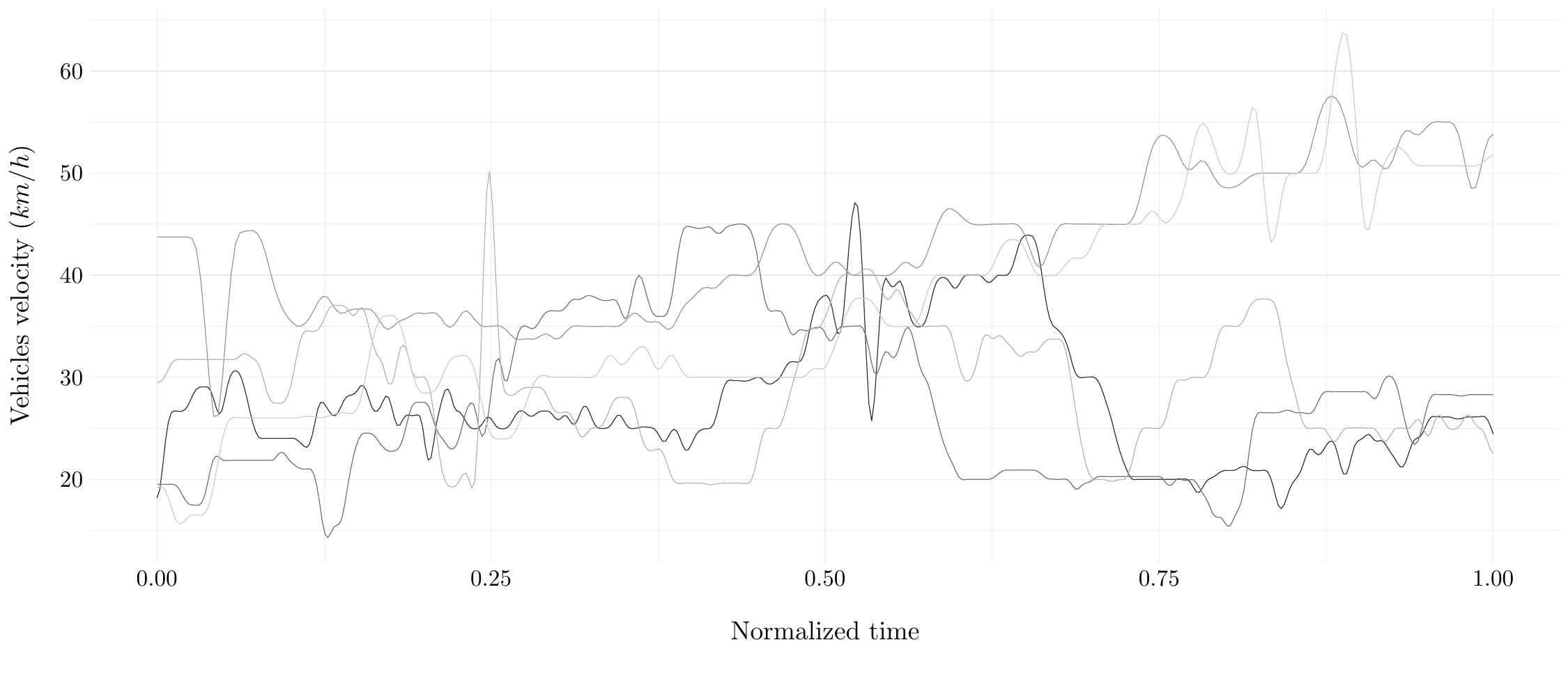}
        \caption{In-between traffic}
    \end{subfigure}    
  \begin{subfigure}{\textwidth}
        \centering
        \includegraphics[scale=0.3]{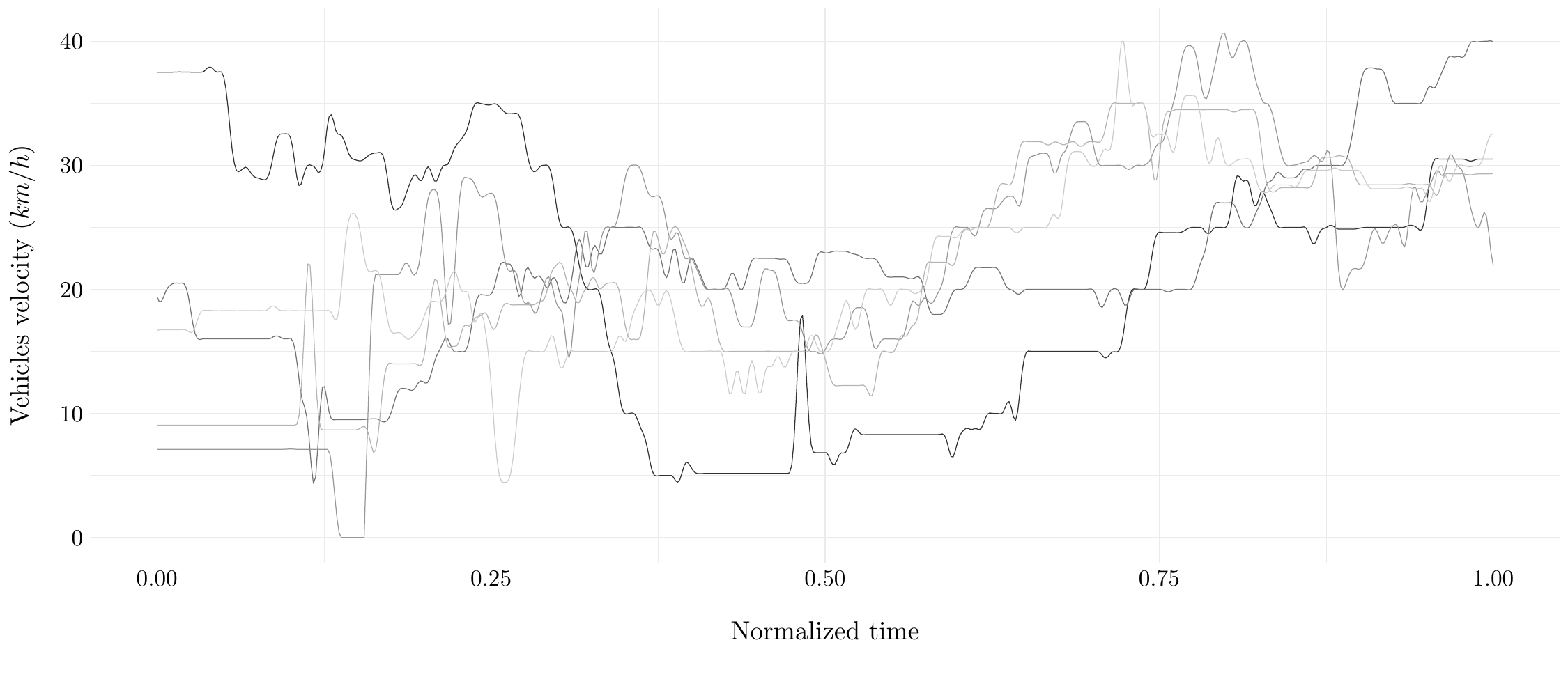}
        \caption{Dense/low-speed traffic}
    \end{subfigure}
    
    \caption{I-80 dataset illustration of the clusters: a sample of five velocity curves from each of the three groups of curves}
    \label{fig:subgroup_samples}
\end{figure}

\section{Complements on the real-data applications} \label{supp:real_data}

In this section, we point out the fact that our situation is not specific only to the traffic flow data, but can be applied to other real datasets.

\subsection{Canadian weather}

The Canadian Weather dataset \cite{ramsay_applied_2002,ramsay_functional_2005} records the daily temperature and precipitations in Canada averaged over the period from 1960 to 1994. Here, we are interested in the average daily temperature for each day of the year. It contains the measurements of $35$ canadian stations. Here, we have $\N{0} = 35$ and $\mu = 365$. A sample of five temperature curves has been plotted in the Figure \ref{fig:temp_sample}.  Figure \ref{fig:temp_H} presents the estimation of $\HT$ for different $\T$. We see that the estimation varies around $1$ with $\widehat{K}_0 = 25$.

\begin{figure}
    \centering
    \begin{subfigure}{.45\textwidth}
        \centering
        \includegraphics[scale=0.2]{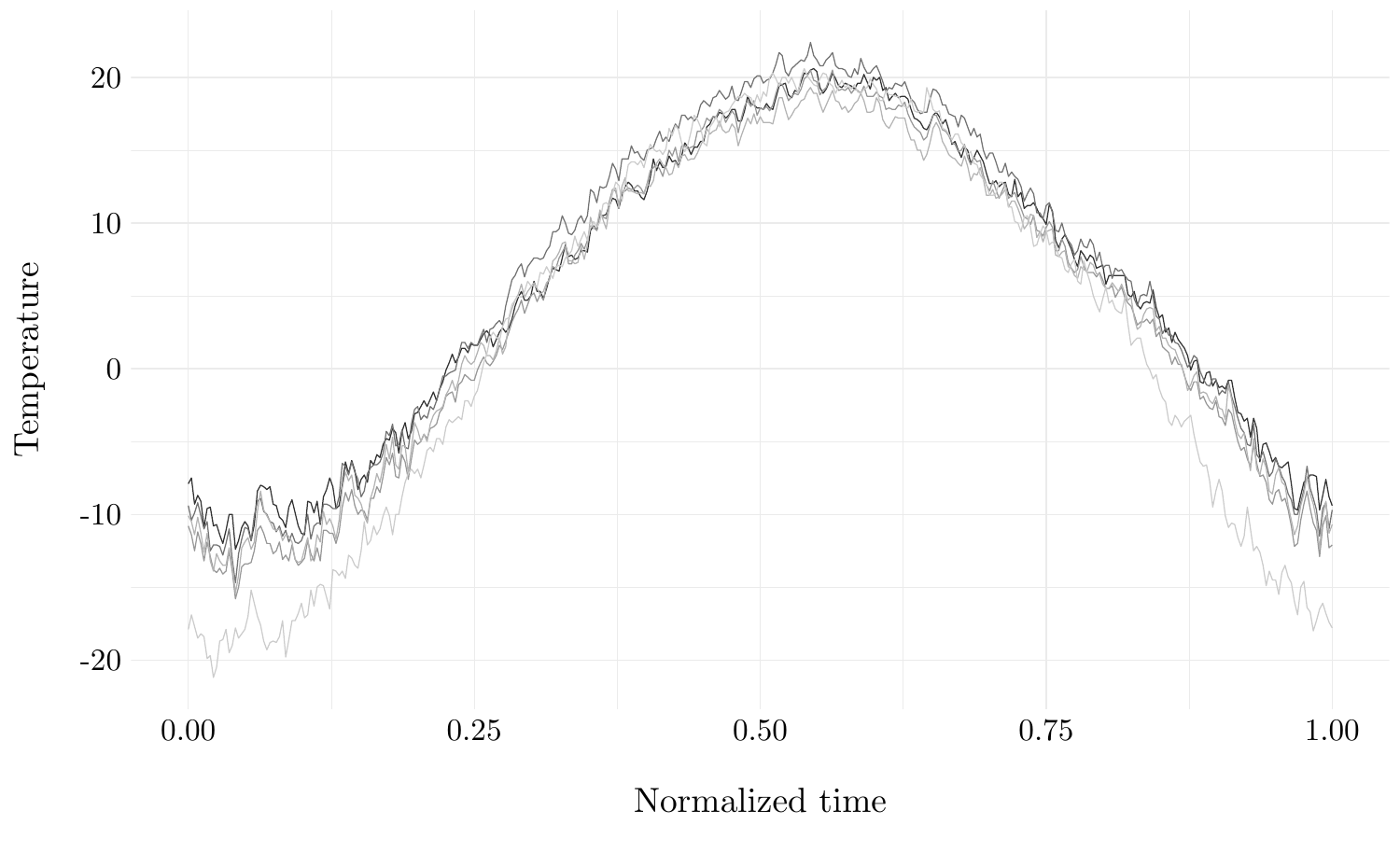}
        \caption{A sample of five temperature curves.}
        \label{fig:temp_sample}
    \end{subfigure}
    \begin{subfigure}{.45\textwidth}
        \centering
        \includegraphics[scale=0.2]{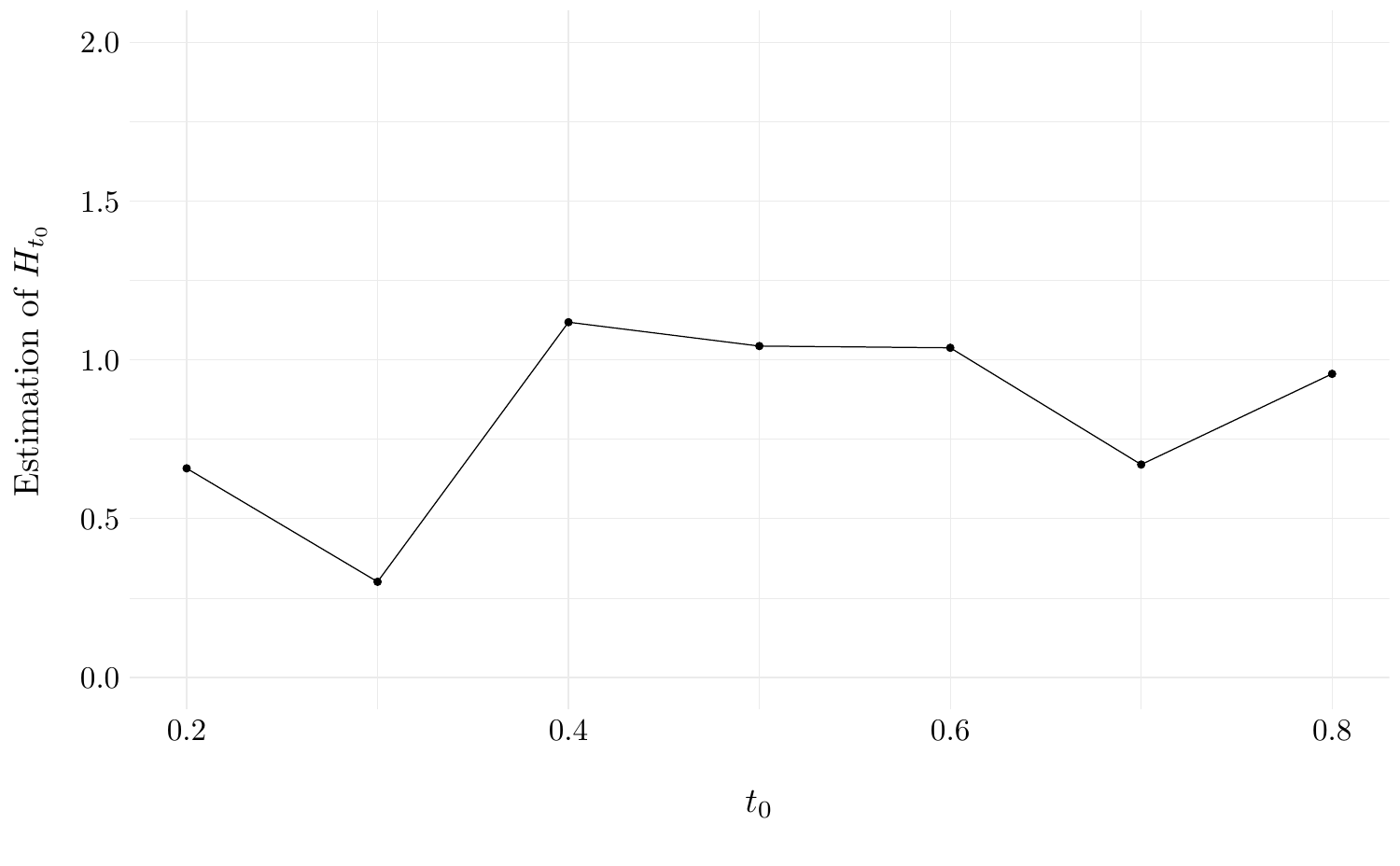}
        \caption{Estimation of $\HT$}
        \label{fig:temp_H}
    \end{subfigure}
    
    \caption{Canadian weather dataset illustration.}
    \label{fig:temp}
\end{figure}

\subsection{Household Active Power Consumption}

The Household Active Power Consumption dataset is part of the Monash University, UEA, UCR time series regression archive \citep{tan_monash_2020} and was sourced from the UCI repository\footnote{https://archive.ics.uci.edu/ml/datasets/Individual+household+electric+power+consumption}. The data measures diverse energy related features of a house located in Sceaux, near Paris every minute between December 2006 and November 2010. In total, its represents around $2$ million data points. These data are used to predict the daily power consumption of a house. Here, we are only interested in the daily voltage. The dataset contains $\N{0} = 746$ time series of $\mu = 1440$ measurements.
Figure \ref{fig:power_sample} presents a sample of five curves from this dataset. The estimation of the local regularity $\HT$, plotted in Figure \ref{fig:power_H}, is around $0.5$ with $\widehat{K}_0 = 73$.

\begin{figure}
    \centering
    \begin{subfigure}{.45\textwidth}
        \centering
        \includegraphics[scale=0.2]{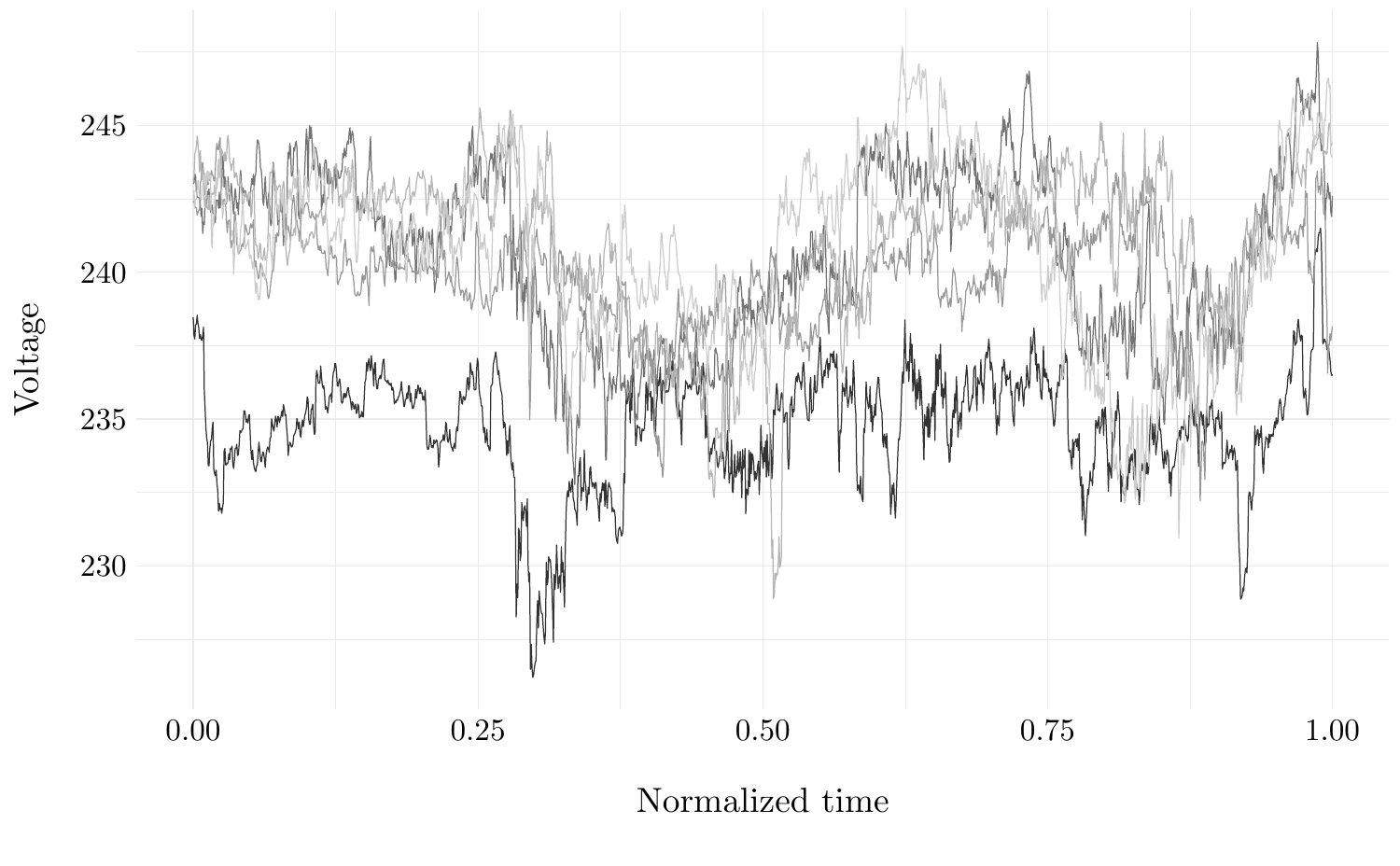}
        \caption{A sample of five power curves.}
        \label{fig:power_sample}
    \end{subfigure}
    \begin{subfigure}{.45\textwidth}
        \centering
        \includegraphics[scale=0.2]{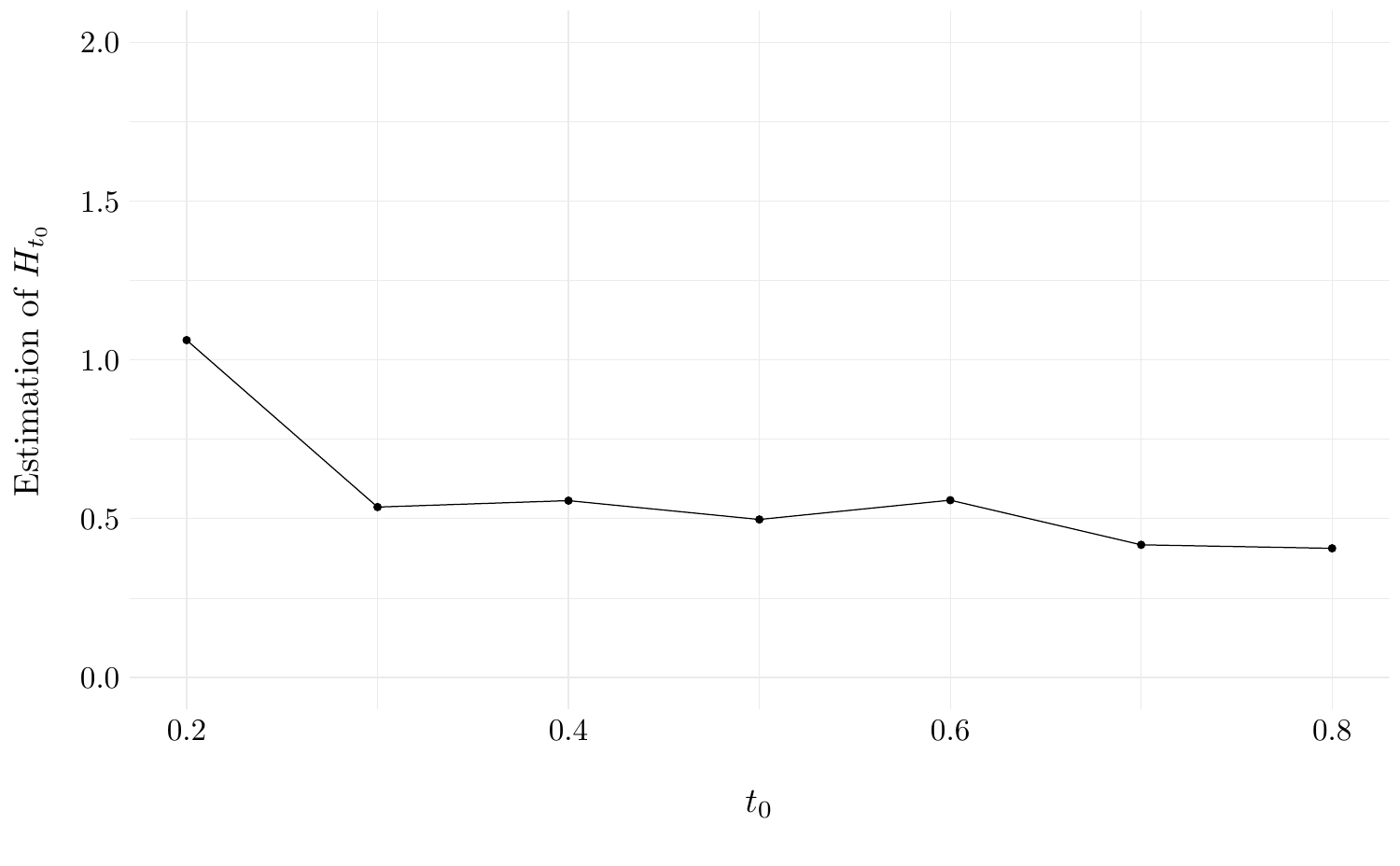}
        \caption{Estimation of $\HT$}
        \label{fig:power_H}
    \end{subfigure}
    
    \caption{Household active power consumption dataset illustration.}
    \label{fig:power}
\end{figure}

\subsection{PPG-Dalia}

The PPG-Dalia dataset is also part of the Monash University, UEA, UCR time series regression archive \citep{tan_monash_2020} and was also sourced from the UCI repository\footnote{https://archive.ics.uci.edu/ml/datasets/PPG-DaLiA}. PPG sensors are widely used in smart wearable devices to measure heart rate \cite{reiss_deep_2019}. They contain a single channel PPG and 3D accelerometer motion data recorded from $15$ subjects performing various real-life activities. Measurements from each subject are segmented into $8$ second windows with $6$ second overlaps, resulting in $\N{0} = 65000$ time series of $\mu = 512$ features. Here, we are interested in the PPG channel. A sample of five curves is plotted in Figure \ref{fig:ppg_sample}. The estimation of the local regularity $\HT$ is also around $0.5$ (see Figure \ref{fig:ppg_H}) with $\widehat{K}_0 = 25$.

\begin{figure}
    \centering
    \begin{subfigure}{.45\textwidth}
        \centering
        \includegraphics[scale=0.2]{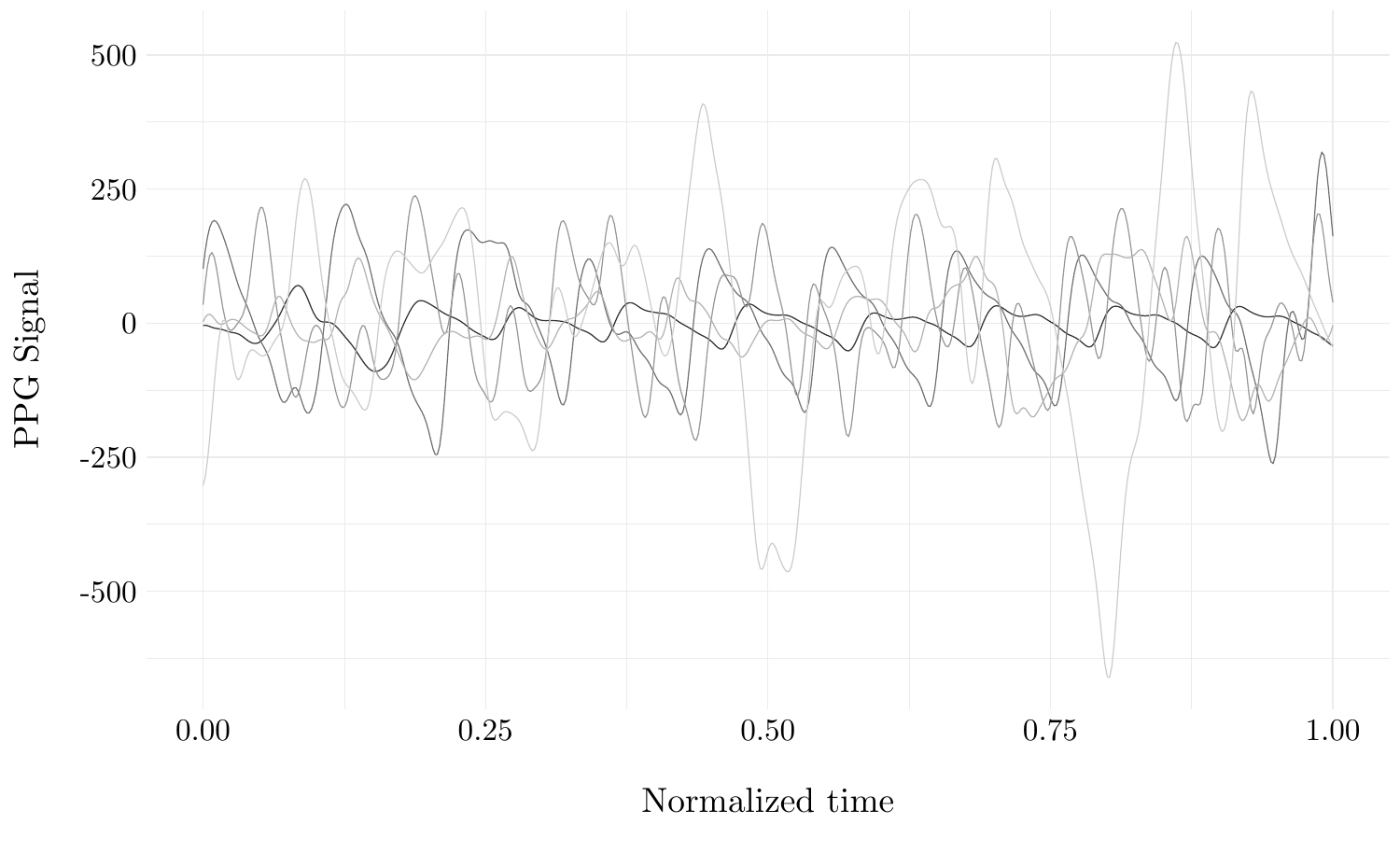}
        \caption{A sample of five PPG curves.}
        \label{fig:ppg_sample}
    \end{subfigure}
    \begin{subfigure}{.45\textwidth}
        \centering
        \includegraphics[scale=0.2]{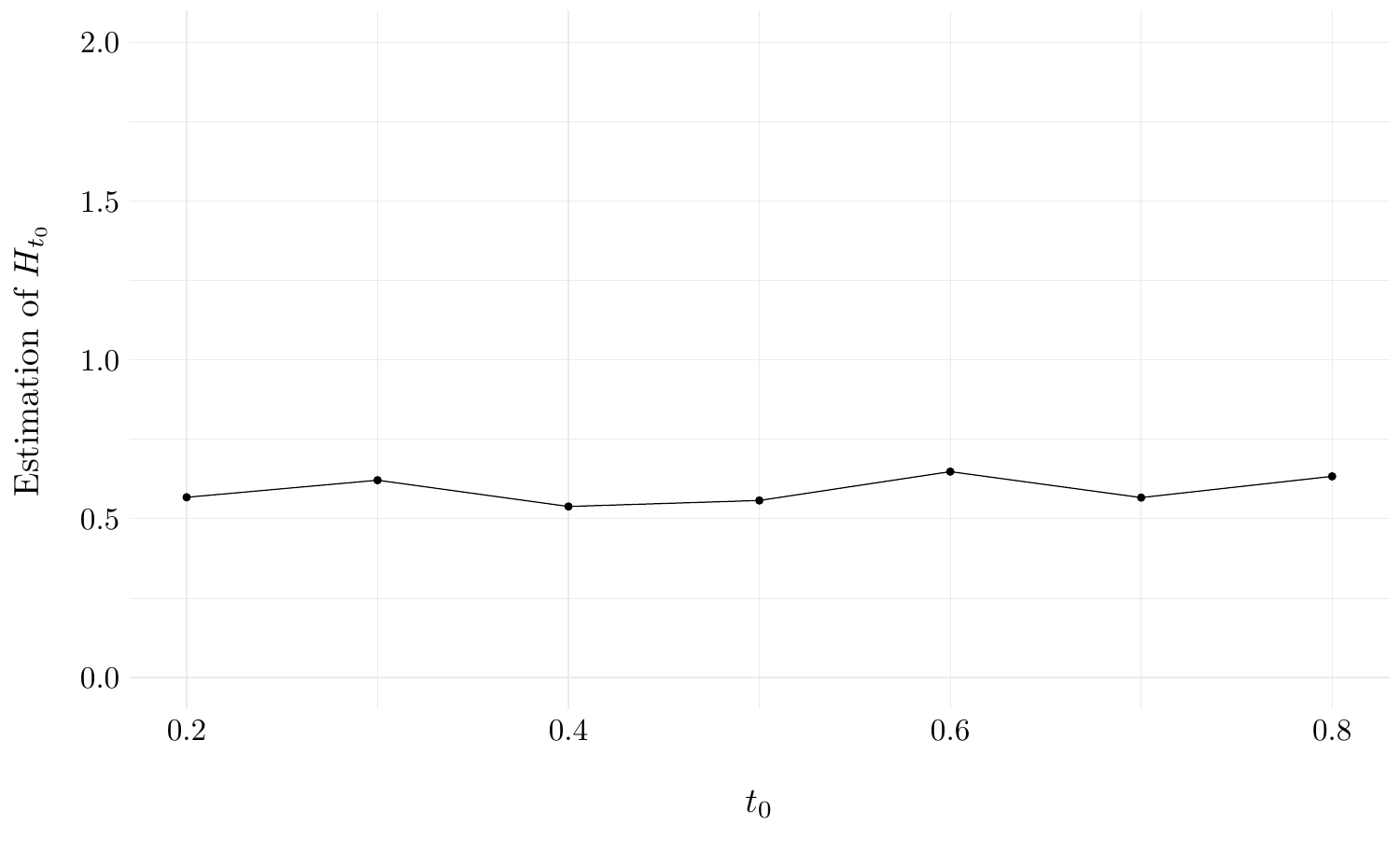}
        \caption{Estimation of $\HT$}
        \label{fig:ppg_H}
    \end{subfigure}
    
    \caption{PPG-Dalia dataset illustration.}
    \label{fig:ppg}
\end{figure}